\documentclass[12pt]{amsart}
\usepackage[utf8]{inputenc}
\usepackage[english]{babel}
\usepackage{amsmath,amssymb,amsfonts,amsthm}
\usepackage[unicode]{hyperref}
\usepackage{graphicx}
\usepackage[draft]{fixme}
\usepackage{tikz}
\usepackage[notref,notcite,final]{showkeys}
\usetikzlibrary{arrows}
\newcounter{thm}
\newtheorem{theorem}[thm]{Theorem}
\newtheorem*{theorem*}{Theorem}
\newtheorem{lemma}[thm]{Lemma}
\newtheorem{proposition}[thm]{Proposition}
\newtheorem*{conjecture*}{Conjecture}
\newtheorem{corollary}[thm]{Corollary}

\newtheorem{example}[thm]{Example}
\newtheorem{remark}[thm]{Remark}
\theoremstyle{definition}
\newtheorem{definition}[thm]{Definition}

\newcommand{\Sing}{\mathop{\mathrm{Sing}}}
\newcommand{\Sep}{\mathop{\mathrm{Sep}}}
\newcommand{\Per}{\mathop{\mathrm{Per}}}

\newcommand{\eps}{\varepsilon}

\newcommand{\bbR}{\mathbb R}

\newcommand{\bbN}{\mathbb N}

\usepackage[hypcap=true]{caption}
\usepackage[hypcap=true]{subcaption}

\graphicspath{{figures/}}
\numberwithin{thm}{section}
\author{N. Goncharuk}
\author{Yu.Ilyashenko}
\thanks{The authors were supported in part by the grant RFBR 16-01-00748}
\title{Large bifurcation supports}
\date{}
\begin{document}
\begin{abstract}
 In the study of global bifurcations of vector fields on $S^2$, it is important to distinguish a set "where the bifurcation actually occurs", -- the bifurcation support. Hopefully, it is  sufficient to study the bifurcation in a neighborhood of the support only.

The first definition of bifurcation support was proposed by V.Arnold in \cite{AAIS}. However this set appears to be too small, see \cite{I}. In particular, the newly discovered effect, an open domain in the space of  three-parametric families on $S^2$ with no structurally stable families \cite{IKS}, is not visible in a neighborhood of the bifurcation support.

In this article, we give a new definition of "large bifurcation support" that accomplishes the task. Roughly speaking, if we know the topological type of the phase portrait of a vector field, and we also know the bifurcation in a neighborhood of the large bifurcation support, then we know the bifurcation on the whole sphere.
\end{abstract}
\maketitle

\tableofcontents
\newpage

\section{Introduction}
\label{sec-intro}
Global bifurcation theory in the plane and two-sphere has two major goals (amidst others): classification of the bifurcations met in generic few-parameter families (with one, two, three parameters), and the study of structural stability (or instability) of these families. In \cite{IKS}, locally generic structurally unstable three parameter families of vector fields on the sphere were discovered. After that the question whether a generic unfolding of a  particular class of degeneracies is structurally stable becomes non-trivial; an \emph{a priori} valid answer ``yes'' (conjectured in \cite{AAIS}) is no more expected.
The present paper is designed to be helpful in the study of any classification and structural stability problem in the global bifurcation theory on the sphere.

\subsection{Who bifurcates?}

Consider a non-hyperbolic, i.e. structurally unstable, vector field. It may have both hyperbolic and non-hyperbolic singular points and limit cycles. Under a generic perturbation, the hyperbolic singular points and limit cycles do not bifurcate, but the non-hyperbolic ones do. A natural question arises: what subset of the phase portrait of a non-hyperbolic vector field actually bifurcates?
The goal of this paper is to answer this question.

For any perturbation of a non-hyperbolic vector field, a closed invariant subset of the phase portrait of this field called \emph{large bifurcation support} (abbreviated as LBS) is distinguished. (We would prefer to use the simpler term \emph {bifurcation support}, but it is already introduced by Arnold \cite[Sec. 3.2]{AAIS}, and has a different meaning.) In order to check whether two perturbations of two orbitally topologically equivalent vector fields are equivalent as the families of  vector fields  on the whole sphere, one has to check only that these families are equivalent in arbitrary small  neighborhoods of their  large bifurcation supports. For example, two generic two-parameter perturbations of orbitally topologically equivalent  vector fields  with a polycycle    ``heart'', see Fig. \ref{fig:lh-heart}, are equivalent iff they are equivalent in an arbitrary small neighborhood  of the polycycle (this is an easy consequence of the results of this paper, but its  proof is not yet written). A similar statement for  vector fields  with a polycycle  ``lune'', Fig. \ref{fig:lh-lune}, is wrong. The reason is that a large bifurcation support  for any perturbation of the first degeneracy coincides with the polycycle ``heart''; for the perturbation of the second degeneracy, the  large bifurcation support may be much larger than the ``lune''.

\begin{figure}[h]
\centering
\subcaptionbox{\label{fig:lh-heart}}{ \includegraphics[width=0.3\textwidth]{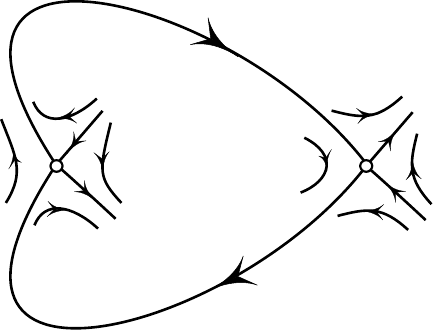}}
 \hfil
\subcaptionbox{\label{fig:lh-lune}}{ \includegraphics[width=0.35\textwidth]{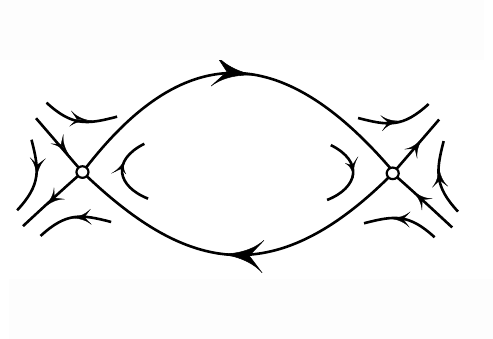}}
\caption{``Heart'' and ``lune'' }\label{fig:lh}
 \end{figure}

This paper is heavily based on one of the key results of the qualitative theory of planar differential equations:
the complete topological classification of their phase portraits. This classification in different equivalent forms is given in \cite{ALGM}, \cite{Mark}, \cite{Neum}, \cite{Peix}. We used the most recent form, so called LMF graphs suggested in \cite{Fed}. This form allows us to make use of  the theory of planar graphs.

A large part of our arguments is based on the topology of planar vector fields  in the spirit of the Poincar\'{e}-Bendixon theorem. When we started working on this paper, we could never suggest that this topology would be so rich; see, e.g., Boundary lemma \ref{lem-U-arcs-top} below.

%

\subsection{Vector fields with finiteness properties}

\begin{definition}
We say that a vector field $v\in Vect(S^2)$ satisfies a
\emph{Lojasiewicz inequality} at $0$ if there is  a $k\in \bbN$, $k\ge 1$, and $c>0$ such that $\|v(x)\| \ge c \|x\|^k$ on some neighborhood of $0$.
\end{definition}

Let $Vect\, S^2$ be the set of all $C^\infty $ vector fields on $S^2$, and $Vect^*\,S^2$ be the set of all vector fields with isolated singular points satisfying Lojasievicz inequality and with finitely many cycles. It is known that  analytic vector fields with isolated singular points have finitely many limit cycles, but this  is a difficult result  \cite{I91}, \cite{E}. 

\begin{conjecture*}  Smooth vector fields met in a generic finite-parameter family belong to $Vect^*\,S^2$.
\end{conjecture*}

We will simply assume that all vector fields that we consider have this property. That is, all through the paper by default a vector field $v$ belongs to $Vect^*\,S^2$.

\subsection{Axiomatic description of the LBS}

Here and below $B$ is an open ball in $\bbR^n$.

\begin{definition}
    \label{def:fam}
    A smooth \emph{family of vector fields} on $S^2$ with the base $B\subset \bbR^n$ is a smooth vector field $V$ on $B \times S^2$ tangent to the fibers $\{\eps\}\times  S^2$, $\eps \in B$.
    The dimension of a family is the dimension of its base.

    We will also write $V=\{v_{\eps}\}_{\eps\in B}$ where $v_{\eps}$ are restrictions of $V$ to each fiber.

    A smooth \emph{local family of vector fields} is a germ of a family of vector fields at $S^2\times \{0\} \subset S^2 \times B$. In other words, it is the family of vector fields with the base $(\bbR^n,0)$.
\end{definition}

In Section \ref{ssec-moderate} below, we define \emph {moderate topological equivalence} of vector fields. For vector fields with hyperbolic singular points on the sphere, this notion is defined in \cite{IKS}. We also recall a classical notion of a weak topological equivalence.
\begin{definition}   \label{def-supp-axiom1}
Suppose that for any local smooth family of vector fields $V =\{v_{\eps}\}_{\eps\in (B,0)} \subset  Vect^*\,S^2$, a closed $v_0$-invariant subset $\Lambda (V)\subset S^2$ is defined. This set is called the \emph{large bifurcation support } of $V$ if it has the following property:

\emph{Let two vector fields $v_0$ and $w_0$ be orbitally topologically equivalent on $S^2$. Let $V$ and $W$ be unfoldings of $v_0$ and $w_0$ that are moderately equivalent in some neighborhoods of $\Lambda (V)$, $\Lambda (W)$; let this moderate equivalence agree with the topological equivalence for $\eps=0$. Then the families $V$ and $W$ are weakly topologically equivalent on the whole sphere.}
\end{definition}

Roughly speaking, equivalence of the unperturbed vector fields on the whole sphere and moderate equivalence of their perturbations in  neighborhoods of their large bifurcation supports imply weak equivalence of the perturbations on the whole sphere. In this defintion, ``to be a large bifurcation support'' is a property of the mapping $V\mapsto \Lambda(V)$, not of an individual set $\Lambda(V)$.

The whole sphere $S^2$ (i.e. the mapping $V\mapsto \Lambda(V)=S^2$) is obviously a large bifurcation support, but it is trivial. Below we give an explicit description of the large bifurcation support for any family of vector fields from $Vect^*\,S^2$, which is in general much smaller than $S^2$. The main result of the paper claims that this set is a large bifurcation support in the sense of Definition~\ref{def-supp-axiom1}. Below we will use the term ``large bifurcation support'' for the set we construct in Sec. \ref{subsub:def-LBS}.

\subsection{Applications}

Classification problems form an essential part of the catastrophe theory. It is crucial to know large bifurcation supports  for the classification of global bifurcations in $k$-parametric families with small $k$.

\section{Definitions and the main result}

\subsection{Separatrices}

Here we briefly recall some known definitions and introduce some new ones.

\begin{definition}
 A singular point $P$ of a vector field $v$ is called \emph{hyperbolic} if both real parts of its two eigenvalues at $P$ are non-zero.
\end{definition}

%

 \begin{figure}[h]
 \begin{center}
\includegraphics[width=\textwidth]{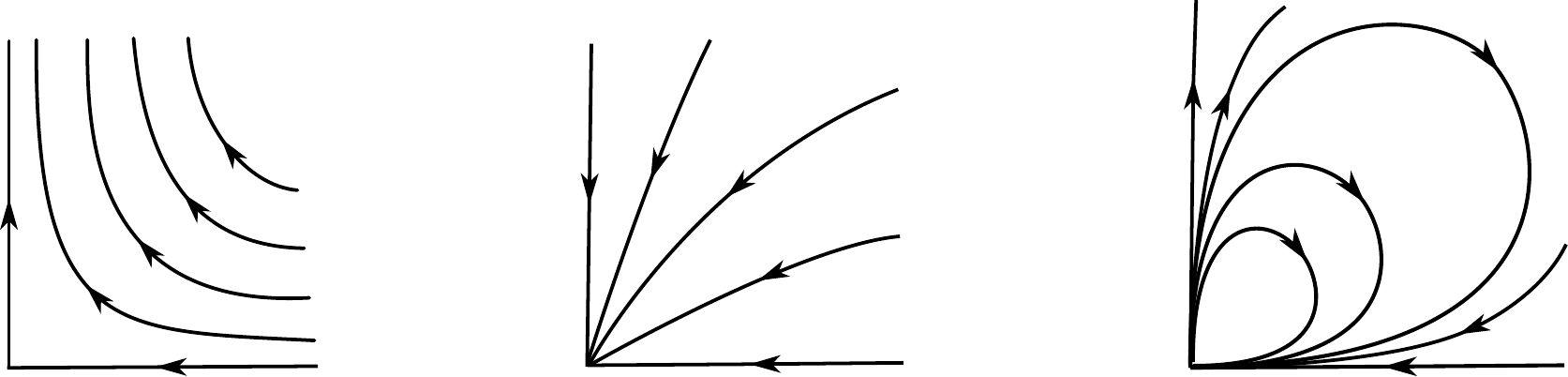}
\end{center}
\caption{Hyperbolic sector, parabolic sector and elliptic sector}\label{fig-sectors}
 \end{figure}

\begin{definition}  \label{def-char} A phase curve of a differential equation on the plane is called a \emph{characteristic trajectory of a singular point} if, as $t \to +\infty $ or $t \to -\infty $, it approaches the singular point and becomes tangent to a straight line.

If a singular point has a characteristic trajectory, it is called \emph{characteristic}.
\end{definition}

The following classical theorem can be found in many sources. It relies on the desingularization theorem \cite{Dum}; see \cite[Sec. 1.5]{DumLlibArt} for the explicit statement and Sec. 3 of the same book for the reduction to the desingularization theorem.

\begin{theorem}  \label{thm:char-pt} Suppose that a $C^{\infty}$-smooth vector field $v$ satisfies Lojasiewicz inequality at all singular points.  Then in a neighborhood of each singular point, it may
\begin{itemize}
 \item be topologically equivalent to a center or a focus;
 \item have a finite sectorial decomposition: namely, it has a neighborhood that is split by characteristic trajectories into a finite union of sectors with smooth boundaries, and in each sector,  $v$ is topologically equivalent to one of the fields shown in Fig. \ref{fig-sectors}. These sectors are called hyperbolic, parabolic and elliptic sector respectively.
\end{itemize}

\end{theorem}

\begin{definition}  A \emph{separatrix} is a phase curve that contains one of two bounding phase curves of  a hyperbolic sector of some singular point $P$.
The separatrix is called \emph{stable} if its $\omega$-limit set is $P$, and \emph{unstable} if its $\alpha$-limit set is $P$. If a curve is a stable and an unstable separatrix simultaneously (for two different singular points or for one and the same), then it is called a \emph{separatrix connection}.
\end{definition}
\begin{remark}   For a hyperbolic saddle, the definition of a separatrix above coincides with the classical one.
\end{remark}

For $v\in Vect^*\,S^2$, we will use the classical Poincare-Bendixson theorem (see e.g. \cite[Sec. 1.7, Corollary 1.30]{DumLlibArt}):
\begin{theorem}[Poincare--Bendixson]
 For each vector field $v\in Vect^*\,S^2$, the $\omega$-limit set (and the $\alpha$-limit set) of each point is one of the following:
  \begin{itemize}
  \item a singular point;
  \item a cycle;
  \item a monodromic polycycle.
 \end{itemize}

\end{theorem}
A polycycle of a vector field  $v$ is a union of a finite number of singular points of $v$ joined by trajectories. The polycycle is called \emph{monodromic} if it admits a Poincare map at least on one side of it.
\subsection{Moderate equivalence}
\label{ssec-moderate}

\begin{definition}
    Two vector fields are called \emph{orbitally topologically equivalent} if there exists a homeomorphism $H$ of the phase space which identifies their phase portraits and preserves the direction of time parametrization on phase curves. We will also say that $H$ conjugates these two vector fields.
\end{definition}

There are three definitions of equivalence of \emph{families} of vector fields: strong, weak, and moderate. Two of them are classical, and the third one is new.

For a family $V=\{v_{\eps}\mid \eps \in B\}$ of vector fields, let $\Sing V$, $\Per V$, and $\Sep V$ be subsets of $B\times S^2$ formed by all singular points,  all limit cycles, and  all  separatrices of $v_{\eps}$ respectively. We will also use the notation $\Sing v \subset S^2$,  $\Per v \subset S^2$, and $\Sep v \subset S^2$ for the union of all singular points, limit cycles, and separatrices of an individual vector field $v$. The set $S(v):=\Sing v\cup \Per v \cup \Sep v$ is an \emph{extended separatrix skeleton} (the set of all \emph{singular trajectories} of $v$) and plays a special role in topological classification of vector fields, see Theorem \ref{th-MNP} below.

\begin{definition}
    \label{def:weak-eq}
    Two local families of vector fields on $S^2$, $V = \{ v_\eps , \eps \in (B,0) \}$ and $ W = \{ w_\eps , \eps \in (B',0) \}$
     are \emph{equivalent} at $\eps=0$ if there exists a map
    \begin{equation}
               \mathbf H: (B,0) \times S^2 \to (B',0)\times S^2, \ (\eps , x) \mapsto (h(\eps ), H_\eps (x)),
    \end{equation}
    such that $h$ is a homeomorphism, $h(0)=0$, and for each $\eps \in (B,0)$ the map $H_\eps : S^2 \to S^2 $ conjugates  $v_\eps $ and $w_{h(\eps )}$.
They are \emph{strongly} equivalent provided that $\mathbf H$ is a homeomorphism on $(B,0)\times S^2$. They are \emph{weakly} equivalent if we do not pose any additional requirements on  $\mathbf H$. They are \emph{moderately} equivalent provided that $\mathbf H$ is continuous with respect to
$(\eps,x )$ on the set
\begin{equation}
\label{eq-set1}
S(v_0)  \cup \partial ((\overline {\Per V} \cup \overline {\Sep V}) \cap \{\eps=0\})
\end{equation}
and $\mathbf H^{-1}$ is continuous with respect to $(\eps,x )$ on the set
\begin{equation}
\label{eq-set2}
S(w_0) \cup \partial ((\overline {\Per W} \cup \overline {\Sep W}) \cap \{\eps=0\})
\end{equation}
\end{definition}

\begin{remark}
 The notion of moderate equivalence  was introduced in \cite{IKS} for hyperbolic vector fields. In this case, \begin{equation}
 \label{eq-subset}
        \partial ((\overline {\Per V} \cup \overline {\Sep V}) \cap \{\eps=0\}) \, \subset \, S(v_0),
\end{equation}
 so it is not necessary to include this set; $\mathbf H$ must be continuous on extended separatrix skeleton only.
 In general, \eqref{eq-subset} does not hold. For example, we may take a bifurcation of a non-hyperbolic node $\dot x = x^3-\eps x, \dot y = y$; the non-hyperbolic node (for $\eps=0$) bifurcates into a saddle surrounded by two nodes (for $\eps>0$). Saddle separatrices of $v_{\eps}$ are vertical and accumulate to trajectories $(x=0, y>0)$ and $(x=0, y<0)$ that do not belong to  $S(v_0)$.
\end{remark}

\begin{remark}    In literature strong equivalence is usually called \textbf{topological equivalence}, weak equivalence is \textbf{weak topological equivalence}.
\end{remark}

Strong equivalence is too restrictive: families with very simple bifurcations may have numerical \cite{MP} and functional \cite{Rous85} invariants that distinguish different equivalence classes.
Weak equivalence is too loose: families with apparently different bifurcations may be weakly topologically equivalent.
Moderate equivalence seems to be more adequate because it takes interesting objects for the one family to the corresponding objects of the other family.

For Definition \ref{def-supp-axiom1} above, we need a local version of moderate equivalence (a moderate equivalence in neighborhoods of given closed invariant subsets). We will apply this version to neighborhoods of large bifurcation supports.
\begin{definition}
\label{def-moderate-local}
    Two local families of vector fields on $S^2$, $V = \{ v_\eps , \eps \in (B,0) \}$ and $ W = \{ w_\eps , \eps \in (B',0) \}$,
     are \emph{moderately equivalent in neighborhoods of closed sets} $Z_1, Z_2\subset S^2$ if
     \begin{enumerate}
      \item $Z_1$ is $v_0$-invariant, and $Z_2$ is $w_0$-invariant;

      \item  There exists a neighborhood $U \supset Z_1$ and a map
    \begin{equation}
               \mathbf H: (B,0) \times U \to (B',0)\times S^2, \ (\eps , x) \mapsto (h(\eps ), H_\eps (x)),
    \end{equation}
    such that $h$ is a homeomorphism, $h(0)=0$, and for each $\eps \in (B,0)$ the map $H_\eps \colon U \to S^2 $ conjugates vector fields $(v_\eps)|_{U} $ and $(w_{h(\eps )})|_{H_{\eps}(U)}$;

 \item $H_0(Z_1)=Z_2$, and moreover,

 \item \label{it-nbhd-cond} For each neighborhood $V$ of $\{\eps=0\}\times Z_1$, its image $\mathbf H(V)$  contains some neighborhood of $\{\eps=0\}\times Z_2$. The same holds for the inverse map $\mathbf H^{-1}$;

    \item \label{it-H-contin}The map $\mathbf H$ is continuous with respect to $(\eps,x )$ on the intersection of its domain with \eqref{eq-set1}.

    The map  $\textbf H^{-1}$ is  continuous with respect to $(\eps, x)$ on the intersection of its domain with \eqref{eq-set2}.

\end{enumerate}
\end{definition}

\begin{remark}
\label{rem-shrink-U}
 If a homeomorphism $\mathbf H$ satisfies the above conditions for two families $V,W$ in neighborhoods of closed sets $Z_1,Z_2$, and the corresponding neighborhood is $U\supset Z_1$, then $\mathbf H$ satisfies these conditions for any smaller neighborhood $U'\subset U$, $U'\supset Z_1$.
\end{remark}

\subsection{Explicit definition of the large bifurcation support}

\subsubsection{Non-interesting limit cycles}

\begin{figure}[h]
\begin{center}
\includegraphics[width=0.3\textwidth]{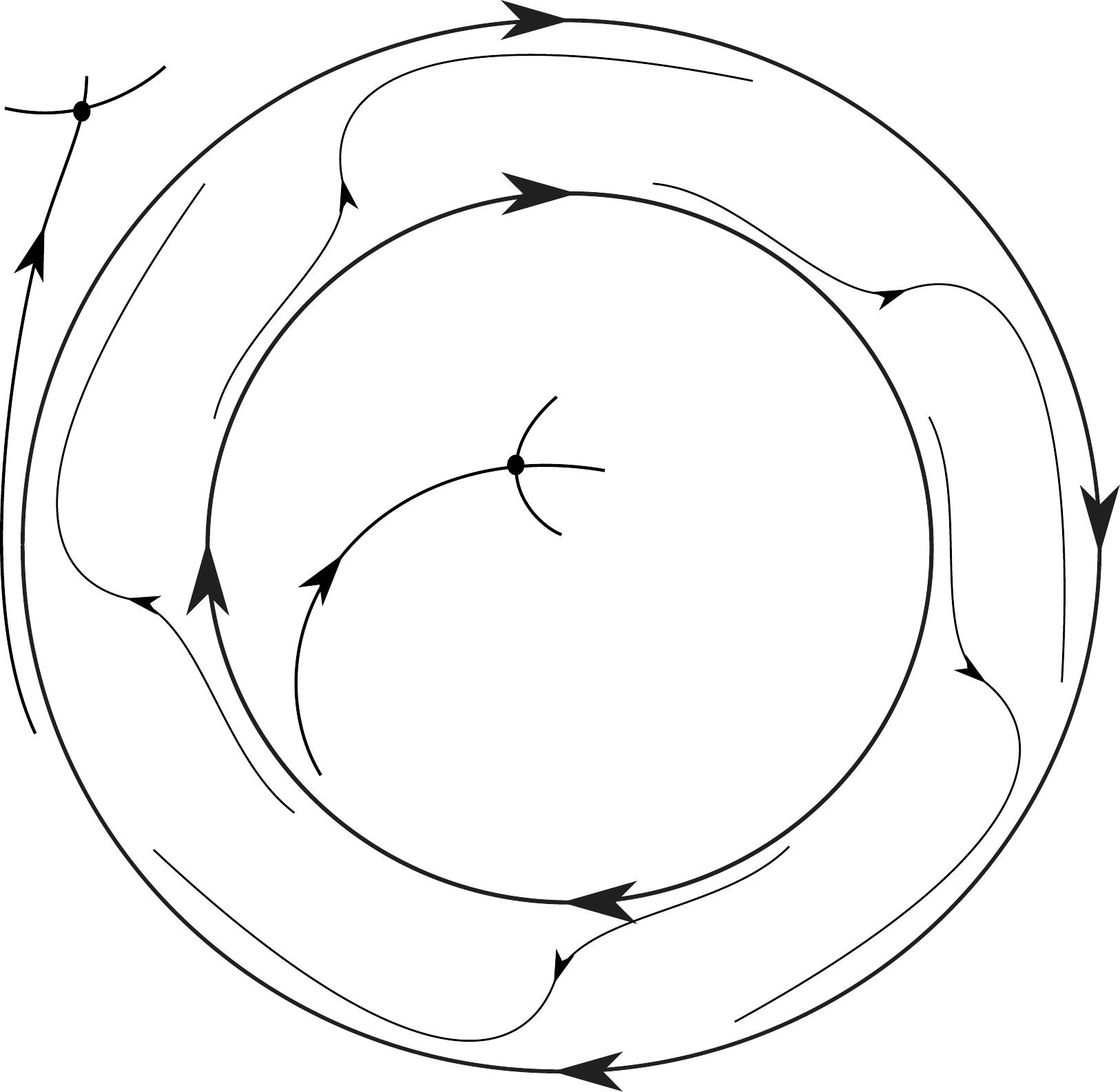}\hfil\includegraphics[width=0.3\textwidth]{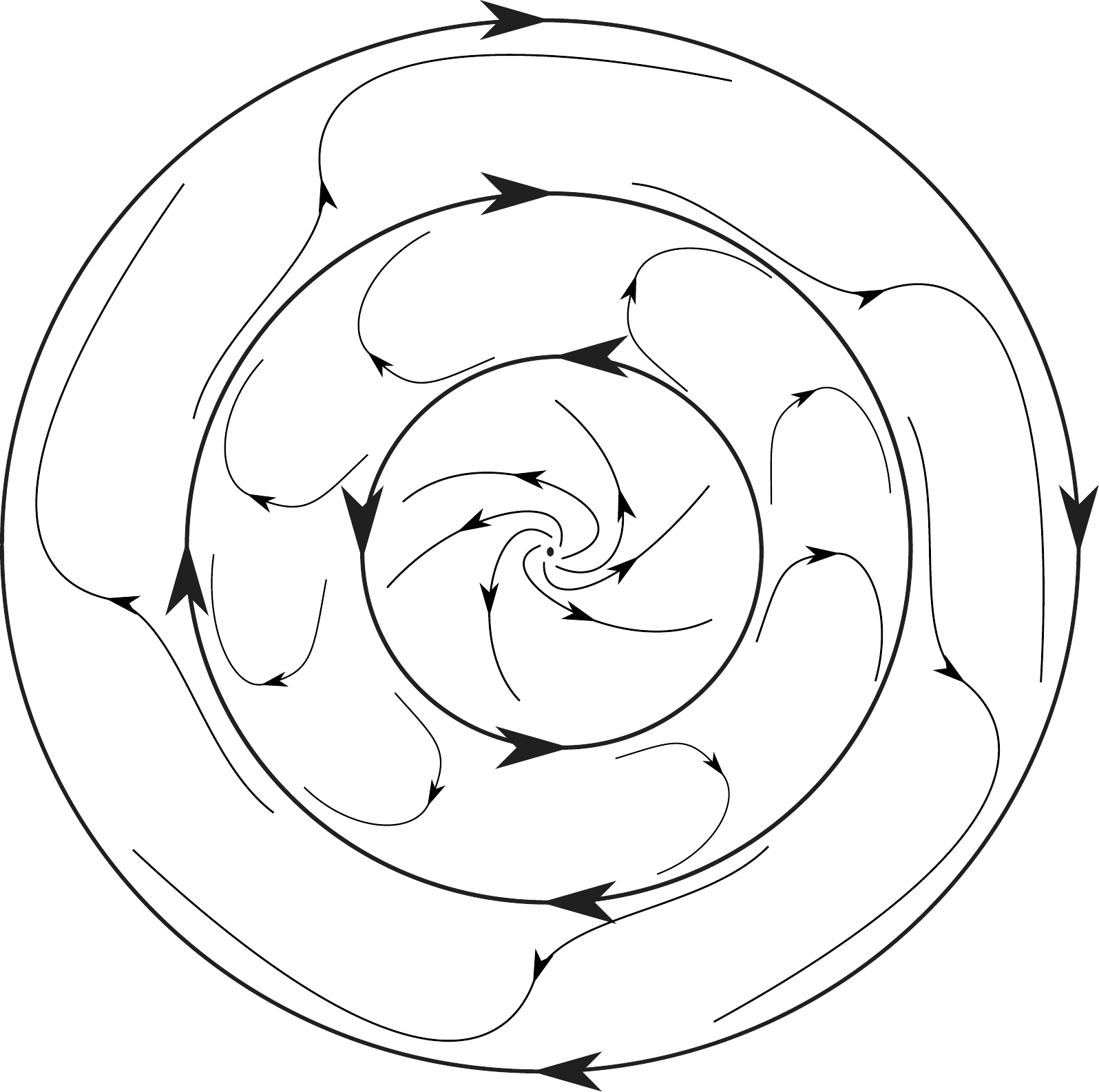}
\end{center}
\caption{Two possible cases for a non-interesting nest}
\end{figure}

\begin{definition} A \emph{nest} of limit cycles of a vector field is the maximal set of nested cycles with no singular points in between them.
\end{definition}
A nest can consist of one limit cycle.

This definition has an additional restriction (absence of singular points between cycles) in comparison with the classical one. 
Clearly, the annulus between two neighboring cycles of one nest is a canonical region, i.e. is filled by trajectories that wind towards these cycles in the past and in the future. To be in one nest is an equivalence relation; given a limit cycle, \emph{its nest} is a unique nest that contains it.

\begin{definition}  A limit cycle is called \emph{semi-stable} if for some choice of coordinate on a small transversal to this cycle,  the Poincar\'{e} map satisfies  $P(x) > x$ for $x\neq 0$ and $P(0)=0$ (here $0$ is the intersection point of the transversal with the cycle).
\end{definition}
Clearly, for $v\in Vect^*\,S^2$, each limit cycle can be attracting, repelling, or semi-stable.
\begin{definition}
\label{def-nic}
A limit cycle of a vector field $v\in Vect^*\,S^2$ is called \emph{non-interesting} if one of the following holds:
\begin{enumerate}
\item \label{it-nonsemist}its nest contains at least one attracting or one repelling limit cycle;
\item \label{it-semist}all the cycles in the nest are semi-stable, but inside the inner cycle or outside the outer one there is only one
hyperbolic singular point.
\end{enumerate}
\end{definition}
\begin{remark}  By the index theorem, this singular point is either attractor or repeller.
\end{remark}

Note that hyperbolic cycles are all  non-interesting due to the definition.
The motivation for this definition is the following: when a non-interesting cycle bifurcates,
nothing interesting happens; there is no interaction between the dynamics inside and outside it.

\begin{definition}  An $\alpha $- or $\omega $-limit set of a non-singular point of a vector field is called \emph{non-interesting} if it is a hyperbolic repeller (respectively, attractor), or a non-interesting limit cycle. Otherwise it is called \emph{interesting}.
\end{definition}

\subsubsection{Large bifurcation support: an explicit definition}
\label{subsub:def-LBS}
\begin{definition}  \label{def:elbs}   \emph{Extra large bifurcation support} $ELBS(v_0)$ of a vector field $v_0$ is the union of all non-hyperbolic singular points and non-hyperbolic limit cycles of $v_0$, plus the closure of the set of all nonsingular points for which both $\alpha $- and $\omega $-limit sets are interesting.
\end{definition}

\begin{remark}
 $ELBS(v_0)$ contains all non-singular points of $v$ except (open) basins of attraction and repulsion of non-interesting $\alpha$-, $\omega$-limit sets. However we retain and include in $ELBS(v_0)$ all non-hyperbolic limit cycles, including non-interesting ones. As for singular points, $ELBS(v_0)$ contains all non-hyperbolic singular points and does not contain hyperbolic attractors and repellers. It contains a hyperbolic saddle if and only if one of its unstable (stable) separatrices has an interesting $\omega$- ($\alpha$-) limit set.
\end{remark}

Now the main definition comes.

\begin{definition}   \label{def:lbs}
\emph{Large bifurcation support} of a local family $V$ of vector fields is
$LBS(V) = ELBS(v_0) \cap \left(\Sing v_0\cup (\overline {\Per V} \cup \overline{\Sep V}) \cap \{\eps=0\}\right)$.
\end{definition}

So $LBS(V)$ contains all singular points and cycles of $v_0$ that belong to $ELBS(v_0)$ (see Remark \ref{rem-shrink-U}) and  all non-singular accumulation points of cycles and separatrices of $v_{\eps}$, $\eps\to 0$, if these accumulation points have interesting $\alpha$- and $\omega$-limit sets.

\begin{remark}
 For vector fields with hypebolic fixed points only, $LBS(V)$ depends only on $v_0$.
 It is not clear whether this is the case for all generic families.
\end{remark}

\begin{figure}[h]
  \includegraphics[width=0.8\textwidth]{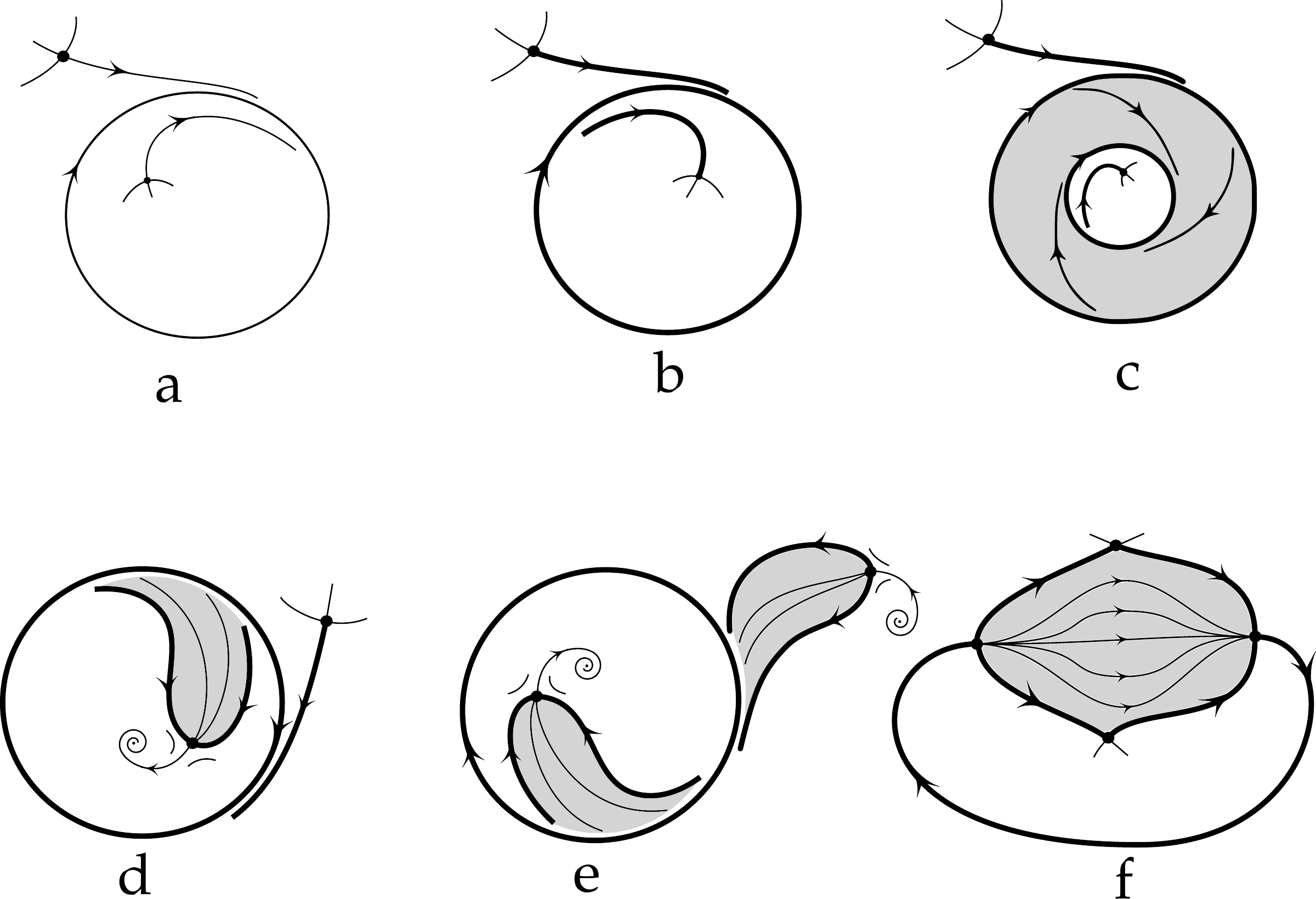}
\caption{Examples of a vector field $v_0$; large bifurcation supports for generic unfoldings of $v_0$ are shown in gray and thick. We only show interesting parts of phase portraits; some hyperbolic sinks and sources are not shown.}\label{LBS-ex}
 \end{figure}

\begin{example}
Consider an unfolding $V$ of each of the vector fields $v_0$ shown on Fig. \ref{LBS-ex}. In each case, the number of parameters in $V$ equals the codimension of the degeneracy of $v_0$, and $V$ is a generic family. The large bifurcation support $LBS(V)$ is shown in thick curves and gray domains. In more details:
\begin{itemize}
 \item Fig. \ref{LBS-ex}a (generic vector field $v_0$): the set $LBS(V)$ is empty.
 \item Fig. \ref{LBS-ex}b (degeneracy of codimension 1): the set $LBS(V)$
contains the limit cycle and the two saddles with their separatrices winding to the cycle in the positive or negative time. The cycle is interesting.
\item  Fig. \ref{LBS-ex}c (degeneracy of codimension 2): the set $LBS(V)$
contains the two saddles, their separatrices that wind to the cycles in the positive or negative time, and the whole annulus between the cycles. Both cycles are interesting; one may prove that saddle connections accumulate to all trajectories inside the annulus.
\item  Fig. \ref{LBS-ex}d (degeneracy of codimension 2): the set $LBS(V)$
contains the outer saddle, its separatrix that winds  onto the cycle, the cycle itself, and the closure of the parabolic sector of the saddlenode.
\item  Fig. \ref{LBS-ex}e (degeneracy of codimension 3): the set  $LBS(V)$
contains  the cycle and the closures of parabolic sectors of  saddlenodes.
\item   Fig. \ref{LBS-ex}f  (degeneracy of codimension 3. This is the polycycle collection ``lips'' studied by Kotova and Stanzo, \cite{KS}):  $LBS(V)$
contains the separatrix connection between saddlenodes and the closure of the common parabolic sector of the saddlenodes, because cycles of $v_{\eps}$, $\eps\to 0$, accumulate to all these ordits, as shown in \cite{KS}.
\end{itemize}
\end{example}

\subsection{Main Theorem}

\begin{theorem}   \label{thm:main}  Large bifurcation support $LBS(V)$ defined above satisfies Definition~\ref{def-supp-axiom1}.
\end{theorem}

\begin{remark}
 The set $ELBS(v_0)$ is also a large bifurcation support in terms of Definition~\ref{def-supp-axiom1}.
 The proof is completely analogous to that for $LBS(V)$; one may check that $ELBS(v_0)$ satisfies the properties listed in Sec. \ref{sec-prop-LBS}, which are the only properties we need in the proof of the Main theorem below.

 However we prefer to prove the stronger result, for a smaller set $LBS(V)$.
\end{remark}

Let us give a more detailed and slightly improved statement of the same theorem.

\begin{theorem}[Main Theorem]  \label{thm:main1}  Let two vector fields $v_0$ and $w_0$ be orbitally topologically equivalent on $S^2$; denote the corresponding homeomorphism by $\hat H $. Let $V=\{v_{\eps}, \eps\in (B,0)\} \subset Vect^*\,S^2, W=\{w_{\eps}, \eps\in (B',0)\} \subset Vect^*\,S^2$ be smooth families unfolding these fields. Suppose that there exists a neighborhood $U$ of $LBS(V)$ and a map
$$
\mathbf H: (B,0) \times U \to  (B',0) \times S^2, \ \mathbf H(\eps ,x) = (h(\eps ), H_\eps (x)),
$$
$h(0)=0$, 
which is a moderate equivalence of $V,W$ in neighborhoods of $LBS(V), LBS(W)$ in the sense of Definition \ref{def-moderate-local}. Suppose moreover that $\hat H|_U = H_0$.

Then the families $V$ and  $W$ are weakly equivalent on the whole sphere; namely there exists a map
$$
\mathbf{\hat H}: (B,0) \times S^2 \to (B',0) \times S^2, \ \mathbf{\hat H}(\eps ,x) = (h(\eps ), \hat H_\eps (x))
$$
that provides a weak equivalence of the families $V$ and $W$.
\end{theorem}

\begin{remark}
 We do \textbf{not} assert that $\hat H_{\eps}|_{U} =H_{\eps}$, and this  is not true in the general case.
\end{remark}

\begin{remark}
 Remark \ref{rem-shrink-U} above shows that moderate equivalence in some neighborhood of $LBS(V)$ implies moderate equivalence in any sufficiently  small neighborhood of $LBS(V)$. We will have to shrink $U$ in the proof. 
\end{remark}

\begin{remark}   Note that the maps $\mathbf H$ and $\mathbf{\hat H}$ are skew products over \emph{the same} map $h$ of the bases. This is the only difference between Theorem \ref{thm:main} and Main Theorem.
\end{remark}

The Main theorem \ref{thm:main1} solves Problem 1 from \cite{I}. Up to now, this is the only general statement about bifurcations in the
families of vector fields with an arbitrary number of parameters. Several tempting conjectures about such bifurcations were suggested in \cite{AAIS}, but they all turned to be wrong \cite{KS}, \cite{IKS}, \cite{YuINS}.
The authors do not know any other non-trivial statement, even a conjecture, about bifurcations in generic families with an arbitrary number of parameters that would seem to be true.

\section{Strategy of the proof}
Our goal is to establish, for small $\eps$, an orbital topological  equivalence of two planar vector fields $v_{\eps}$, $w_{h(\eps)}$. We will use the criterion of orbital topological equivalence due to R.Fedorov \cite{Fed} (based on the classical book \cite{ALGM}); this result is close to the results due to  L.Markus,  D.Neumann, and M.M.Peixoto \cite{Mark}, \cite{Neum}, \cite{Peix}. In the following subsection we present this result. Simultaneously  we recall the notion of canonical regions and describe the properties of these regions needed in the future.

\subsection{Separatrix skeletons and canonical regions} \label{sub:skel}

First, let us formulate the result of L.Markus,  D.Neumann, and M.M.Peixoto \cite{Mark}, \cite{Neum}, \cite{Peix} following the book of Dumortier, Llibre and Artes \cite{DumLlibArt}, Sec.1.9. We only consider the case $v \in Vect^*\,S^2$; the result holds true for arbitrary $C^{\infty}$-smooth vector fields, but definitions should be modified for this general case (see \cite{DumLlibArt}).

\begin{definition} A \emph{ separatrix skeleton} of a vector field $v\in Vect^*\,S^2$ is $\Sing v \cup \Sep v$.
An \emph{extended separatrix skeleton} $S(v)$ of a vector field $v\in Vect^*\,S^2$ is $S(v):=\Sing v \cup \Per v \cup \Sep v$. A \emph{canonical region} of $v$ is a connected component of its complement $\mathbb R^2 \setminus S(v)$.
\end{definition}

We will use the extended separatrix skeleton rather than the  separatrix skeleton.

\begin{proposition}
\label{prop-canonreg-common-limset}
For $v\in Vect^*\, (S^2)$, all points of the same canonical region $R$ have coincident $\alpha$- and $\omega$-limit sets.
\end{proposition}
\begin{proof}
First, prove that the set of points in $R$ with the same $\omega$-limit set is open. This follows from continuous dependence of solutions of ODEs on initial data. Indeed, take $x\in R$. 

If $\omega(x)$ is a cycle or a monodromic polycycle, then the future semi-trajectory of $x$ intersects a transversal loop around this cycle or polycycle. So the trajectories  starting in a neighborhood of $x$ also intersect this loop, and thus have the same $\omega$-limit set as $x$. The statement is proved.

If $ \omega(x)=:P$ is a singular point, we use its sectorial decomposition, see Theorem \ref{thm:char-pt}. Note that the future semi-trajectory of $x$ enters an attracting parabolic sector or an elliptic sector of $P$. It may not enter a hyperbolic sector because $x$ does not belong to a separatrix. If a future semi-trajectory of $x$ enters an attracting parabolic sector, then future semi-trajectories starting in some neighborhood of $x$ enter the same attracting parabolic sector of $P$, thus their $\omega$-limit set is also $P$. If a future semi-trajectory of $x$ enters an elliptic sector, future semi-trajectories starting in its neighborhood may enter either the same elliptic sector of $P$, or an adjacent attracting parabolic  sector of $P$. In any case, their $\omega$-limit set is $P$.

Since $R$ is connected, it cannot be a union of several disjoint open sets. So all points of $R$ have the same $\omega$-limit set. The same arguments apply to the $\alpha$-limit set.
\end{proof}

\begin{definition}
 The \emph{completed} separatrix skeleton of a vector field $v\in Vect^*\,S^2$  is the union of the extended separatrix skeleton  together with one orbit from each one of the canonical regions.

 Two completed separatrix skeletons $C_1,C_2$ are \emph{topologically equivalent}
if there exists a homeomorphism from $S^2$ to $S^2$
that maps the orbits of $C_1$ to the orbits of $C_2$ preserving the orientation.
\end{definition}

\begin{theorem}[Markus--Neumann--Peixoto Theorem]
\label{th-MNP}
Assume    that $v_1, v_2 \in  Vect^*\,S^2$. Then $v_1,v_2$ are topologically equivalent if and only if their completed separatrix skeletons are equivalent.
\end{theorem}

The following proposition, see \cite[Proposition 1.42, p. 34]{DumLlibArt}, gives a list of possible canonical regions. It motivates the fact that separatrix skeletons classify vector fields: on the complement to $S$, the dynamics is trivial.
    \begin{proposition}
    \label{prop-parall}
      Every canonical region of $v\in Vect^*\,S^2$ is parallel, i.e. topologically equivalent to one of the following:
\begin{itemize}
 \item A strip flow, defined on $\bbR^2$ by the system of differential equations $\dot x=1, \dot y=0$;
\item A spiral flow, defined on $\bbR^2 \smallsetminus \{0\}$ the system of differential equations $\dot r=r$, $\dot \theta=0$ in polar coordinates.
\end{itemize}
    \end{proposition}
    The book \cite{DumLlibArt} also lists the case of an annular flow, given on $\bbR^2 \smallsetminus \{0\}$ by $\dot r=0, \dot \theta =1$ in polar coordinates.  This case corresponds to the infinite set of cycles, hence is not possible for $v \in Vect^* \, S^2$.

We will also need the following corollary of Proposition \ref{prop-canonreg-common-limset}:
\begin{proposition}
\label{prop-dELBS}
For $v \in Vect^*(S^2)$, each canonical region of $v$ either belongs to $ELBS(v)$, or does not intersect it. In particular, $\partial ELBS(v) \subset S(v)$.
\end{proposition}
\begin{proof}
 Consider the set of all points in $S^2\setminus S(v)$ whose $\alpha$- and $\omega$-limit sets are interesting. Due to Proposition \ref{prop-canonreg-common-limset}, this set is a union of several canonical regions of $v$. Note that $ELBS(v)$ is the closure of this set plus some subset of $S(v)$, which implies the statement.
\end{proof}

\subsection{LMF graphs}
\label{sec-LMF}
The extended separatrix skeleton is not a graph on a sphere, because separatrices can wind around limit cycles. However we may turn it into a graph if we truncate the separatrices to their intersections with transversal loops of their $\alpha$- or $\omega$-limit sets.

In \cite{Fed}, R. Fedorov assigned a graph to each vector field on the plane and proved that two vector fields are orbitally topologically equivalent if their graphs are isotopic in $S^2$. The proof was based on the classical book \cite{ALGM} where the complete set of topological invariants was given in the form of  ''schemes``. We will use the graphs introduced by Fedorov, and we will call them  \emph{LMF graphs} (Leontovich, Mayer, Fedorov graphs) of planar vector fields.

In this section, we recall the construction of LMF graphs.
We only consider vector fields from $Vect^*\,S^2$.

Choose an orientation on $S^2$.

\begin{definition}
For a domain in $S^2$ with smooth boundary, we say that the boundary is oriented clockwise (resp. counterclockwise) with respect to the domain if the domain is to the right  (resp. to the left) of its oriented boundary.

 Let a closed curve $\gamma$ on $S^2$ be disjoint to a connected set $D$. We say that $\gamma$ is oriented clockwise (resp. counterclockwise) with respect to $D$ if it is oriented clockwise (resp. counterclockwise) with respect to the disk it bounds on the sphere that contains $D$.
\end{definition}

\textbf{Transversal loops around $\alpha$- and $\omega$-limit sets}

Given a smooth vector field $v$ on $S^2$, choose a transversal loop for each side of each its limit cycle, each monodromic side of each its polycycle, and around each attracting or repelling singular point of $v$. We assume that the annulus between the transversal loop and the corresponding  $\alpha$- or $\omega$-limit set does not  contain points of other transversal loops, and  the vector field $v$ in this annulus is orbitally topologically equivalent to the standard vector field $\dot r = \pm (1-r), \dot \phi =1$ in $\{r>1\}$.

Fix a counterclockwise orientation on the chosen loop with respect to the corresponding  cycle, polycycle or singular point. From now on, we always consider transversal loops with this orientation.

\textbf{Truncated separatrices}

If some  separatrix $\gamma$ of a singular point $P$ of $v$ crosses a transversal loop $l$ chosen above,  consider a \emph{truncated separatrix}: an arc of $\gamma$ between $P$ and the cross-point of $\gamma$ with $l$. 

\begin{remark}
Assume that an outgoing separatrix $\gamma$ of $P$ does not cross such loops. Poincare-Bendixson theorem implies that its $\omega$-limit set can only  be a characteristic point,  $\omega(\gamma)=Q$. So this separatrix is a characteristic trajectory for $Q$, and  its germ at $Q$ is $C^1$-smooth (see Theorem \ref{thm:char-pt}). We conclude that all non-truncated separatrices are $C^1$-smooth curves that join singular points of $v$.
\end{remark}

\begin{definition}
\label{def-LMF}
\emph{LMF graph} of a vector field $v\in Vect^* \, S^2$ is a graph $LMF(v)$ embedded in $S^2$ which consists of the following elements:
    \begin{itemize}
        \item Vertices:

        \begin{enumerate}
            \item All singular points of $v$;

            \item All \emph{truncation vertices}: cross-points of separatrices of $v$ with transversal loops chosen above;

            \item A point on each cycle;

            \item A point on each \emph{empty} transversal loop, i.e. on the transversal loop that does not cross separatrices of $v$.
        \end{enumerate}

        \item Edges:

        \begin{enumerate}
            \item Unstable (stable) separatrices of singular points, if their $\omega$- (resp., $\alpha$-)limit sets are characteristic points.

            \item Truncated unstable (stable) separatrices of singular points, if their $\omega$- (resp., $\alpha$-) limit sets are not characteristic points.

            \item Limit cycles (this edge starts and ends at the vertex of type 3).

              \item Pieces of transversal loops between subsequent truncation vertices, or the whole empty transversal loops.

              \item One homoclinic trajectory of $v$ in each elliptic sector of a non-hyperbolic singular point.
    \end{enumerate}
    \end{itemize}
    \textbf{Orientation}

	 The orientation of edges of types 1, 2, 3, 5 is induced by time parametrization. The orientation of edges of type 4  is counterclockwise with respect to the $\alpha$- or $\omega$-limit set corresponding to the transversal loop, as mentioned above.

    \textbf{Labeling}

    LMF graph is considered together with the following labels.

    Each vertex is labeled by the description of its type, namely the labels say \emph{Singular Point (SP),  Truncation Vertex (TV), Vertex on a Limit Cycle (VLC), Vertex on an Empty Transversal Loop (VETL)}. Similarly, the labels on the edges say  \emph{Stable Separatrix (SS), Unstable Separatrix (US), Separatrix Connection (SC), Stable Truncated Separatrix (STS), Unstable Truncated Separatrix (UTS), Limit Cycle (LC), Outgoing Transversal Loop (OTL), Ingoing Transversal Loop (ITL), Trajectory in the Elliptic Sector (TES)}. We say that a transversal loop is \emph{ingoing} if this is a loop around its $\omega$-limit set; otherwise we say that the transversal loop is \emph{outgoing}. 

    \end{definition}

Fig. \ref{fig-Mayer} shows the part of the phase portrait of a vector field and the corresponding part of the LMF graph. We used abbreviations of the labels described above. 

   The relation of the $LMF$ graphs with  separatrix skeletons is the following.
    The  edges of the $LMF$ graph except transversal loops and loops in elliptic sectors belong to the extended separatrix skeleton,  and their orientation is induced by the time parametrization. The face of an LMF graph may be:
    \begin{itemize}
     \item  a canonical region of $v$, possibly truncated by transversal loops of its $\alpha$- and $\omega$-limit sets, which depends on types of these $\alpha$-, $\omega$-limit sets. It will be possibly cut by a loop in an elliptic sector;
     \item a petal in an elliptic sector;
     \item an annulus between an $\alpha$- or $\omega$-limit set of $v$ and its transversal loop.
    \end{itemize}
  The orbits in canonical regions that are included to the completed separatrix skeleton keep the same information as labeling.

\begin{figure}[h]
\begin{center}
\includegraphics[width=\textwidth]{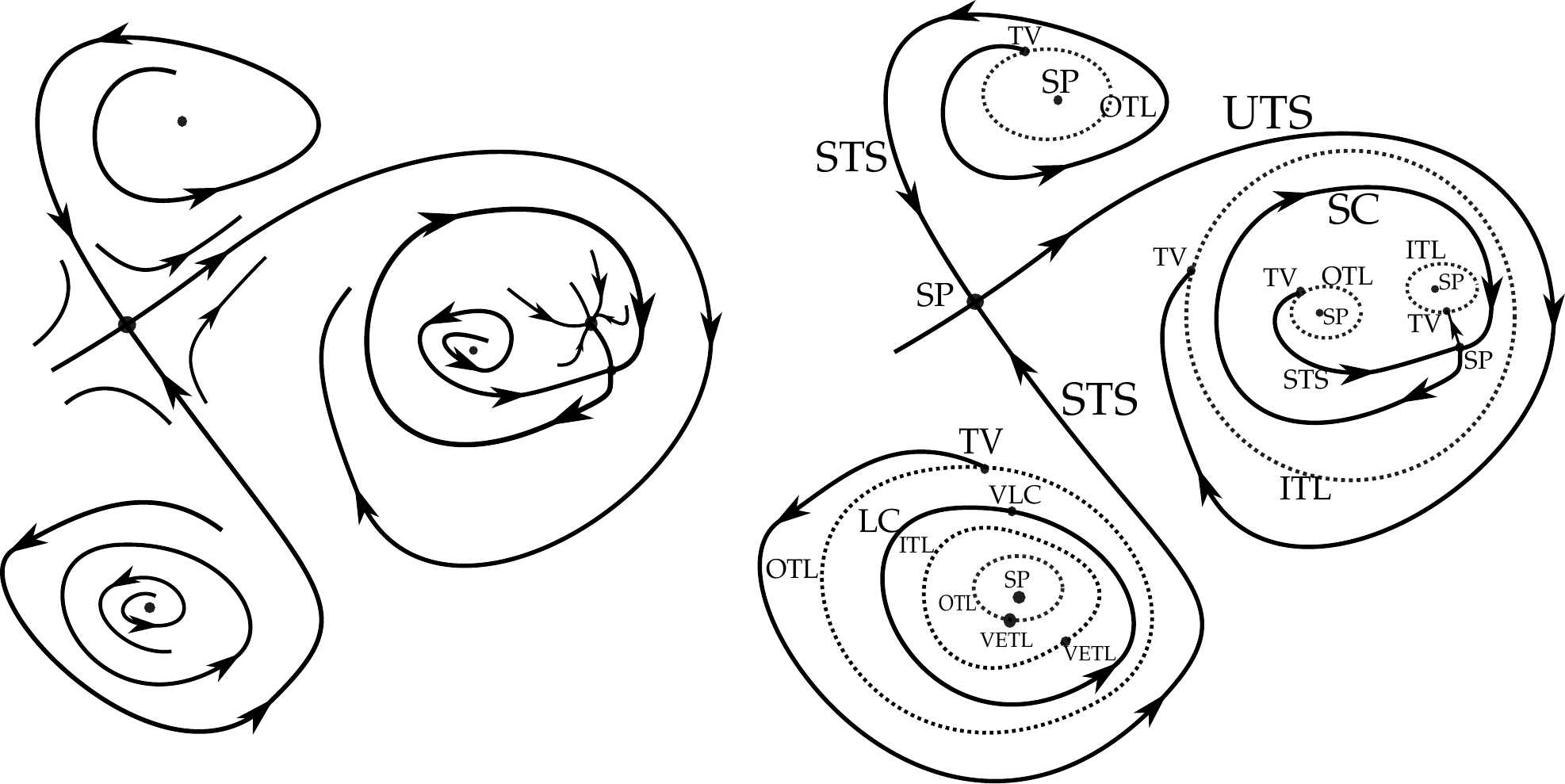}
\end{center}
\caption{A phase portrait of a vector field and its LMF graph. For the meaning of labels, see Definition \ref{def-LMF}.}\label{fig-Mayer}
\end{figure}

The classification of canonical regions (Proposition \ref{prop-parall}) yields the following classification of faces of $LMF$ graphs:
\begin{lemma}[Faces of the LMF graph]
    \label{lem-faces}
      Each open face $F$ of the LMF graph of a smooth vector field $v\in Vect^* \, S^2$ is either a topological disc, or a topological annulus. In the second case, the following cases are possible:

      \begin{itemize}
        \item $F$ is a domain between an $\alpha$- or $\omega$- limit set (sink or source, cycle, or polycycle) of $v$ and its transversal loop;
        \item $F$ is a domain between two transversal loops (of sinks, sources, cycles or   polycycles) of $v$.
      \end{itemize}

    \end{lemma}
\begin{proof}
 This follows from Proposition \ref{prop-parall} above. Indeed, some faces of the LMF graph of $v$ are annuli between  $\alpha$- or $\omega$- limit sets of $v$ and their transversal loops. To obtain all other faces of the  LMF graph, we can take all canonical regions of $v$ and truncate them by transversal loops mentioned above: we cut off pieces of canonical regions that are between $\alpha$- or $\omega$- limit sets of $v$ and their transversal loops, and possibly cut along loops in elliptic sectors.
 
If a canonical region carries a strip flow, its  $\alpha$- and $\omega$-limit sets may be surrounded or not surrounded  by transversal loops; these loops will intersect the canonical region in topological intervals transversal to the flow. If this canonical region contains a loop in elliptic sector, this loop  intersects a canonical region in a topological interval along the flow. In any case, after truncation and cutting, this canonical region will produce face(s) of $LMF(v)$ topologically equivalent to disc(s). 

For a  canonical region carrying a spiral flow, its $\alpha$- and $\omega$-limit sets are necessarily surrounded by transversal loops. This canonical region after truncation will become an annular face of $LMF(v)$  between two transversal loops of   $\alpha$- and $\omega$-limit sets.
\end{proof}

    We use the following result of R. Fedorov \cite{Fed}, based on the previous result of Andronov, Leontovich, Gordon, Mayer \cite{ALGM}. This result is close to Theorem \ref{th-MNP}.

\begin{theorem}[R. Fedorov, \cite{Fed}]
\label{th-Fedorov}
    If two LMF graphs $\Gamma_1 = LMF(v), \Gamma_2=LMF(w)$ of two vector fields $v,w$ are isotopic on the sphere (i.e. there exists an orientation-preserving homeomorphism of the sphere which maps one to another, preserves orientation on edges and matches labels on edges and vertices), then $v$ and $w$ are orbitally topologically equivalent.
\end{theorem}

\begin{remark}
Theorem \ref{th-Fedorov} in \cite{Fed} was proved for a slightly different construction of the graph. Below we list the differences.

 \begin{itemize}
  \item We label  transversal loops as ingoing or outgoing, while Fedorov puts these labels on singular points and each side of cycles or polycycles themselves.
  \item  Fedorov does not add transversal loops around singular points if they are characteristic attractors or repellers.
  \item Fedorov does not add empty transversal loops.
  \item We do not describe a labeling for center-type vertices (as Fedorov does) because they do not appear in $Vect^* \, S^2$, due to its definition.
 \end{itemize}

 These differences do not affect Theorem \ref{th-Fedorov}. Indeed, let two vector fields $v_1,v_2$ have isotopic LMF graphs. Recall that all attractors (both characteristic and non-characteristic) are locally topologically equivalent, and the same holds for repellers. So we may and will assume that all attractors and repellers of $v_1,v_2$ are characteristic.  Prove that $v_1$ and $v_2$ have isotopic Fedorov's graphs. 

Indeed, looking at the graph (with its embedding into $S^2$), one can  determine which transversal loop corresponds to which $\alpha$- or $\omega$-limit set, and put labels (\emph{attracting, repelling}) on each side of these sets as in Fedorov's graph. Further, we erase empty transversal loops from both graphs (the only information they bear is labeling). Finally, we remove transversal loops around all attractors and repellers and let truncated separatrices that terminated at these loops enter the singular points themselves. Since the LMF graphs were isotopic, the resulting Fedorov's graphs will be isotopic. 

This reduces Theorem \ref{th-Fedorov} for Fedorov's graphs to Theorem \ref{th-Fedorov} for LMF graphs described above. 
%
\end{remark}

The proof of Main Theorem will consist of proving the isotopy of LMF graphs of $v_{\eps}$, $w_{h(\eps)}$ for small $\eps$.

\subsection{Isotopy of graphs on $S^2$}

 We use the following theorem from graph theory (see \cite[Theorem 2]{Graph} for the more general result).

    \begin{theorem}
    \label{th-graphs}
     Suppose that two oriented connected planar graphs $\Gamma_1, \Gamma_2$ are embedded in $S^2$ by maps $\phi_1 \colon \Gamma_1 \to S^2, \phi_2 \colon \Gamma_2 \to S^2$. Choose an orientation in $S^2$.

	Suppose that $g\colon\Gamma_1\to \Gamma_2 $ is an isomorphism of oriented graphs $\Gamma_1, \Gamma_2$, and suppose that the graph isomorphism $g$ preserves a counterclockwise order of edges at each vertex (induced by the immersions $\phi_1, \phi_2$).
	
        Then the map $\phi_2\circ g \circ \phi_1^{-1}$ can be extended to the orientation-preserving homeomorphism of $S^2$, in particular $\phi_1(\Gamma_1)$ is isotopic to $\phi_2(\Gamma_2)$.
    \end{theorem}
        The idea of the proof of this theorem is to establish the correspondence of faces of $\Gamma_1, \Gamma_2$ using the information on the order of edges in each vertex, and to define a sphere homeomorphism inside each face.

    LMF graphs are usually not connected; some of their faces can be annuli, see Lemma \ref{lem-faces} above.  We will use the following theorem.

      \begin{theorem}
    \label{th-graphs-components}
        Suppose that two oriented planar graphs $\Gamma_1, \Gamma_2$ (not neccessarily connected) are embedded in $S^2$ by maps $\phi_1 \colon \Gamma_1 \to S^2, \phi_2 \colon \Gamma_2 \to S^2$, and their (open) faces in $S^2$ are topological discs or annuli. Choose an orientation in $S^2$.

	Suppose that these graphs are isomorphic as oriented graphs. Suppose that the graph isomorphism $g$ preserves a counterclockwise order of edges at each vertex (induced by the immersions $\phi_1, \phi_2$). Suppose that the map $\phi_2\circ g \circ \phi_1^{-1}$ extends to an orientation-preserving homeomorphism of the annuli-shaped faces.
	
        Then the the map $\phi_2\circ g \circ \phi_1^{-1}$ can be extended to the orientation-preserving homeomorphism of $S^2$, so $\phi_1(\Gamma_1)$ is isotopic to $\phi_2(\Gamma_2)$.
    \end{theorem}

\begin{proof}
The idea of the proof is to add edges through all annuli-shaped faces of our graph, so that the extended graph is connected, and then use Theorem \ref{th-graphs}. Formally, for each annuli-shaped face we do the following.

 Let $A_1 \subset S^2$ be an annuli-shaped open face of $\phi_1(\Gamma_1)$, and let $G$ be the homeomorphism that extends $\phi_2\circ g \circ \phi_1^{-1}$ to $A_1$. Then $A_2:=G(A_1)$ is an open face of $\phi_2(\Gamma_2)$.

 Let $\phi_1(V_1),\phi_1(V),\phi_1(V_2) \in S^2$ be three subsequent vertices on one of the two boundary components of $A_1$ (the orientation on $\partial A_1$ is induced by the orientation on $A_1\subset S^2$). Let $\phi_1(W_1),\phi_1(W),\phi_1(W_2)$ be  three subsequent vertices on another boundary component of $A_1$. If one of boundary components contains only two vertices, we put $V_1=V_2$; if it contains only one vertex, we put $V_1=V=V_2$.

 Take a continuous curve $\gamma \subset A_1$ joining $\phi_1(V)$ to $\phi_1(W)$. To the graph $\Gamma_1$, add the edge joining $V$ to $W$. Extend $\phi_1$ so that $\phi_1([VW])=\gamma$.

 Take a curve $G(\gamma) \subset A_2$ joining $G(\phi_1(V))$ to $G(\phi_1(W))$. Similarly, to the graph $\Gamma_2$, add the edge joining $\phi_2^{-1}(G(\phi_1(V))) = g(V)$ to $\phi_2^{-1}(G(\phi_1(W)))= g(W)$. Extend $\phi_2$ so that it takes this edge to the curve $G(\gamma)$.

 Finally, extend $g$ to the graph isomorphism of enlarged graphs, by  putting $g([VW]) = [g(V), g(W)])$.

 A counterclockwise order of edges at $V$ contained a part $[VV_2], [VV_1]$; now this part changed to $[VV_2], [VW],[VV_1]$. A counterclockwise order of edges at $g(V)$ contained a part $[g(V)g(V_2)], [g(V)g(V_1)]$; now this part changed to $[g(V)g(V_2)], [g(V) g(W)], [g(V)g(V_1)]$. So $g$ still preserves a counterclockwise order of edges at $V$; similarly, it preserves the order at $W$.

 We repeat this process for each annuli-shaped face. Finally, we get \emph{connected} graphs $\tilde \Gamma_1, \tilde \Gamma_2$, because the number of their connected components decreases after each step of extension. These new graphs satisfy the assumptions of Theorem \ref{th-graphs}.

 So the initial map $\phi_2\circ g \circ \phi_1^{-1}$ (as well as the extended one) can be extended to the homeomorphism of the sphere.

\end{proof}

\subsection{Idea of the proof of the Main Theorem}

We are going to prove that under assumptions of Main Theorem, for small $\eps$, two vector fields $v_{\eps}$ and $w_{h(\eps)}$ are topologically equivalent. Due to the definition of moderate topological equivalence, we are given the map $\mathbf H = (h,H_{\eps})$, $h \colon B \to B'$, such that $H_{\eps} \colon U \to H_{\eps}(U)$ conjugates $v_{\eps}$ to $w_{h(\eps)}$ in neighborhoods of large bifurcation supports. We are also given a map $\hat H\colon S^2 \to S^2$ that conjugated $v_0$ to $w_0$ on the whole sphere.

We will not directly extend $H_{\eps}$ to the whole sphere. We will rather prove that two graphs $LMF(v_{\eps})$ and $LMF(w_{h(\eps)})$ are isomorphic for small $\eps$. Then we refer to Theorem \ref{th-graphs-components} together with Theorem \ref{th-Fedorov} and conclude that for small $\eps$, there exists a homeomorphism $\hat H_{\eps} \colon S^2 \to S^2$ that conjugates $v_{\eps}$ to $w_{h(\eps)}$. The family of maps $\mathbf{\hat H} = (h,\hat H_{\eps})$ is a weak topological equivalence of the families $\{v_{\eps}\}$ and $\{w_{h(\eps)}\}$ as required.

Note that we do not guarantee that $(\hat H_{\eps})|_{U} = H_{\eps}$.

The following theorem will imply the Main Theorem:
\begin{theorem}
 Under the assumptions of  Main Theorem, for sufficiently small $\eps$, the graphs  $LMF(v_{\eps})$ and $LMF(w_{h(\eps)})$ are isomorphic as oriented graphs, and the isomorphism $G_{\eps}$ meets the conditions of Theorem \ref{th-graphs-components}.
\end{theorem}
We will construct the isomorphism $G_{\eps}$ on the LMF-graphs as subsets of $S^2$.
Roughly speaking, in order to define $G_{\eps}$, we use $H_{\eps}$ whenever it is defined, i.e. inside a neighborhood $U$ of the large bifurcation support of  $v_{\eps}$.
Outside  $U$, all singular points and cycles of $v_0$ are hyperbolic, thus $v_{\eps}$ has close singular points and cycles. When we define $G_{\eps}$ on singular points and cycles outside $U$, we use $\hat H$ plus  continuation of hyperbolic singular points and cycles with respect to the parameter. The edges of $LMF(V)$ that are partly inside $U$ and partly outside it will be one of our main concerns.

\section{Properties of large bifurcation supports and moderate topological equivalence}

\subsection{Large bifurcation support}
\label{sec-prop-LBS}
In this section, we list the fundamental properties of the set $LBS(V)$ described above. These are the only properties we are going to use in the proofs. 

\begin{enumerate}
\item $LBS(V)$ is a closed $v_0$-invariant set (Proposition \ref{prop-LBS-closed-vinv}).
\item Hyperbolic singular points and hyperbolic limit cycles of $v_0$ do not belong to $LBS(V)$.  All non-hyperbolic singular points and  non-hyperbolic cycles of $v_0$ belong to $LBS(V)$ (Proposition \ref{prop-nohyp-inLBS}).
\item Non-interesting non-hyperbolic cycles of $v_0$ are connected components of $LBS(V)$ (Remark  \ref{rem-nohyp-inLBS}).
\item Sep-property (Proposition \ref{prop-seps-v0}) and Separatrix lemma (Lemma \ref{lem-sep-connect}).
\item No-entrance property (Lemma \ref{lem-seps}).
\item No cycles of mixed location (Proposition \ref{prop-inU-or-hyp}).
\item Moderate topological equivalence of two families $V, W$ implies continuity of conjugacy on $\{\eps=0\}\times \partial LBS(V)$, $\{\eps=0\}\times \partial LBS(W)$ (Proposition \ref{prop-dLBS})
\end{enumerate}

Now we pass to the exact statements.
\subsubsection{$LBS(V)$ is closed and invariant}

\begin{proposition}
\label{prop-LBS-closed-vinv}
	If $V\subset Vect^*\,S^2$, then  both $ELBS(v_0)$ and $LBS(V)$ are closed  $v_0$-invariant sets.
\end{proposition}
\begin{proof}

The set of points with interesting $\alpha$- and $\omega$-limit sets under $v_0$ is $v_0$-invariant. Thus its closure is closed and $v_0$-invariant. So  $ELBS(v_0)$ is closed and $v_0$-invariant.
The set $(\Sing v_0\cup (\overline {\Per V} \cup \overline{\Sep\, V}) \cap \{\eps=0\})$ is closed and $v_0$-invariant. The set $LBS(V)$ is closed and $v_0$-invariant as the intersection of two closed and $v_0$-invariant sets.

\end{proof}

Though the topology of $LBS(V)$ may be complicated, it has finitely many connected components due to the following proposition. 
    \begin{proposition}
\label{lem-LBS-finite}
        If $v\in Vect^*\,S^2$, then each closed $v$-invariant set $A\subset S^2$ has finitely many connected components.

\end{proposition}
In particular, this holds for $A=LBS(V)$.
\begin{proof}
If $A$ is closed and $v$-invariant, then each its connected component is closed and $v$-invariant.
Hence each connected component of $A$ contains trajectories of $v$ together with their $\omega$- and $\alpha$-limit sets.  Due to Poincare-Bendixson theorem, each $\alpha$- and $\omega$-limit set contains either a singular point of $v$, or a cycle of $v$. However, each vector field $v\in Vect^* \, S^2$ has finitely many  singular points and cycles. Each connected component of $A$ contains at least one of them. Thus the number of connected components is finite.
\end{proof}
\subsubsection{$\alpha$- and $\omega$-limit sets in $LBS(V)$}
The next proposition follows immediately from the definition of $LBS(V)$. 
\begin{proposition}
\label{prop-nohyp-inLBS}
 Large bifurcation support $LBS(V)$ does not contain hyperbolic attractors, hyperbolic repellers, or hyperbolic cycles of $v_0$. It contains all non-hyperbolic singular points, non-hyperbolic cycles, and all separatrix connections of $v_0$.
\end{proposition}

Note that due to Poincare-Bendixson theorem, $\alpha$- and $\omega$-limit sets are singular points, limit cycles, and monodromic polycycles. Since monodromic polycycles are formed by separatrix connections, they belong to $LBS(V)$. This implies the following remark.

\begin{remark}
\label{rem-nohyp-inLBS}
     All interesting $\alpha$-, $\omega$-limit sets of $v_0$ except some saddles belong to $LBS(V)$. All non-interesting $\alpha$-, $\omega$-limit sets of $v_0$ except non-hyperbolic non-interesting cycles belong to its complement.
\end{remark}
 We will also need the following proposition.

  \begin{proposition}
  \label{prop-nonint-compon}
   For families $V$ with  $v_0\in Vect^* \, S^2$, non-hyperbolic non-interesting cycles are connected components of $LBS(V)$.
\end{proposition}
\begin{proof}
A neighborhood of a cycle is filled by points whose semi-trajectories (in positive or negative time) wind around this cycle. However a point whose semi-trajectory winds around a non-interesting cycle does not belong to $LBS(V)$, due to the definition of $ELBS(v_0)$. So the intersection of $LBS(V)$ with  a neighborhood of a non-interesting cycle is this cycle only.
\end{proof}
This motivates the following definition.
\begin{definition}
Denote $LBS^*(V) = LBS(V) \setminus \{\text{non-interesting cycles of }v_0\}$.
\end{definition}
The set $LBS^*(V)$ is closed and $v_0$-invariant, due to the previous proposition.

\subsubsection{Sep-property}
\begin{proposition}
\label{prop-seps-v0}
    Suppose that for an unstable  separatrix $\gamma$ of $v_0$, $\omega_{v_0}(\gamma)$ intersects $LBS^*(V)$; equivalently, $\gamma$ hits arbitrarily small neighborhood of $LBS^*(V)$.  Then $\gamma\subset LBS(V)$.

     The same statement holds for stable separatrices and $\alpha$-limit sets.
\end{proposition}
\begin{proof}
Let $\gamma$ be the separatrix mentioned in the lemma. Suppose that it is unstable; the case of stable separatrices is treated in the same way. By the Poincare-Bendixson theorem, the set $\omega_{v_0}(\gamma)$ may be a singular point, a limit cycle or a polycycle. 

Prove that $\omega_{v_0}(\gamma)$ is interesting. Indeed, all polycycles are interesting limit sets, and all singular points and limit cycles in $LBS^*(V)$  are also interesting due to Remark \ref{rem-nohyp-inLBS}. 

Since $\gamma$ is an unstable separatrix, $\alpha_{v_0}(\gamma)$  is a saddle; so it is  interesting by definition.
We conclude that both $\alpha_{v_0}(\gamma)$ and $\omega_{v_0}(\gamma)$  are interesting. Hence $\gamma \subset ELBS(v_0) $. Since $\gamma \subset \Sep (v_0)$, it belongs to $LBS(V)$. \end{proof}

\subsubsection{Separatrix lemma}
Recall that the upper topological limit $\overline \lim A_k$ of a sequence of sets $A_k$ in a topological space is a set of points $x$ such that any neighborhood of $x$ intersects infinitely many of $A_k$; in other words, this is the set of all limit points of the sequence $A_k$.
\begin{lemma}[Separatrix lemma]
\label{lem-sep-connect}
 Let $\gamma_k$  be separatrices of vector fields $v_{\eps_k}$ that connect two interesting singular points, and $\eps_k\to 0$. Then
 $\overline \lim \gamma_k \subset LBS^*(V)$. The same holds for stable separatrices.

 In particular, all separatrix connections of $v_{\eps}$ for small $\eps$ are close to $LBS^*(V)$.
\end{lemma}

\begin{proof}
  Let $x$ belong to $\overline \lim \gamma_k$. Prove that $x\in LBS^*(V)$. Passing to a subsequence, we may and will assume that  $x=\lim_{k\to \infty} x_k$ where $x_k \in \gamma_{{k}}$.

 Consider three cases:

 \textbf{Case 1.} $x\in \Per(v_0)$.

 Assume that $x$ is not in $LBS^*(V)$; then it belongs to a non-interesting (parabolic or hyperbolic) cycle. In both cases, it is easy to see that either $\alpha$- or $\omega$-limit set of a close point $x_k$ under a close vector field $v_{\eps_k}$ is either a cycle that bifurcates from a non-interesting nest, or a non-interesting sink/source inside the nest. Both cases are impossible for separatrices $\gamma_{k}\ni x_k$, and the contradiction shows that $x\in LBS^*(V)$.

 \textbf{Case 2.} $x \notin (\Sing(v_0) \cup \Per(v_0)$).

 Suppose that $x$ has a non-interesting $\alpha$-limit set under $v_0$. Then a close point $x_k$ has a non-interesting $\alpha$-limit set under a close vector field $v_{\eps_k}$, which is impossible for separatrices $\gamma_k\ni x_k$. The contradiction shows that  $\alpha_{v_0}(x)$ is interesting; similarly, $\omega_{v_0}(x)$ is interesting. Since $x\notin \Sing (v_0)$, we conclude that $x\in ELBS(v_0)$. Since $x$ is a limit point of separatrices, $x\in \overline {\Sep V}$, so we have $x\in  LBS(V)$. Since $x$ does not belong to a limit cycle, $x\in LBS^*(V)$.

\textbf{ Case 3.} $x \in \Sing(v_0)$.

 The set $\overline \lim \gamma_k$ is connected as a limit of connected sets. If it coincides with $x$, then separatrices $\gamma_k$ of $v_{\eps_k}$ collapse to $x$, thus $x$ is non-hyperbolic; hence $x\in LBS^*(V)$. If $\overline \lim \gamma_k$ does not coincide with $x$, then arbitrarily close to $x$, there are non-singular limit points of $\gamma_k$. They all belong to $LBS^*(V)$ due to the previous case. Hence $x$ belongs to $LBS^*(V)$, because $LBS^*(V)$ is closed.
  \end{proof}

\subsubsection{No-entrance lemma}
 \begin{definition}
\label{def-entering}
For a vector field $v$, a separatrix $\gamma$ of a singular point $P$ \emph{does not enter} an open set $\Omega\subset S^2$ if one of the following holds:
 \begin{itemize}
  \item $\gamma$ does not intersect $\partial \Omega$;
  \item $P \in \Omega$ and the cross-point $\gamma\cap \partial \Omega$ is unique.
 \end{itemize}

\end{definition}

\begin{lemma}[No-entrance lemma]
        \label{lem-seps}
        In assumptions of the Main Theorem, there exists an arbitrarily small neighborhood $U^*$ of $LBS^*(V)$ such that for sufficiently small $\eps$, no separatrices of $v_{\eps}$ enter $U^*$.
    \end{lemma}
 The proof is postponed till Sec. \ref{ssec-lem-seps}. The statement does \emph{not} hold true for any sufficiently small neighborhood of $LBS^*(V)$. To use this statement, in Sec. \ref{sec-choice-U} we will have to restrict ourselves to special neighborhoods $U$ of $LBS(V)$ instead of all sufficiently small neighborhoods, even though we have moderate equivalence for all small neighborhoods of $LBS(V)$ (see Remark \ref{rem-shrink-U}).  

\subsubsection{No cycles of mixed location}
Each limit cycle of $v_\eps$ either lies in a neighborhood of  $LBS(V)$ or completely outside it. In more detail, we have the following proposition (it also treats singular points, which is analogous but simpler). 

\begin{proposition}
\label{prop-inU-or-hyp}
For any smooth local family $V \subset Vect^*\, S^2$ of vector fields and any small neighborhood $U$ of $LBS(V)$, for sufficiently small $\eps$, each singular point of $v_{\eps}$ is either inside $U$, or belongs to a continuous family $P_{\eps}, \eps \in (B, 0)$, of hyperbolic singular points of $v_{\eps}$ such that $P_0 \notin LBS(V)$.

Each limit cycle of $v_{\eps}$ is either inside $U$, or belongs to a continuous family $c_{\eps}, \eps \in (B,0)$, of hyperbolic limit cycles of $v_{\eps}$ such that $c_0$ does not belong to $LBS(V)$.
\end{proposition}

\begin{proof}
Any singular point $P$ of $v_{\eps}$, $\eps$ small, is close to some singular point $P_0$ of $v_0$. If $P_0\in LBS(V)$, then $P \in U$. If $P_0\notin LBS(V)$, then $P_0$ is hyperbolic (see Proposition  \ref{prop-nohyp-inLBS}), so locally structurally stable. Thus $P$ belongs to a continuous family of singular points of $v_{\eps}$ as required.

The proof for limit cycles is a bit more complicated. The set $\overline {\Per V} \cap \{\eps=0\}$ (a \emph{limit periodic set}) is described by \cite[Theorem 5, Section 2.1.2]{Rous}. This theorem claims that the limit cycles of $v_{\eps}$ as $\eps\to 0$ may accumulate to:
\begin{itemize}
\item a hyperbolic limit cycle of $v_0$;
\item a non-hyperbolic limit cycle of $v_0$;
 \item a non-hyperbolic singular point of $v_0$;
 \item a \emph{polycycle} of $v_0$, namely a finite union of trajectories $\varphi_i$ and singular points $P_i$ of $v_0$, $i=1,\dots,n$ (some of these points may coincide), such that $\alpha(\varphi_i)=P_i, \omega(\varphi_i)=P_{i+1}$, and $\omega(\varphi_n) = P_1$.
\end{itemize}
Note that the polycycle in the last case may be non-monodromic. 
Clearly, the proposition holds in the first case. In the second  and the third cases, it holds true as well: non-hyperbolic singular points and cycles of $v_0$ belong to $LBS(V)$, so the corresponding limit cycles of $v_{\eps}$ belong to $U$ for sufficiently small $\eps$.

Prove that any polycycle of $v_0$ belongs to $LBS(V)$. Indeed, the points $P_i$  are neither sinks nor sources, because some orbits enter $P_i$ and some quit. Hence $P_i$ are interesting limit sets. Thus $\varphi_i$ has interesting $\alpha$-, $\omega$-limit sets, so $\overline \varphi_i \subset ELBS(v_0)$. But  $\overline \varphi_i\subset \overline {\Per V}$, so $\overline \varphi_i \subset LBS(V)$. This completes the proof.

\end{proof}

\subsubsection{Relation to moderate equivalence}
\begin{proposition}
\label{prop-dLBS}
Let two local families $V$ and $W$ be moderately equivalent in some neighborhood of $LBS(V)$, $LBS(W)$ in the sense of Definition \ref{def-moderate-local}. Then the corresponding maps $\mathbf H$ and $\mathbf H^{-1}$ are continuous in $\eps,x$ on the sets $\{\eps=0\}\times  \partial LBS(V)$ and $\{\eps=0\}\times  \partial LBS(W)$ respectively.

 \end{proposition}
\begin{remark}
 In the proof of Main Theorem, we will only use the continuity of $\mathbf H$, $\mathbf H^{-1}$ on the above sets  and on $ \{\eps=0\}\times \Sep (v_0|_{U})$, $ \{\eps=0\}\times \Sep (w_0|_{\hat H(U)})$.
 \end{remark}

\begin{proof}

 Since the boundary of intersection and the boundary of union belong to the union of boundaries,
 $$\partial LBS(V) \, \subset \, \partial ELBS(v_0)\, \bigcup \, \partial \Sing v_0 \cup \partial ((\overline{\Per\, V} \cup \overline{\Sep\, V}) \cap \{\eps=0\}).$$

   Note that $\partial \Sing v_0 = \Sing v_0$. Due to Proposition \ref{prop-dELBS}, $\partial ELBS(v_0)\subset S(v_0)$. Thus $$\partial LBS(V)\, \subset\, S(v_0) \cup \partial ((\overline {\Per V} \cup \overline {\Sep V}) \cap \{\eps=0\})$$
   which is the set \eqref{eq-set1} from the definition of moderate equivalence. The same arguments apply to $W$.
  This completes the proof.
\end{proof}

\subsection{Moderate topological equivalence}
This section contains simple topological statements that follow from the definition of moderate equivalence. 

The following proposition enables us to work in small neighborhoods of $LBS(V)$ and  $LBS(W)$.

\begin{proposition}
\label{prop-heps-detached}
 Under assumptions of the Main Theorem, for each neighborhood $\tilde U^+$ of  $LBS(W)$, there exists a small open neighborhood $U$ of $LBS(V)$, such that $H_{\eps}(U) \subset \tilde U^+$ for all small $\eps$.
\end{proposition}

Since hyperbolic sinks, sources and cycles of $w_0$ are outside $LBS(W)$ (Proposition \ref{prop-nohyp-inLBS}), this Proposition immediately implies the corollary:
\begin{corollary}
\label{cor-heps-detached}
 For small $U\supset LBS(V) $ and any small $\eps$, the set $H_{\eps}(U)$ is detached from all hyperbolic attracting and repelling singular points and cycles  of $w_0$.
\end{corollary}

\begin{proof}[Proof of Proposition \ref{prop-heps-detached}]
Suppose that for some neighborhood $\tilde U^+ \supset LBS(W)$, the statement does not hold true. So there exists a sequence of shrinking open neighborhoods $U_n$: $\cap U_n =LBS(V)$, and a sequence $\eps_n \to 0$, such that none of the domains $H_{\eps_n}(U_n)$ are contained in $\tilde U^+$. Then there exists a sequence $x_n \in U_n$ such that $H_{\eps_n}(x_n)\notin \tilde U^+$.

Due to Requirement \ref{it-nbhd-cond}  of Definition \ref{def-moderate-local}  of moderate equivalence (applied to $\mathbf H^{-1}$), the set $H_{\eps_n}^{-1}(\tilde U^+)$ for small $\eps_n$ is a neighborhood of $LBS(V)$, in particular, $LBS(V) \subset H_{\eps_n}^{-1}(\tilde U^+)$. So $H_{\eps_n}(LBS(V))\subset \tilde U^+ $,   hence $x_n \notin LBS(V)$.

Since sphere is compact, we can choose a convergent subsequence $x_n'$ of $x_n$. 
Since $U_n$ shrink to the closed set $LBS(V)$, this subsequence converges to a point of $\partial LBS(V)$. This contradicts the continuity of the map  $\mathbf H$ on $\{\eps=0\}\times \partial LBS(V)$ (see Proposition \ref{prop-dLBS}).
\end{proof}

The next proposition shows that under assumptions of the Main theorem, the conjugacy of families $V$ and $W$ respects connected components of $LBS(V)$, $LBS(W)$ and their neighborhoods. 

\begin{proposition}
\label{prop-U12-intersect}
In assumptions of Main Theorem, let $U, \tilde U^+$ be the same as in  Proposition \ref{prop-heps-detached}. Suppose that each connected component of $\tilde U^+$ contains one, and only one, connected component of $LBS(W)$. Suppose that each connected component of $U$ contains a connected component of $LBS(V)$. Then for sufficiently small $U$,  for small $\eps_1,\eps_2$ and for two different connected components $U_1, U_2$ of $U$, their images $H_{\eps_1}(U_1)$ and $H_{\eps_2}(U_2)$ do not intersect and belong to different connected components of $\tilde U^+$.
\end{proposition}
\begin{proof}
Recall that  $H_{\eps}(U)\subset \tilde U^+$ for all small $\eps$ as in Proposition \ref{prop-heps-detached}.
 Let $C_i$ be a connected component of $LBS(V)$ that belongs to $U_i$. By assumption, connected components $\hat H(C_1)$, $\hat H(C_2)$ of $LBS(W)$ are in different connected components of $\tilde U^+$; let $\tilde U_1, \tilde U_2$ be these connected components of $\tilde U^+$.

 Since $\hat H(C_1)\subset \tilde U_1$, the whole component $\hat H(U_1)$ belongs to  $\tilde U_1$. Due to Proposition \ref{prop-dLBS}, $\mathbf H$ is continuous on $\partial C_1$. So   for small $\eps_1$, $H_{\eps_1}(\partial C)$ is close to $H_0(\partial C) = \hat H(\partial C)\subset \tilde U_1$. Hence $H_{\eps_1}(U_1)\supset H_{\eps_1}(\partial C_1)$ intersects $\hat H(U_1)$; we conclude that  $H_\eps(U_1)\subset \tilde U_1$ for all small $\eps$. Similarly, $H_{\eps}(U_2) \subset \tilde U_2$ for all small $\eps$, which finishes the proof.
 
 \end{proof}

\section{Combinatorial equivalence of the LMF graphs: first steps of the construction}

\subsection{General approach}
Recall that  a neighborhood $U$ of $LBS(V)$ is fixed in the statement of the Main Theorem and does not depend on $\eps$. However we will have to shrink $U$ further in the proof; this is possible due to Remark \ref{rem-shrink-U}.

Our goal is to define an isomorphism
$$G_{\eps} \colon LMF(v_{\eps})\to  LMF(w_{h(\eps)})$$
of the $LMF$ graphs   of $v_\eps $  and $w_{h(\eps )}$, and to prove that it meets the conditions of Theorem \ref{th-graphs-components}. For each element $c$ (i.e. vertex or edge) of $LMF(v_\eps)$ that belongs to $U$ we will define
$$
G_\eps |_c := H_\eps |_c.
$$
It turns out that the elements  of $LMF(v_\eps)$ that lie outside $U$ depend continuously on $\eps$ (say, all singular points outside $U$ are hyperbolic, thus structurally stable).
When we define $G_{\eps}$ on these parts of $LMF(v_{\eps})$, we use $\hat H$ plus continuous continuation in $\eps$ of hyperbolic singular points and cycles.

For a smooth local family $V\subset Vect^*\,S^2$ of vector fields, each hyperbolic singular point $P_0$, each hyperbolic limit cycle $c_0$ and each germ of a separatrix of a hyperbolic saddle $(\gamma_0, P_0)$ of a vector field $v_0$ generates a continuous family of hyperbolic points, cycles and germs of separatrices $P_{\eps}, c_{\eps}, (\gamma_{\eps}, P_{\eps})$ respectively of $v_{\eps}$. This enables us to give the following definition.

\begin{definition}[Notation]
\label{def-pi}
In the above assumptions, let $\pi_{\eps}$ be a continuous map that depends continuously on $\eps$, $\eps$ small,  and takes homeomorphically $P_{\eps}$ to $P_0$, $c_{\eps}$ to $c_0$ and $(\gamma_{\eps}, P_{\eps})$ to $(\gamma_{0}, P_0)$, preserving time orientation.

In assumptions of Main Theorem, let $\tilde \pi_{\eps}$ be the analogous map for the family $W$ playing the role of $V$.
\end{definition}

\subsection{Partial definition of $G_{\eps}$}
\label{sec-partial-Ge}
\subsubsection{Singular points, limit cycles, and germs of separatrices}
\label{ssec-Ge-skeleton}
Denote by $\Sep^*v_{\eps}$ the set of germs of separatrices of $v_\eps $ at singular points of $v_{\eps}$. Note that a separatrix connection of $v_{\eps}$ corresponds to two germs in $\Sep^*  v_{\eps}$. Let us define $G_\eps $ on $\Sing  v_{\eps} \cup \Per  v_{\eps} \cup \Sep^*v_{\eps} =: S^*(v_{\eps})$.

Let $p\in S^*(v_{\eps})$. If $p \notin U$, then $p$ is a hyperbolic singular point, or belongs to a cycle, or belongs to a germ of a separatrix of a hyperbolic saddle. We define
$$
G_\eps (p) := {\tilde \pi_{h(\eps)} }^{-1} \circ \hat H \circ \pi_\eps (p).
$$
(see Definition \ref{def-pi} for the definition of $\pi_{\eps}, \tilde \pi_{\eps}$). If $p \in U$, we define
$$
G_\eps (p) := H_\eps (p).
$$
This completes the definition of $G_\eps $ on $S^*(v_{\eps})$. Note that we have constructed $G_\eps $ on vertices of type $1,3$ (singular points and points on limit cycles) and edges of type $3$ (limit cycles), because these vertices and edges belong to $S^*(v_{\eps})$.

\begin{remark}
\label{rem-Geps-toptype}
$G_{\eps}$ preserves topological types of singular points and limit cycles, thus preserves labels on the vertices of type $1,3$ and edges of type $3$. $G_{\eps}$ also preserves the time orientation on cycles and germs of separatrices of $v_{\eps}$.
\end{remark}

\subsubsection{Elliptic sectors}
\label{ssec-Ge-elliptic}
Each elliptic sector $E$  of $v_\eps $ corresponds to a non-hyperbolic singular point $P$ of $v_\eps$; we have $P\in U$ by Proposition \ref{prop-inU-or-hyp}. Consider a germ of the elliptic sector $(E,P)$; its image under $H_{\eps}$ is a  germ of an elliptic sector of $w_{h(\eps)}$. Denote it by  $\tilde E:=H_\eps ((E,P))$. 
In the construction of $LMF(v_{\eps})$, we need to choose a trajectory $l$ of $v_{\eps}$ in $E$. We may and will assume that $l$ is close to $P$ so that $l \subset U$. In the elliptic sector $\tilde E$ of $w_{h(\eps)}$, we need to choose a homoclinic trajectory of $w_{h(\eps)}$; let us choose $H_{\eps}(l)$. Then we define
$$
G_\eps |_l := H_\eps |_l.
$$
We have constructed $G_\eps $ on  edges of type $5$ (homoclinic trajectories in elliptic sectors).
\begin{remark}
\label{rem-Ge-elliptic}
 $G_{\eps}$ preserves time orientation on edges of type 5 (homoclinic trajectories in elliptic sectors), and preserves incidence of edges of type 5 and vertices of type 1.
\end{remark}

\subsubsection{Non-truncated separatrices}
\label{ssec-Ge-nontrunc}
Let $\gamma $ be a non-truncated separatrix of $v_{\eps}$. Then for small $\eps $, it belongs to $U$, due to Separatrix Lemma \ref{lem-sep-connect} above. Define
$$
G_\eps |_\gamma := H_\eps |_\gamma.
$$
This agrees with the definition of $G_{\eps}$ on the germs of separatrices. We have constructed $G_\eps $ on  edges of type $1$ (non-truncated separatrices).

 \begin{remark}
 \label{rem-nontrunc-label}
$G_{\eps}$ preserves incidence of vertices of type 1 and edges of type 1.   For non-truncated separatrices (edges of type 1), $G_{\eps}$ preserves labels. Indeed, the labels say that the separatrix is stable, unstable or a separatrix connection. Since $G_{\eps}$ is induced by a homeomorphism $H_{\eps}$ in a neighborhood of such edge, it preserves such labels.
 \end{remark}

\begin{remark}
Now $G_{\eps}$ is defined on all  monodromic polycycles of $v_{\eps}$, because they  are formed by non-truncated separatrices. Hence $G_{\eps}$ is defined on all possible $\alpha$- and $\omega$-limit sets of $v_{\eps}$ (singular points, limit cycles and monodromic polycycles) of $v_{\eps}$.
 \end{remark}

\subsubsection{The graph correspondence $G_{\eps}$ is a bijection}
\label{ssec-Ge-bij}
 By now, we have constructed $G_{\eps}$ for small $\eps$ on all vertices and edges of the $LMF(V)$ disjoint from the transversal loops. Let us prove that this map is one-to-one.
\begin{proposition}
 \label{prop-G-bij}
 For small $U\supset LBS(V)$, for sufficiently small $\eps$, the map  $G_{\eps}$ defined in Sec. \ref{ssec-Ge-skeleton}, Sec. \ref{ssec-Ge-elliptic}, and Sec. \ref{ssec-Ge-nontrunc} is one-to-one on singular points, limit cycles, non-truncated separatrices,  and trajectories in elliptic sectors (vertices of types $1,3$, edges of type $1,3,5$) of $v_{\eps}$, $w_{h(\eps)}$. It preserves incidence of these vertices and edges, labels and time orientation.
 \end{proposition}
\begin{proof}
Fix $U\supset LBS(V)$ such that for all small $\eps$, $H_{\eps}(U)$ is detached from hyperbolic cycles and singular points of $w_0$. This is possible due to Corollary \ref{cor-heps-detached} above.

Note that $G_{\eps}$ is defined on all edges and vertices of $LMF(v_{\eps})$ listed in the proposition. Each of the listed edges is either completely inside $U$, or completely outside it. In $U$,  $G_{\eps}$ is induced by $H_{\eps}$,  so is injective.
Outside $U$, the map $G_{\eps} = \tilde \pi_{h(\eps)}^{-1} \circ \hat H \circ \pi_{\eps}$ is a composition of three injective maps, so is injective as well.

For all listed vertices and edges that are inside $U$, their images under $G_{\eps}$ are located inside $H_{\eps}(U)$; for the edges and vertices outside $U$, their images are close to hyperbolic singular points and cycles of $w_0$. Due to Corollary \ref{cor-heps-detached}, $H_{\eps}(U)$ is detached from these hyperbolic singular points and cycles of $w_0$, thus $G_{\eps}$ is injective for small $\eps$.

Prove that $G_{\eps}$ is surjective. Recall that $\tilde U := \cap_{|\eps| \le \eps_0} H_{\eps}(U)$ is a neighborhood of  $LBS(W)$, due to Requirement \ref{it-nbhd-cond} of Definition \ref{def-moderate-local} of moderate equivalence.

For small $\eps$, each non-truncated separatrix of $w_{h(\eps)}$ belongs to $\tilde U$ (Separatrix lemma \ref{lem-sep-connect}). Thus it belongs to the range of $H_{\eps}$. So $G_{\eps}$ is surjective on truncated separatrices. Each elliptic sector of $w_{h(\eps)}$ is an elliptic sector of a non-hyperbolic singular  point, and all such points belong to $\tilde U \subset H_{\eps}(U)$. So the edge of type 5 in  this sector is the image under $G_{\eps}=H_{\eps}$ of the edge of type 5 of $LMF(v_{\eps})$. Thus $G_{\eps}$ is surjective on edges of type 5.

Each singular point and each limit cycle of $w_{h(\eps)}$ is  either completely inside $\tilde U$, or completely outside it (Proposition \ref{prop-inU-or-hyp} applied to $W$).  In the first case, this singular point (cycle) is  the image of some cycle or singular point of $(v_\eps)|_{U}$ under $H_{\eps}=G_{\eps}$. In the second case, it belongs to a continuous family of singular points (cycles) of $(w_\eps)|_{S^2\setminus U}$, thus belongs to the range of $G_{\eps} = \tilde \pi_{h(\eps)}^{-1} \circ \hat H \circ \pi_{\eps}$.

So $G_{\eps}$ is surjective on the union of verteces and edges of $v_\eps$ disjoint from the transversal loops.

This map preserves incidence of vertices and edges, labels and time orientation due to Remarks \ref{rem-Geps-toptype}, \ref{rem-Ge-elliptic}, \ref{rem-nontrunc-label} above.
\end{proof}

\subsubsection{Transversal loops}
\label{ssec-tr}
Consider a hyperbolic sink, a source  or a limit cycle of $v_0$; in all the  three cases, we denote this object by $P$. Let $l$ be a transversal loop around $P$. We assume that $U$ is sufficiently small so that it does not intersect $l$. Since $P$ is hyperbolic, it
persists under small perturbations, so $l$ is a transversal loop for the corresponding object $P_{\eps}:=\pi^{-1}_{\eps} P$ of $v_{\eps}$, $\eps$ small. We may and will assume that $l$ belongs to the graph  $LMF(v_{\eps})$ as a transversal loop for $P_{\eps}$. Now,  $\tilde l = \hat H(l)$ is a transversal loop of  $\hat H(P)$ for $w_0$. Moreover, $\tilde l$ is a transversal loop of the corresponding object  $\tilde \pi^{-1}_{\eps} (\hat H(P)) =  G_\eps (P_{\eps})$ of $w_{h(\eps)}$. We may and will assume that $\tilde l$  belongs to the  graph $LMF(w_{h(\eps )})$ as a transversal loop of $G_{\eps}(P_{\eps})$.
 Define
$$
G_\eps |_l := \hat H|_l.
$$
If a cycle, polycycle, a sink or a source $P$ of $v_{\eps}$ belongs to $U$, we choose its transversal loop $l$ so that $l \subset U$ and the annulus between $P$ and $l$ belongs to $U$. Then $H_{\eps}(l)$ is a transversal loop of $H_{\eps}(P)$ for $w_{h(\eps)}$, so we may assume that $H_{\eps}(l)$ belongs to $LMF(w_{h(\eps)})$.
We define
$$
G_\eps|_{l} := H_\eps|_{l}.
$$
Note that all limit cycles and singular points are either inside or outside $U$ by Proposition \ref{prop-inU-or-hyp}. All monodromic polycycles are inside $U$ due to Separatrix lemma \ref{lem-sep-connect}. So we have already constructed $G_\eps$ on all transversal loops that belong to $LMF(v_{\eps})$.

Note that we did not yet define $G_\eps$ on the truncation verteces of  $LMF(v_\eps)$.
In Sec. \ref{ssec-Ge-final}, we will have to modify $G_{\eps}$ on transversal loops so that it provides a correct identification of truncation vertices. However we will not change the set $G_{\eps}(l)$ for a transversal loop $l$.

\begin{remark}
\label{rem-Ge-trans}
$G_{\eps}$ preserves the correspondence of $\alpha$-, $\omega$-limit sets and their transversal loops: if $l$ is a transversal loop of a cycle, polycycle or  singular point $P$ of $v_{\eps}$, then $G_{\eps}(l)$ is a transversal loop of $G_{\eps}(P)$.

This implies that $G_{\eps}$ is one-to-one on transversal loops of $LMF(v_{\eps})$, $LMF(w_{h_{\eps}})$, because it is one-to-one on limit cycles, polycycles and singular points of $v_{\eps}, w_{h(\eps)}$.
\end{remark}

\section{Main lemmas and the proof of Main Theorem}
\label{sec-proof}

In this section, we formulate two main lemmas and prove the Main Theorem modulo these lemmas.
\subsection{The plan of the proof}

In the previous section, we have partially constructed the required isomorphism $G_\eps\colon LMF(v_{\eps}) \to LMF(w_{h(\eps)}) $. However $G_{\eps}$ is not yet defined on truncated separatrices; it is only defined on their germs at singular points. To complete the construction, we will need the following Correspondence lemma: if separatrices of $v_{\eps}$ cross a transversal loop of $v_{\eps}$, then the corresponding separatrices of $w_{h(\eps)}$ cross the corresponding transversal loop of $w_{h(\eps)}$; the formal statement appears below. This lemma will enable us to extend $G_{\eps}$ to truncated separatrices and truncation vertices, and we will be forced to modify restrictions of $G_{\eps}$ to transversal loops so that it provides a correct identification of truncation vertices. However, for a transversal loop $l$, we will not change its image  $G_{\eps}(l)$.

After $G_{\eps}$ is constructed, we have to verify the assumptions of Theorem \ref{th-graphs-components} for it. The most non-trivial assumption concerns annular faces of LMF graphs. We will state and prove the Annuli faces lemma to handle this problem.

\subsection{Main lemmas}

   \begin{lemma}[Annuli faces lemma]
   \label{lem-an-faces}
    In assumptions of Main theorem, let $\eps$ be sufficiently small. Let $A$ be an annular face of $LMF(v_{\eps})$. Then the map $G_{\eps}$ (see Sec. \ref{sec-partial-Ge}) takes $\partial A$ homeomorphically to the boundary of an annular face $\tilde A$ of $LMF(w_{h(\eps)})$. Moreover, $G_{\eps}$ extends to an orientation-preserving homeomorphism that takes $A$ to $\tilde A$.
   \end{lemma}

The proof constitutes Section \ref{sec-annuli}. Clearly, the modification of $G_{\eps}$ on transversal loops will not affect this lemma.
\vskip 0.5 cm

Now we introduce notation for Correspondence lemma.
Recall that we always choose counterclockwise orientation on transversal loops with respect to their $\alpha$-($\omega$-)limit sets, see Sec. \ref{sec-LMF}. We will call this ''proper orientation``.

Take a transversal loop $l$ that belongs to $LMF(v_{\eps})$.  Denote by $\{\gamma_i\}, i=1,\dots, n$, all separatrices of singular points $P_i$ of $v_{\eps}$ that cross $l$ (the case $n=0$ is not excluded). We suppose that the cross-points $p_i := \gamma_i \cap l$, i.e. truncation vertices on $l$,  are ordered cyclically along $l$. Note that if a singular point $P$ has several separatrices that intersect $l$, then it appears several times in the list  $\{P_i\}$.

Let $\tilde \gamma_i$ be the separatrix of $w_{h(\eps)}$ that corresponds to $\gamma_i$, i.e. contains the germ $G_{\eps}((\gamma_i, P_i))$.

 \begin{lemma}[Correspondence lemma]
    \label{lem-seps-behave}
     In assumptions of Main Theorem,  let $l$ be a  properly oriented transversal loop of $v_{\eps}$ that belongs to $LMF(v_{\eps})$. Let $\{\gamma_i\}$ be all separatrices of $v_{\eps} $ that intersect $l$, so that the corresponding truncation vertices $p_i$ are ordered cyclically along $l$.
    Then for sufficiently small $\eps$,

    1) The corresponding separatrices  $\{\tilde \gamma_i\}, i=1, \dots, n$ of $w_{h(\eps)}$  cross  the properly oriented transversal  loop $\tilde l:=G_{\eps}(l)$, and the truncation vertices $\tilde p_i := \tilde \gamma_i \cap \tilde l$ are ordered cyclically along $\tilde l$.

    2) There are no more truncation vertices on $\tilde l$.
   \end{lemma}
  The proof constitutes Section \ref{sec-Corr}. The proof is simple if we only consider separatrices that are completely inside $U$ or comlpetely inside  its complement.  The problem occurs if the separatix has \emph{mixed location}: belongs partly to  $U$ and partly to its complement. We will classify such separatrices as well as the boundary components of $U$ that may intersect them.

\subsection{Isomorphism of LMF graphs}
\label{ssec-Ge-final}
This section completes the construction of the graph isomorphism $G_{\eps} \colon LMF(v_{\eps}) \to LMF(w_{h(\eps)}) $.

 Let $l$ be a transversal loop in $LMF(v_{\eps})$. We will introduce $\gamma_i, p_i, \tilde l,$ $\tilde \gamma_i$, and $\tilde p_i$ as in Correspondence lemma \ref{lem-seps-behave}. Note that $p_i$ are all truncation vertices on $l$ and $\tilde p_i$ are all truncation vertices on $\tilde l$.

Now we modify $G_{\eps}|_{l}$ so that $G_{\eps}(l) = \tilde l$, the map takes $p_i$ to $\tilde p_i$ and preserves counterclockwise orienation on $l, \tilde l$.
If $l$ is empty, i.e. does not intersect separatrices of $v_{\eps}$, it contains one vertex of type 4. Correspondence lemma implies that $\tilde l$ does not intersect separatrices of $w_{h(\eps)}$, so it also contains one vertex of type 4. In this case, we modify $G_{\eps}|_l$ so that it matches these vertices of type 4 and preserves counterclockwise orienation on $l, \tilde l$.

$G_{\eps}$ takes the germs $(\gamma_i, P_i)$ of truncated separatrices to $(\tilde \gamma_i, G_{\eps}( P_i))$, due to the definition of $\tilde \gamma_i$. We extend $G_{\eps}$ to the whole truncated separatrices so that it identifies truncation vertices: $G_{\eps} (p_i) =\tilde p_i$.
This completes the construction of $G_{\eps}$ on vertices of type 2, 4 and edges of type 2, 4.

\begin{remark}
\label{rem-G-24}
 $G_{\eps}$ is one-to-one on vertices and edges of type $2$, $4$, due to Remark \ref{rem-Ge-trans} and Correspondence lemma. It preserves incidence of vertices and edges, labels, and orientation on transversal loops.

Each truncated separatrix of $w_{h(\eps)}$ terminates on some transversal loop, so all of them are in the range of $G_{\eps}$. This shows that $G_{\eps} $ is surjective on truncated separatrices. Injectivity is clear because $G_{\eps}$ is injective on germs of truncated separatrices. We conclude that $G_{\eps}$ is one-to-one on truncated separatrices.
\end{remark}

\vskip 0.2 cm

Proposition \ref{prop-G-bij} and Remark \ref{rem-G-24} show that the map  $G_{\eps}$ is one-to-one on vertices and edges of $LMF(v_{\eps})$ and $LMF(w_{h(\eps)})$, preserves incidence of vertices and edges, labels and orientation. So $G_{\eps}$ is a graph isomorphism.

\subsection{Isotopy of the LMF graphs}

We are going to check the assumptions of Theorem \ref{th-graphs-components}: $G_{\eps}$ preserves orders of edges in all vertices of $LMF(v_\eps)$ and extends to annuli-shaped faces of $LMF(v_{\eps})$, $LMF(w_{h(\eps)})$. The second statement follows directly from Annuli faces Lemma \ref{lem-an-faces}. Now we check the first statement in all vertices of $LMF(v_{\eps})$.

\textbf{Vertices of type 1 (singular points) inside $U$.}

The edges that start at such vertex $P$ are
\begin{itemize}
 \item edges of type 1, i.e. non-truncated separatrices of $v_{\eps}$. They belong to $U$ due to Separatrix lemma \ref{lem-sep-connect}.
\item edges of type 5, i.e. homoclinic trajectories in elliptic sectors of $P$.
\item edges of type 2, i.e. truncated separatrices of that singular point $P$.
 \end{itemize}
On the edges of types 1,5, and on the germs of edges of type 2,  $G_{\eps}$ coincides with $H_{\eps}$.
    Hence it preserves cyclical  orders of edges at all vertices of type 1 inside $U$, because so does $H_{\eps}$.

\textbf{Vertices of type 1 outside $U$}

Note that vertices of type 1 outside $U$ are either hyperbolic saddles, or hyperbolic sinks, or sources.
Hyperbolic sinks and sources are isolated vertices of $LMF(v_{\eps})$, and there is nothing to prove for them.  Let $P$ be a hyperbolic saddle outside $U$; on the (germs of) edges adjacent to $P$, we have  $G_{\eps}((\gamma, P)) = \tilde \pi_{h(\eps)}^{-1} \circ \hat H \circ \pi_{\eps} ((\gamma, P))$. All maps in this composition preserve orders of separatrices at hyperbolic saddles, so $G_{\eps}$ preserves order of edges at $P$.

\textbf{Vertices of type 2: truncation vertices}

  Such vertex has degree 3: one of the corresponding edges is of type 2 (a truncated separatrix), and two other edges are of type 4 (two arcs of a transversal loop, or possibly one arc coinciding with the whole loop).
    The order of edges in such vertex is always such that the truncated separatrix is ''from the right-hand side`` with respect to the orientation along the transversal loop. Indeed, the transversal loop is oriented counterclockwise with respect to its $\alpha$- or $\omega$-limit set (see the definition of LMF graphs), and the separatrix crosses it from the other side. So the order of edges in this vertex is determined by the orientation of the edges of the graph, and $G_{\eps}$ preserves this orientation.

     \textbf{Vertices of type $3$ and type $4$ (points on limit cycles and on empty transversal loops)}

Such vertices are only joined to themselves by edges of types 3 and 4 respectively. So the cyclical order in such vertices is trivial, and there is nothing to prove.

       Main Theorem is now proved modulo main lemmas.

   \section{Auxiliary lemmas}

The proofs of both main lemmas, as well as the proof of No-entrance  lemma (see Lemma \ref{lem-seps}), are heavily based on the following Boundary lemma.
 
 \subsection{Boundary lemma}
\label{ssec-Blemma}  
In this section, we formulate the  Boundary lemma. Its proof is postponed till  Section \ref{sec-U}.

\begin{definition}
Let $U$ be an open domain. A  point $p\in \partial U$ is called an \emph{inner} topological tangency point of $\partial U$ with $v$ if a germ of the trajectory of $p$ under $v$ is inside $\overline U$ and only crosses $\partial U $ at $p$. It is called an \emph{outer} topological tangency point of $\partial U$ with $v$ if this germ is outside $U$ and only crosses $\partial U $ at $p$.
\end{definition}

 Note that if $\partial U$ is smooth and only has isolated quadratic tangencies with $v$, then all these  tangencies are either inner or outer topological  tangency points.

\begin{definition}
\label{def-Sep-pr}
Let $v$ be a smooth vector field. A closed $v$-invariant set $Z\subset S^2$ is said to have a Sep-property if:
\begin{itemize}
\item \emph{For any unstable separatrix $\gamma \not\subset Z$, the set $\omega(\gamma)$ is detached from $Z$;}

\item \emph{For any stable separatrix $\gamma \not\subset Z$, the set $\alpha(\gamma)$ is detached from $Z$.}
\end{itemize}
 \end{definition}

   \begin{figure}[h]
   \begin{center}
  \includegraphics[width=0.3\textwidth]{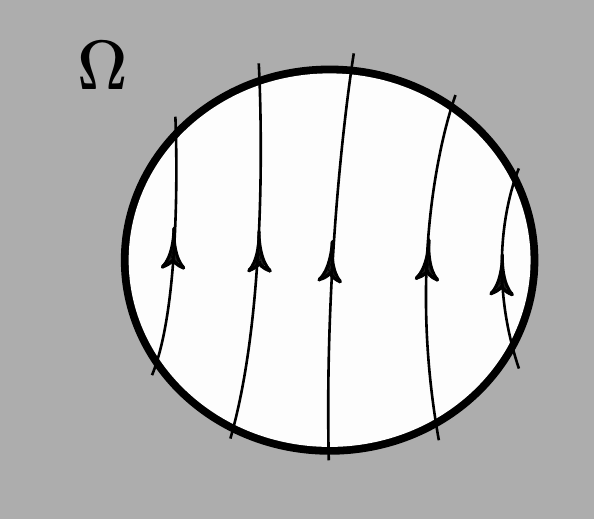}
\hfil
  \includegraphics[width=0.3\textwidth]{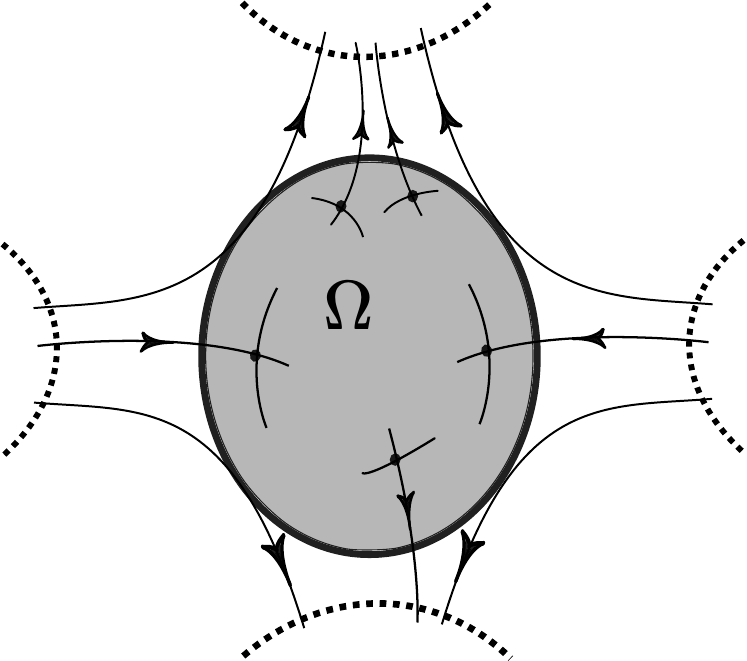}
\hfil
  \includegraphics[width=0.25\textwidth]{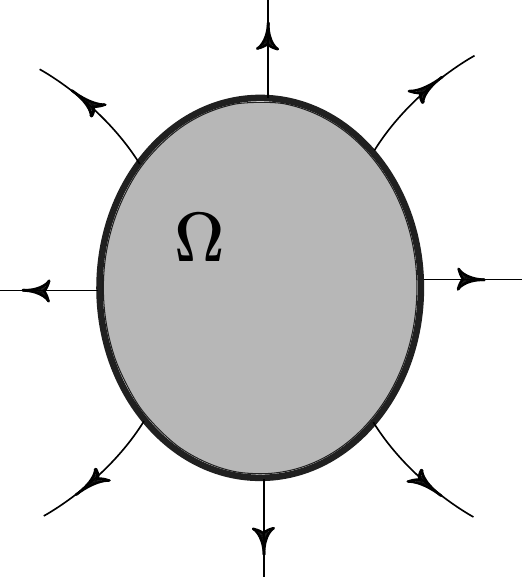}
  \end{center}
 \caption{Type $1$, Type $2$, and Type $3$ boundary components of $\Omega$. For Type $2$, inside the dotted transversal loops there are $\alpha$- and $\omega$-limit sets of maximal transversal arcs of this boundary component. }\label{fig-Case123-def}
\end{figure}

\begin{lemma}[Boundary lemma (see Fig. \ref{fig-Case123-def})]
\label{lem-U-arcs-top}
Let  $v \in Vect^*\,S^2$  be a vector field; let $Z\subset S^2$ be a closed non-empty  $v$-invariant set with Sep-property.

Then there exists an arbitrarily small neighborhood $\Omega$ of $Z$ with smooth boundary and finitely many boundary components with isolated quadratic tangencies of $\partial \Omega$ with $v$, such that any connected component  $\varphi$ of its boundary  $\partial \Omega$ is of one of the following types:
    \begin{itemize}
        \item Type $1$: $\varphi$ contains two inner tangency points of $\partial \Omega$ with $v$ and bounds a disc  $D\subset (S^2\setminus \overline \Omega)$; $v|_{D}$ is orbitally topologically equivalent to the vector field $\partial / \partial x$ in the unit disc.  Trajectories of points of $\varphi$ under $v|_{\overline \Omega}$ belong to $\Omega$.

        \item Type $2$: $\varphi$ contains only outer tangency points of $\partial \Omega$ with $v$ (probably no tangency points). The trajectories of points of $\varphi$ under $v|_{S^2\setminus \Omega}$ belong to $S^2\setminus \overline\Omega$.

         Each maximal transversal arc $\beta \subset  \varphi$ intersects a separatrix of $v|_{\Omega}$.
         If $\beta$ is outgoing, then all its points have a common $\omega$-limit set under $v$.  If $\beta$ is ingoing, then all its points have a common $\alpha$-limit set under $v$. This $\omega$- (resp. $\alpha$-) limit set lies outside $\Omega$.

        \item Type $3$: $\varphi$ is a transversal loop of some $\alpha$- or $\omega$-limit set (an attracting or repelling singular point, a cycle or a polycycle) of $v$, and
        this object belongs to $Z$.

    \end{itemize}
 Moreover, separatrices of $v|_{S^2\setminus \Omega}$ do not intersect $\Omega$ in all the three cases above.
    \end{lemma}

\begin{figure}[h]
 \begin{center}
  \includegraphics[width=0.9\textwidth]{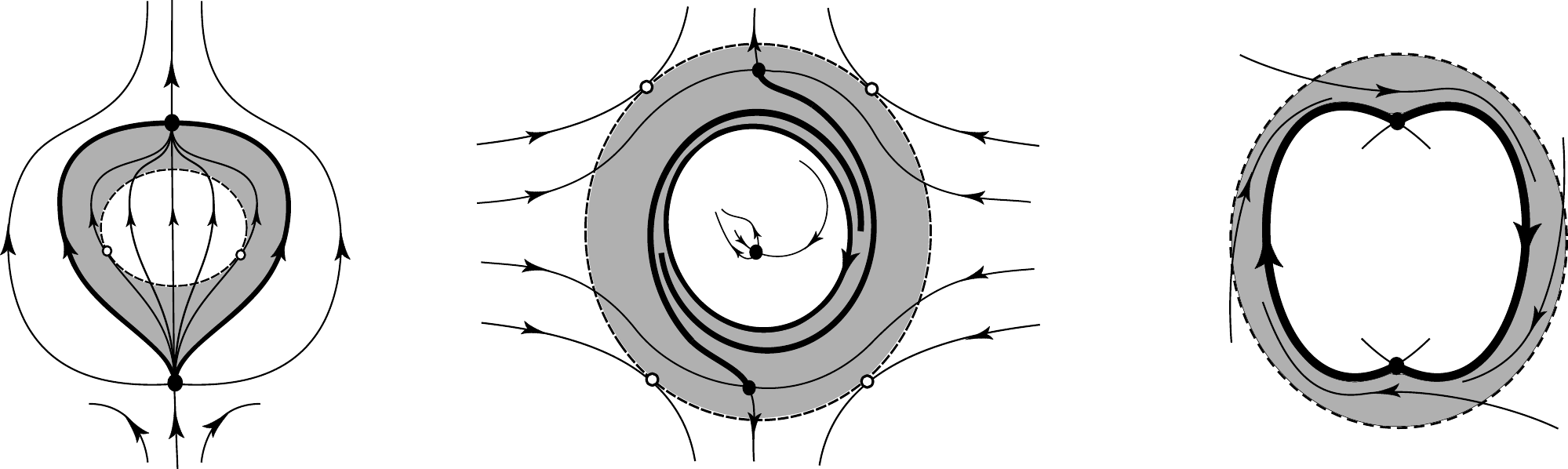}
 \end{center}
\caption{Examples of Type $1$, Type $2$, and Type $3$ boundary components of $\Omega \supset Z$; the set $Z$ is shown in thick, and boundary components are dashed. On all of these pictures, $Z$ is a part of the large bifurcation support of a generic unfolding of $v$}\label{fig-ex-123}
\end{figure}

Note that the characteristic feature of the Type $2$ boundary components is not the presence of outer tangencies (they may be absent), but rather the presence of separatrices that cross these components.

\begin{remark}[Indices of boundary components of $\Omega$]
\label{rem-index}
Recall that we orient $\partial \Omega$ counterclockwise with respect to $\Omega$. For a boundary component $\varphi$  of $\Omega$, suppose that the point $\infty$ of the sphere is located to the left of $\varphi$ (i.e. on the same side as $\Omega$). Then the indices of boundary components of $\Omega$ with respect to $v$ are the following:
\begin{itemize}
\item Type $1$: index $0$;
\item Type $2$: indices $1, 2,\dots$;
\item Type $3$: index $1$. 
\end{itemize}
We will use this remark later to distinguish between boundary components of different types. 
\end{remark}

\begin{proposition}
\label{prop-Case2-nonint-ls}
There exists an arbitrarily small open neighborhood $U^*\supset LBS^*(V)$ that satisfies assumptions of Boundary lemma for $Z=LBS^*(V)$ and $v=v_0$. Moreover, for boundary components of Type $2$, the common $\alpha$- (resp. $\omega$-) limit set under $v_0$ of each ingoing (resp. outgoing) transversal subarc $\beta \subset \partial U^*$ is non-interesting.
\end{proposition}

\begin{proof}
Let us check that $LBS^*(V)$ satisfies assumptions of Boundary lemma. Clearly, it is closed and $v_0$-invariant (Proposition \ref{prop-LBS-closed-vinv}). It  has a Sep-property due to Proposition \ref{prop-seps-v0}. The application of Boundary lemma provides us with a neighborhood $U^*$. This implies the first statement of the proposition. 

Let $\varphi \subset U^*$ be a boundary component of $U^*$ of Type $2$, and let $\beta \subset \varphi$ be its outgoing transversal subarc. Boundary lemma implies that  the common  $\omega$-limit set of $\beta $ under $v_0$ is outside $U^*$. The only interesting  $\omega$-limit sets of $v_0$ outside $LBS^*(V)$ are saddles (see Remark \ref{rem-nohyp-inLBS}). But saddle separatrices of $v|_{S^2\setminus U^*}$ do not intersect $U^*$ due to Sep-property of $LBS(V)$. So $\omega(\beta)$ is non-interesting. The same arguments apply to ingoing transversal arcs and their $\alpha$-limit sets.
\end{proof}

\subsection{Proof of No-entrance  lemma \ref{lem-seps} modulo Boundary lemma}
\label{ssec-lem-seps}
\begin{remark}  \label{rem-Bound-imply-Sep}
In fact, we will prove that any neighborhood $U^*$ that satisfies Boundary lemma (for $v_0$ and $LBS^*(V)$) also satisfies No-entrance lemma.
\end{remark}

\begin{proof}
Choose a neighborhood $U^*$ that satisfies Boundary lemma for $v_0$ and $LBS^*(V)$. It exists due to Proposition \ref{prop-Case2-nonint-ls}. Suppose that unstable separatrices of $v_{\eps}$ enter  $U^*$ for arbitrarily small $\eps$. Then there exists a sequence $\eps_k\to 0$ and points  $p_k \in \partial U^*$ where unstable separatrices of $v_{\eps_k}$ enter  $U^*$. Let $p$ be a limit point of the sequence $p_k \in \partial U^*$. Then  $p \in \partial U^*$; clearly, $p$ belongs to the closure of an ingoing transversal arc of $\partial U^*$.

On the other hand, $p\in (\overline {\Sep V}) \cap \{\eps=0\}$. We claim that $p\in ELBS(v_0)$; this will imply $p\in LBS(V)$ and contradict $p \in \partial U^*$.

Since $p\in  \partial U^*$, it is not singular. Prove that its $\alpha$-limit set under $v_0$ is interesting. Indeed, otherwise the negative semi-trajectory of $p$ under $v_0$ crosses a transversal arc of a non-interesting set $\alpha_{v_0}(p)$. Thus for close points $p_k$, their negative semi-trajectories under close vector fields $v_{\eps_k}$ cross this arc as well. Hence these semi-trajectories cannot be unstable saddle separatrices, and we get a contradiction.

So $\alpha_{v_0}(p)$ is interesting. Let us prove that $\omega_{v_0}(p)$ is interesting.

 Since  $\alpha_{v_0}(p)$ is interesting, Proposition \ref{prop-Case2-nonint-ls} implies that $p$ cannot belong to the boundary component of Type $2$.
      So it belongs to the boundary component of Type $1$ or Type $3$. In both cases, the future semi-trajectory of $p$ under $v_0$ belongs to $U^*$  due to Boundary Lemma \ref{lem-U-arcs-top}.
        The $\omega$-limit set of $p$ under $v_0$ is thus inside $\overline U^*$; due to Remark \ref{rem-nohyp-inLBS}, $\omega_{v_0}(p)$ is interesting.

Hence $\alpha$- and $\omega$-limit sets of $p$ are both interesting. We conclude that $p\in ELBS(v_0) \cap (\overline \Sep V) \cap \{\eps=0\} \subset LBS(V)$, which contradicts $p\in \partial U^*$.
So separatrices of $v_{\eps}$ cannot enter $U^*$.
\end{proof}

\subsection{Choice of $U$}
\label{sec-choice-U}

If $U$ is a neighborhood of $LBS(V)$, we denote by  $U^{*} $ the union of its connected components that do not contain non-interesting cycles of $v_0$.
If $\tilde U^{\pm}$ is a neighborhood of $LBS(W)$, we denote by  $\tilde U^{\pm *} $ the union of its connected components that do not contain non-interesting cycles of $w_0$.

\begin{proposition}
\label{prop-choice}
Under the assumptions of the Main Theorem, there exists an arbitrarily small open  neighborhood $U$ of $LBS(V)$ and arbitrarily small open neighborhoods $\tilde U^{\pm}$ of $LBS(W)$ such that
\begin{itemize}
 \item [-] $U^*$ satisfies Boundary lemma for $v_0$ and  $LBS^*(V)$;
 \item [-] $\tilde U^{\pm *} $ satisfy Boundary lemma for $w_0$ and $LBS^*(W)$;
 \item [-] for all small $\eps$, $\tilde U^- \subset H_{\eps}(U) \subset \tilde U^{+}$;
 \item [-] the sets $U\setminus U^*$, $\tilde U^{\pm} \setminus \tilde U^{\pm*}$ are unions of neighborhoods of non-interesting cycles bounded by their transversal loops.
\item [-] each connected component of $U, \tilde U^\pm$ contains one connected component of $LBS(V), LBS(W)$ respectively.
\end{itemize}

\end{proposition}
Note that this proposition implies that $U, \tilde U^{\pm}$ satisfy the assertions of No-entrance  lemma, see Remark \ref{rem-Bound-imply-Sep}. The last assertion of this proposition shows that Proposition \ref{prop-U12-intersect} on connected components is applicable for $U, \tilde U^{+}$.
\begin{proof}
Choose a neighborhood $\tilde U^{*+}$ of $LBS^*(W)$ satisfying Boundary lemma for the vector field $w_{0}$. We may remove its connected components that do not contain connected components of $LBS^*(W)$; if it is sufficiently small, then each its connected component contains only one component of $LBS^*(W)$.   Add small annular neighborhoods of non-interesting cycles of $w_0$ bounded by transversal loops; we get the required neighborhood $\tilde U^+$ of $LBS(W)$.

Now, take a small neighborhood $U^*$ of $LBS^*(V) $ that satisfies the assumptions of Boundary lemma, and add small annular neighborhoods of non-interesting cycles of $v_0$ bounded by transversal loops.  We get a neighborhood $U\supset LBS(V)$. Due to Proposition \ref{prop-heps-detached}, we may and will assume that  for small $\eps$, $H_{\eps}(U)\subset \tilde U^+$. As above, we assume that each connected component of $U$ contains one connected component of $LBS(V)$.

Recall that for any small $\eps_0$, $\cap_{|\eps|<\eps_0} H_{\eps}(U)$ is a neighborhood of $LBS(W)$ due to the definition of moderate equivalence (Requirement \ref{it-nbhd-cond} of Definition \ref{def-moderate-local}).
We choose $\tilde U^{-*}\supset LBS^*(W)$ that satisfies Boundary lemma for $w_0$, and add small annular neighborhoods of non-interesting cycles of $w_0$ bounded by transversal loops, in order to get a neighborhood $\tilde U^-$ of $LBS(W)$. We assume that $\tilde U^{-}$ is sufficiently small so that $\tilde U^{-} \subset (\cap_{|\eps|<\eps_0} H_{\eps}(U))$. Once again, we assume that each connected component of $\tilde U^-$ contains one connected component of $LBS(W)$.

Finally, $\tilde U^- \subset H_{\eps}(U) \subset \tilde U^+$ for small $\eps$ as required.

\end{proof}

From now on, we assume that $U$, $\tilde U^+$ and $\tilde U^-$ satisfy the proposition above.

\subsection{Images of Type $1$, $2$, and $3$ boundary components}
In the proofs of both main theorems, we will also need results on the images of Type $1$, Type $2$, and Type $3$ boundary components of $U^*$ under $H_{\eps}$. In some sence, they say that the boundary component $H_{\eps}(\varphi)$ of $H_{\eps}(U)$ has similar properties to that of $\varphi$, and also provide some control on the location of $H_{\eps}(\varphi)$ for different $\eps$.

For the three subsequent lemmas, $U$, $\tilde U^{\pm}$ are as in Proposition \ref{prop-choice} and are sufficiently small, i.e. belong to some preassigned neighborhoods of the corresponding large bifurcation supports. From now on, we assume that $H_{\eps}$ extends homeomorphically to $\overline U$, otherwise we slightly diminish $U$.

\begin{lemma}[Images of Type $1$ boundary components]
\label{lem-image-case1}
Under the assumptions of Main Theorem, suppose that $\varphi$ is a Type $1$ boundary component of $\partial U$.
Then $H_{\eps}(\varphi)$ bounds an open topological disc $D \subset S^2\setminus H_{\eps}(U)$, and $w_{h(\eps)}$
has no singular points and limit cycles in $D$.

\end{lemma}

The following lemma is important for Correspondence lemma: it shows that $H_{\eps}$ preserves the correspondence between outgoing (ingoing) transversal arcs of Type $2$ boundary components and their $\omega$- (resp. $\alpha$-)limit sets outside $U^*$.

Here and below the orientation on $\partial U$ is clockwise with respect to $U$.

\begin{lemma}[Images of Type $2$ boundary components]
\label{lem-image-Case2}
 Under the assumptions of Main Theorem, suppose that $\beta$ is a transversal outgoing arc of a Type $2$ boundary component $\varphi\subset \partial U^*$.
 Let $\beta_{\eps}\subset \varphi$ be the maximal arc transversal to $v_{\eps}$ and close to $\beta$.
Put $\tilde \beta_\eps:= H_{\eps}(\beta_\eps) $.

Let $l$ be a transversal loop around $\omega(\beta)$. Put $\tilde l = \hat H(l)$. Then for small $\eps$,
positive semi-trajectories of points of $\tilde \beta_\eps$ under $w_{h(\eps)}$ stay in $S^2\setminus H_{\eps}(\overline {U^*})$, and the Poincare map $\tilde P_{\eps} \colon \tilde \beta_\eps \to \tilde l$ along $w_{h(\eps)}$ is well-defined. The map $\tilde P_{\eps}$ takes the clockwise orientation on $\tilde \beta_\eps$ with respect to $H_{\eps}(U)$ to the counterclockwise orientation on $\tilde l$ with respect to $\hat H (\omega (\beta))$.

Moreover, $\tilde P_{\eps}(\tilde \beta_\eps)\subset \tilde l$ intersects $\tilde P_0(\tilde \beta_0)$ for small $\eps$. The analogous statement holds for ingoing arcs.
\end{lemma}

\begin{corollary}
\label{cor-image-transvers-Case2}
Let $\varphi$ be an outgoing \textbf{transversal} Type $2$ boundary component of $\partial U$, let $l, \tilde l$ be as in Lemma \ref{lem-image-Case2}.  Then for all small $\eps$, there are no singular points and limit cycles of $w_{h(\eps)}$ between two closed curves  $H_{\eps}(\varphi)$ and $\tilde l$.
\end{corollary}

\begin{proof}
This follows from Lemma \ref{lem-image-Case2} above, since the Poincare map $\tilde P_{\eps} \colon H_{\eps}(\varphi) \to \tilde l$ along the orbits of $w_{h(\eps)}$ is well-defined.
\end{proof}

\begin{lemma}[Images of Type  3 boundary components]
\label{lem-image-case3}
Under the assumptions of Main Theorem, let $\varphi\subset \partial U$ be  a Type $3$ boundary component of $\partial U$. Then for all small $\eps$, the oriented curves $H_{\eps}(\varphi)$ and $\hat H(\varphi)$ are homotopic  in $\tilde U^+ \setminus (LBS(W) \cup \Sing w_{h(\eps)} \cup \Per w_{h(\eps)})$.
\end{lemma}

The proofs of these lemmas is postponed till Section \ref{sec-Case123}.

\subsection{Logical relation between subsequent sections}

We now turn to the proof of the main lemmas. The logical relation between sections \ref{sec-annuli} - \ref{sec-Case123} is the following: \ref{sec-U} $\rightarrow$ \ref{sec-Case123} $\rightarrow$ \ref{sec-annuli}  $ \rightarrow$ \ref{sec-Corr}. Yet we start with main lemmas: Annuli faces lemma and Correspondence lemma in Sections \ref{sec-annuli} and \ref{sec-Corr} respectively, making use of the Boundary lemma and Lemmas \ref{lem-image-case1}, \ref{lem-image-Case2}, \ref{lem-image-case3}. Then we prove Boundary lemma and these lemmas in Sections \ref{sec-U}  and \ref{sec-Case123} respectively.

\section{Proof of the Annuli faces lemma}  \label{sec-annuli}

\subsection{Empty annuli lemma}

\begin{definition}\label{def:emptyan}
We say that the annulus $A \subset S^2$ is \emph{empty} with respect to a vector field $v$ if its boundaries are topologically transversal to $v$ and there are no singular points or limit cycles of $v$ inside $A$.

In this case, $v|_A$ is orbitally topologically equivalent to the radial vector field $\partial / \partial r$ in the standard annulus $\{1<r<2\}$.
\end{definition}

Let us now define a collection $L$ of transversal loops around non-interesting nests of $v_0$ (see Definition \ref{def-nic} of non-interesting cycles). This collection will be used in the proof of the Correspondence lemma.

For a non-interesting nest of $v_0$, let us order the cycles by inclusion. The first and the last cycles are called \emph{boundary cycles} of the nest. In the case when we have a nest of non-interesting semi-stable cycles (case \ref{it-semist} in Definition \ref{def-nic}), we suppose that the hyperbolic singular point mentioned in this definition lies inside the \emph{inner} cycle of the nest. This enables us to distinguish the inner and the outer cycle of the nest. 

Transversal loops of the limit cycles of the nest belong to the LMF graph of the vector field $v_0$, but it is possible that they do not belong to the LMF-graph of $v_{\eps}$ because the limit cycles of the nest may destroy. However it is convenient to consider $LMF(v_{\eps})$ together with the transversal loops that encircle the nest. 

\begin{definition}[Collection $L$ of transversal loops] \label{def-tr-loops}
For each non-interesting and \emph{not semi-stable} nest of $v_0$ (case \ref{it-nonsemist} of Definition \ref{def-nic}), fix two transversal loops $l^-, l^+$ of the boundary cycles of the nest such that the whole nest is in the annulus between $l^-, l^+$.

For each non-interesting \emph{semi-stable} nest of $v_0$ (case \ref{it-semist} in Definition \ref{def-nic}), fix an outer transversal loop $l$ of the most outer cycle of the nest, such that $l$  encircles the whole nest.

We orient these loops counterclockwise with respect to the annulus between $l^-, l^+$ or with respect to the disc encircled by $l$ respectively. Let $L$ be a collection of transversal loops thus obtained. We say that $l\in L$ is ingoing if future semi-trajectories of its points enter the corresponding nest, and outgoing otherwise.
\end{definition}

The Annuli faces lemma follows from a more general statement, Empty annuli lemma, which we will also need below in the proof of Correspondence lemma (see Section \ref{sec-Corr}).

\begin{lemma}[Empty annuli lemma]
\label{lem-annuliShaped}
Under the assumptions of Main Theorem, for sufficiently small open $U\supset LBS(V)$, suppose that transversal loops $l_1,l_2$ bound an empty  annulus $A$ for a vector field $v_{\eps}$. Suppose that $l_i$ is either a transversal loop around a hyperbolic singular point or a cycle of $v_\eps$,  or  $l_i\subset \overline U$, or $l_i\in L$.
Let  $\tilde l_i := H_{\eps}(l_i)$ if $l_i\subset \overline U$ and $\tilde l_i := \hat H(l_i)$ in other cases.

Then $\tilde l_1, \tilde l_2$ bound an empty annulus $\tilde A$ for $w_{h(\eps)}$.

Moreover, let the orientation on $\tilde l_i$ be induced by $\hat H$ or $H_{\eps}$ from the orientation on $l_i$.
Then $l_1,l_2$ are oriented with respect to $A$ in the same way as $\tilde l_1, \tilde l_2$ are oriented with respect to $\tilde A$.
\end{lemma}
The last assertion implies that $\hat H$, $H_{\eps}$ restricted to $\partial A$ extend to the homeomorphism of $A, \tilde A$. Note that transversal loops from the $LBS(V)$  either surround a hyperbolic singular point or a cycle of $v_\eps$,  or belong to $ \overline U$. The case $l_i\in L$ will be used in the proof of Correspondence lemma in  Section \ref{sec-Corr} below.

\subsection{Reduction}

\begin{proof}[Proof of the Annuli faces lemma modulo Empty annuli lemma]
Let $A$ be the same as in the Annuli faces lemma, that is, an annuli shaped face of the LMF graph of $v_{\eps }$. Due to the classification of faces of LMF graphs (Lemma \ref{lem-faces}), we have two cases. Consider them one by one.
\begin{itemize}
 \item $A$ is an annulus between a transversal loop $l$  of $v_{\eps}$ and the corresponding $\alpha$- or $\omega$-limit set $c$.
\end{itemize}
Due to Remark \ref{rem-Ge-trans}, $G_{\eps}$ preserves the correspondence of transversal loops and their $\alpha$-, $\omega$-limit sets, so the loop $G_{\eps}(l)$ is a transversal loop for $G_{\eps}(c)$ in $LMF(w_{h(\eps)})$. Thus  $G_{\eps}(l)$ and $G_{\eps}(c)$ bound an annulus $\tilde A$. It remains to prove that $G_{\eps}$ extends to a homeomorphism of $A,\tilde A$, i.e. to analyze whether it preserves orientation on $\partial A$. We have two subcases:

1) If $c\subset U$, then $A$ is inside $U$, due to the choice of transversal loops in Section \ref{ssec-tr}. Then $G_{\eps}$ is induced by $H_{\eps}$ on $\partial A$. So $H_{\eps}$ provides a required extension of $G_{\eps}$ to this annuli-shaped face.

2) If $c$ is outside $U$, then $c $ is either a hyperbolic cycle, or a hyperbolic sink, or a source. We only consider the case when $c$ is a cycle; other cases are analogous but simpler. 
 
Recall that the orientation on $c$ and its transversal loop $l$ is chosen in such a way that $c$ is to the left with respect to $l$; suppose that $l$ is to the left with respect to the timewise orientation of $c$. It remains to prove that the mutual orientation of $\tilde l:= G_{\eps}(l)=\hat H(l)$ and $ \tilde c:= G_{\eps}(c) = \tilde \pi_{h(\eps)}\hat H(\pi_{\eps}^{-1}(c))$ is the same as the orientation of $c,l$ described above. This will imply that $G_{\eps}$ matches the orientations on $\partial A, \partial \tilde A$; thus  $G_{\eps}$ extends to a homeomorphism between the faces $A$ and $\tilde A$.

Indeed, $\tilde c$ is to the left with respect to $\tilde l$ due to the choice of orientation on transversal loops.
Further, $\tilde l = \hat H(l)$ is to the left with respect to $\hat H(c) $ because $\hat H$ is an orientation-preserving homeomorphism. The curve $\hat H(c)$ is close to the cycle $G_{\eps}(c)=\tilde \pi_{h(\eps)}\hat H(\pi_{\eps}^{-1}(c))$ which implies the statement.

\begin{itemize}
  \item    $A$ is an annulus between two transversal loops $l_1,l_2$ of $v_{\eps}$.
\end{itemize}

In this case, Lemma \ref{lem-an-faces} follows from the Empty annuli Lemma \ref{lem-annuliShaped}.

Note that $A$ is an empty annulus of $v_{\eps }$ in the sense of Definition \ref{def:emptyan}. By Empty annuli Lemma \ref{lem-annuliShaped}, the annulus $\tilde A$ between $\tilde l_1 \mbox{ and } \tilde l_2$ is empty for $w_{h(\eps )}$. By construction of $G_\eps ,$  its boundaries $ \ \tilde l_1, \tilde l_2$ belong to $LMF(w_{h(\eps )})$.

Let $l_1$ be an outgoing transversal loop, and let $l_2$ be ingoing; clearly, no unstable separatrices may cross an outgoing  transversal loop $\tilde l_1$ of a cycle, source or a monodromic polycycle. Similarly, no stable separatrices may cross $\tilde l_2$. So no separatrices enter an empty annulus $\tilde A$, thus it forms a face of $LMF(w_{h(\eps)})$.

By Lemma \ref{lem-annuliShaped}, the map $G_{\eps }\colon \partial A \to \partial \tilde A$ may be extended to a homeomorphism between $A$ and $\tilde A$.
\end{proof}

\subsection{Plan of the proof of the Empty annuli lemma \ref{lem-annuliShaped}}

The natural way to prove the Empty Annuli lemma is to compare restrictions to $A$ of the phase portraits of $v_{\eps }$ and $v_0$. The first restriction is trivial; the second one may be quite different, see Figures \ref{fig:onecompu} and \ref{fig:onecompcu}.  We will need the Boundary Lemma for the case shown in Fig. \ref{fig:onecompu}, and both Boundary and Correspondence Lemmas for Fig. \ref{fig:onecompcu}.

Let us pass to the formal proof.

Consider all boundary components of $U$ that are inside $A$. It is possible that some of them are non-contractible inside $A$; then $A$ is split into several smaller annuli. Note that all these boundary components have index $1$ with respect to $v_{\eps}$. Hence they  have index $1$ with respect to $v_0$. Therefore, they  are transversal boundary components for $U$ and $v_0$ (see Boundary Lemma \ref{lem-U-arcs-top}). Clearly, each smaller annulus is an empty annulus for $v_\eps$. We are going to prove the Empty annuli lemma  for each of these  smaller annuli.

Any  smaller annulus  does not contain non-contractive boundary components of $U$. So there are two possible cases: the boundary components $l_1$, $l_2$ of a smaller annulus $A$ belong to the same connected component of $\overline U$ or to the same connected component of $CU=S^2\setminus U$.

Indeed, suppose that $l_1$ and $l_2$ do not belong to the same connected component of $U$. The annulus bounded by $l_1$ and $l_2 $ does not contain non-contractive boundary components of $U$, hence the curves $l_1$ and $l_2$ are not separated by $U$. Therefore, they belong to the same connented component of $CU$.

The two possible cases mentioned above are considered below in Lemmas \ref{lem-same-U} and \ref{lem-S2-U}, so these lemmas conclude the proof. See Figures \ref{fig:onecompu} and \ref{fig:onecompcu} respectively.

\begin{figure}[h]
\begin{center}
\includegraphics[width=0.4\textwidth]{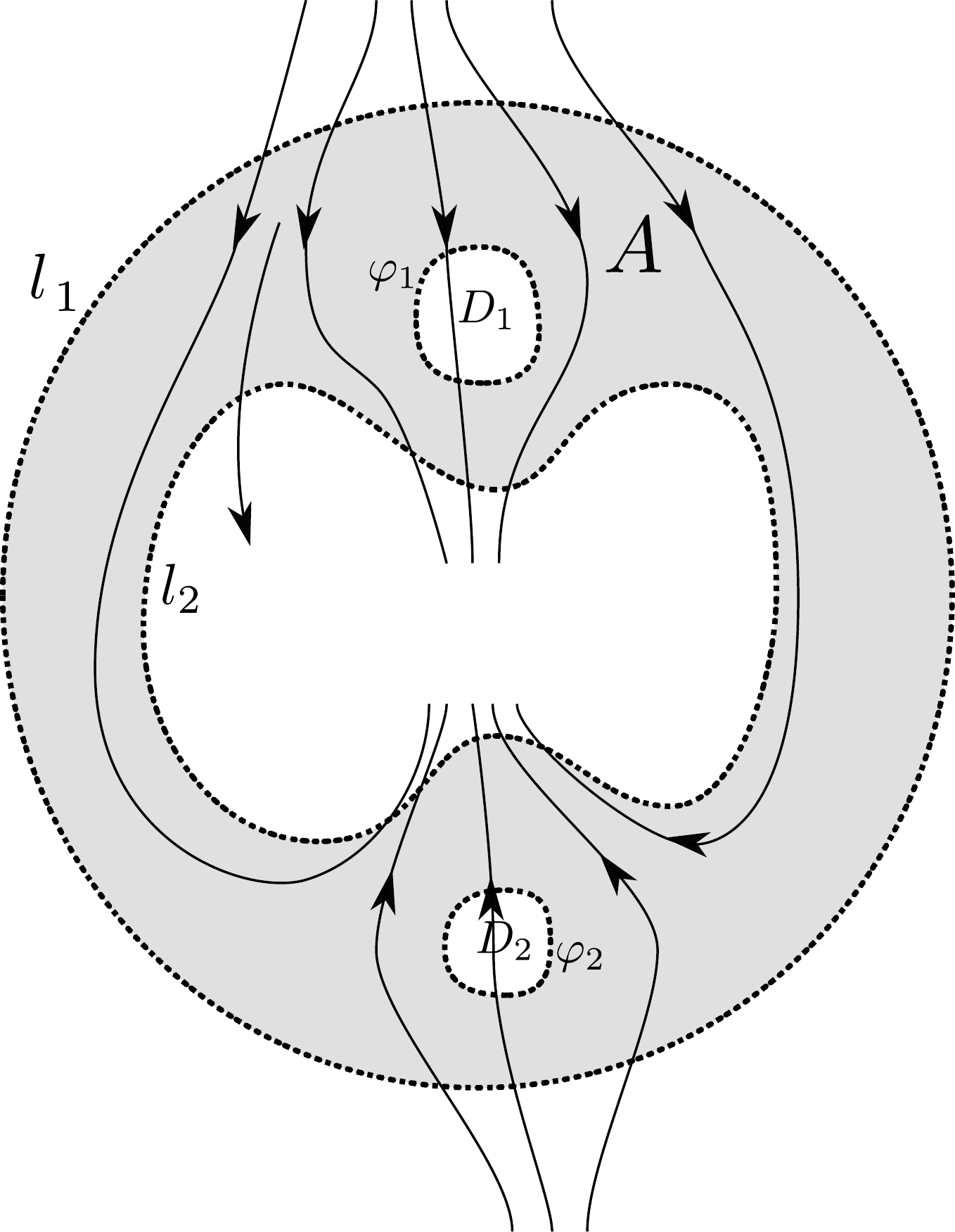}
\hfil
\includegraphics[width=0.4\textwidth]{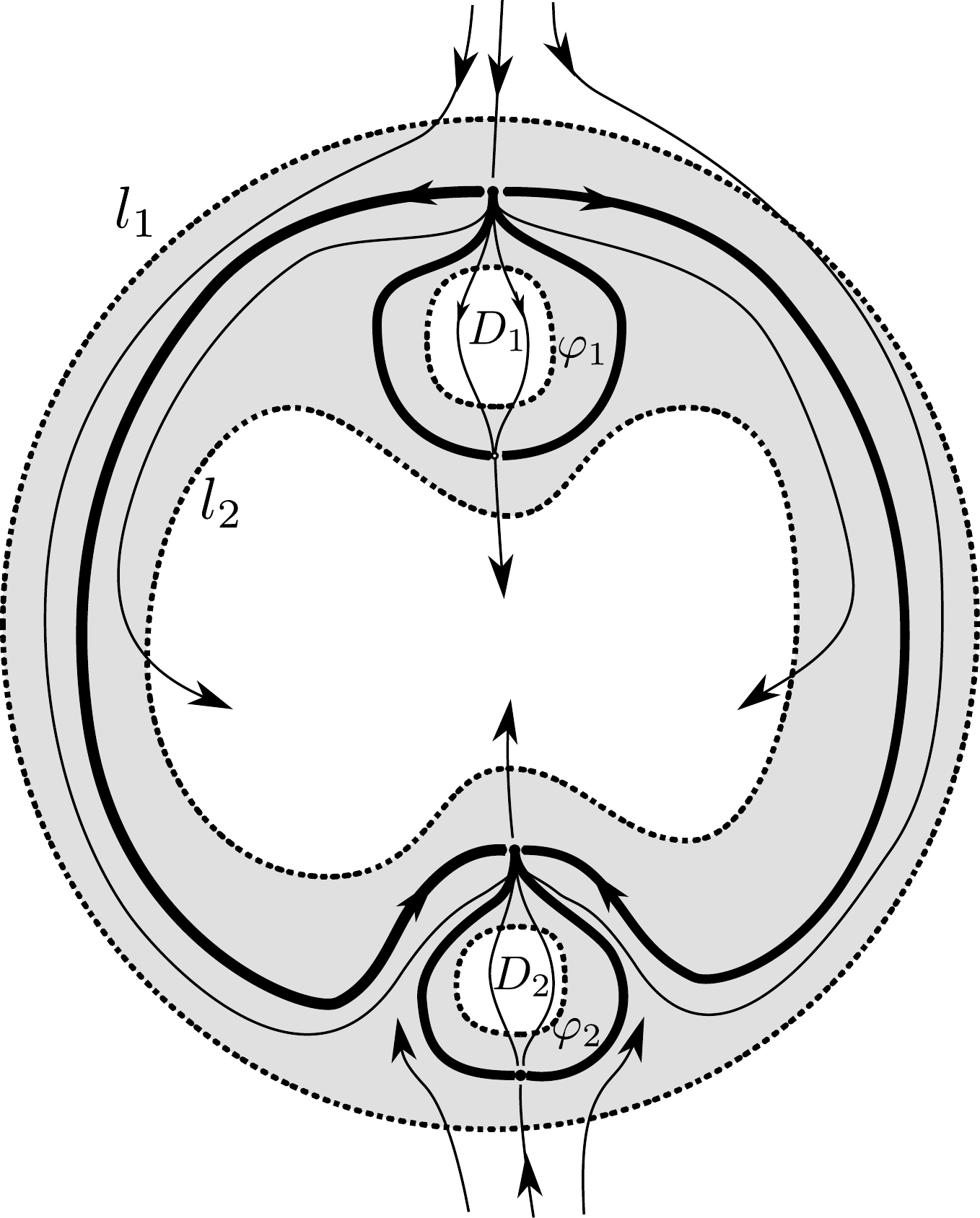}
\caption{Empty annuli lemma: auxiliary lemma \ref{lem-same-U}, orbits of $v_\eps$ (left) and $v_0$ (right). The component of $U$ is shadowed, its boundary is dotted, the large bifurcation support is shown in thick. The vector field $v_0$ has degeneracy of codimension 6; four saddlenodes of $v_0$ vanish as $\eps$ changes, and the annulus $A$ becomes empty. }\label{fig:onecompu}
\end{center}
\end{figure}

\subsection{One connected component of $\bar U$}

\begin{lemma}
\label{lem-same-U}
The statement of Empty annuli lemma holds true if $l_1$, $l_2$ belong to the same connected component of $\overline U$.
\end{lemma}
\begin{proof}
Consider all the boundary components $\varphi_i$ of this connected component of $\overline U$ that are located inside $A$. They are contractive in $A$, so they bound topological discs in $A$. These disks
contain no singular points of $v_{\eps}$, so the index of the vector field $v_\eps$ with respect to each curve $\varphi_i$ is $0$ (we assume that the point $\infty$ on $S^2$ is outside $A$). The same holds for the vector field $v_0$. Thus these components $\varphi_i$ are of Type $1$ (see Remark \ref{rem-index}), they bound discs $D_i \subset CU$, and $A \setminus \cup D_i \subset U$.
So the annulus $\tilde A$ between $H_{\eps}(l_1)$ and $H_{\eps}(l_2)$ is a union of $H_{\eps} (A \setminus \cup D_i)$ and regions $\tilde D_i$ bounded by $H_{\eps}(\varphi_i)$, where $\varphi_i$ are Type $1$ boundary components for $v_0$, see Fig. \ref{fig:onecompu} left. Fig. \ref{fig:onecompu} right shows an example of a vector field with such boundary components of $\overline U$. 

Let us prove that $\tilde A$ is empty for $w_{h(\eps )}$.
Due to Lemma \ref{lem-image-case1}  on the images of Type $1$ components, the regions $\tilde D_i$ inside  $H_{\eps} (\varphi_i)$ do not contain limit cycles and singular points of $w_{h(\eps)}$. The set  $H_{\eps} (A \setminus \cup D_i)$ does not contain singular points and limit cycles of $w_{h(\eps)}$ too, because its preimage under $H_{\eps}$ does not contain  singular points and limit cycles of $v_{\eps}$.

The case when a limit cycle belongs partly to $H_{\eps} (A \setminus \cup D_i)$ and partly to its complement in $\tilde A$ is prohibited by Proposition \ref{prop-inU-or-hyp}: each cycle of $w_{h(\eps)}$ for small $\eps$ either belongs to $\tilde U^- \subset H_{\eps}(U)$, or is close to a hyperbolic cycle of $w_0$ (thus does not intersect $\tilde U^+\supset H_{\eps}(U)$). Thus $\tilde A$ contains no singular points or limit cycles of  $w_{h(\eps)}$.

Let us prove that $\partial \tilde A$ is transversal to  $w_{h(\eps)}$.  Indeed, $\partial  A$ is transversal to $v_0$, thus to $v_{\eps}$; $H_{\eps}$ preserves topological transversality, thus $\partial \tilde A$ is transversal to  $w_{h(\eps)}$.

So $\tilde A$ is empty in this case. Clearly, the map $H_{\eps}|_{\varphi_i}$ may be extended to a homeomorphism $H_i: D_i \to \tilde D_i$. Hence, the map $H_{\eps}|_{\partial A}$ may be extended to a homeomorphism of $A $ by $H_\eps$ on  $A\setminus \cup D_i$ and by $H_i$ on $D_i$.  Thus the last claim of the Empty annuli lemma holds true.
\end{proof}

\subsection{One connected component of $CU$}

\begin{figure}[h]
\begin{center}
\includegraphics[width=0.8\textwidth]{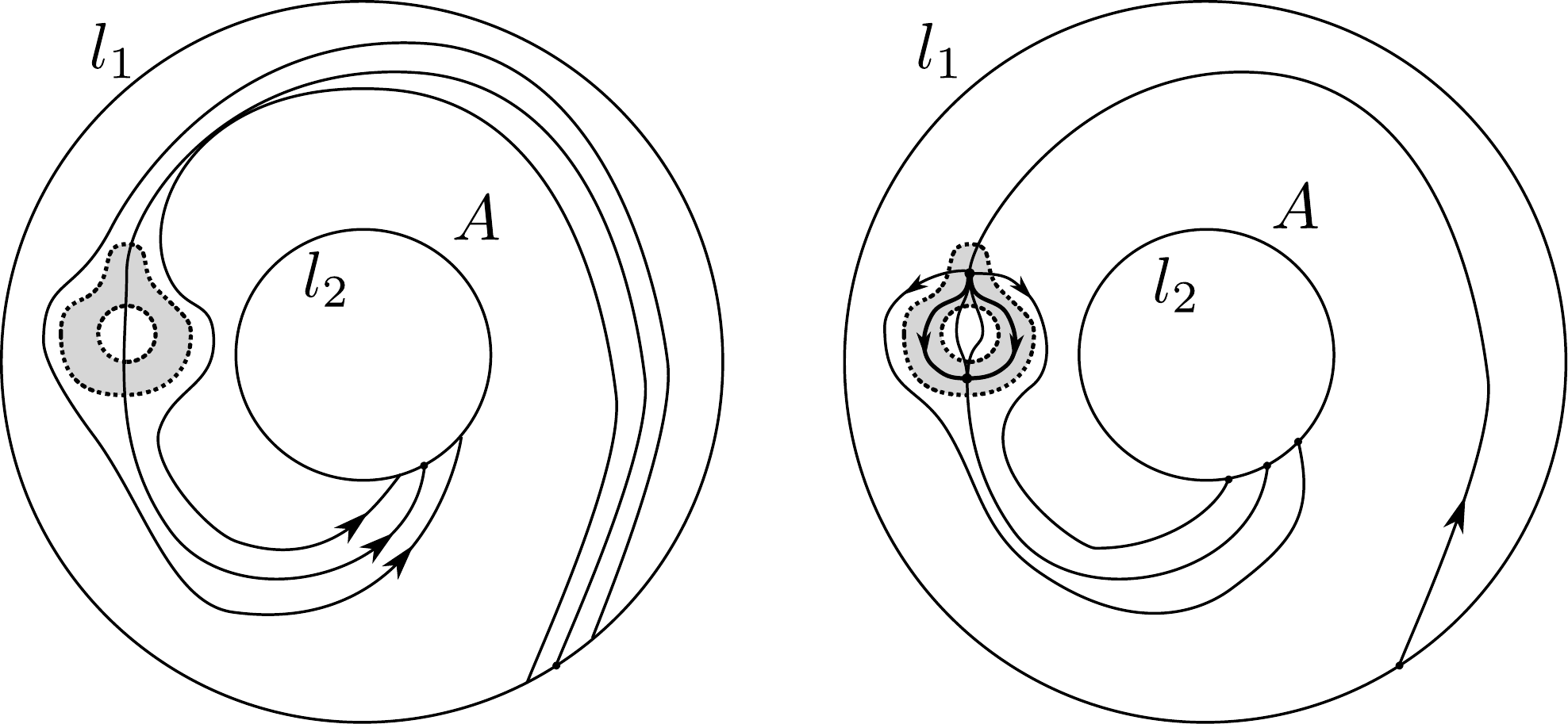}
\caption{Empty annuli lemma: auxiliary lemma  \ref{lem-S2-U}, orbits of $v_\eps$ (left) and of $v_0$ (right). The component of $U$ is shadowed, its boundary is dotted, the large bifurcation support is shown in thick. The vector field $v_0$ has degeneracy of codimension 2. As $\eps$ changes, its two saddlenodes vanish, and the annulus becomes  empty}\label{fig:onecompcu}
\end{center}
\end{figure}
\begin{lemma}
\label{lem-S2-U}
 The statement of the Empty annuli lemma holds true if $l_1$, $l_2$ belong to the same connected component of $CU$.
\end{lemma}

\begin{proof} As the curves $l_1, l_2$ belong to the same component of $CU$, they may either belong to $\partial U$, or to the interior of $CU$. So for each of them, we have  the following three cases:
\begin{enumerate}
  \item A transversal loop of a hyperbolic sink or source, or a hyperbolic cycle, or a non-interesting cycle (namely $l\in L$);
  \item A boundary component of $\bar U$ of Type 2;
  \item A boundary component of $\bar U$ of Type 3.
  \end{enumerate}

If one of $l_1,l_2$ is a Type 2 transversal boundary component of $U$, then the other one  is a transversal loop around the corresponding non-interesting $\alpha$- or $\omega$-limit set. This follows from Boundary lemma and Proposition \ref{prop-Case2-nonint-ls}. In this case Lemma \ref{lem-S2-U} follows from Corollary \ref{cor-image-transvers-Case2} above.

Let us prove that in all other cases the following implication holds: if $A$ is empty for $v_{\eps}$ and $\hat H (A)$ is empty for $w_{h(\eps)}$, then $\tilde A$ is empty for $w_{h(\eps)}$.

Suppose that both $l_1, l_2$ are of Type 3.
In this case, Lemma \ref{lem-image-case3} implies that $\hat H(l_i)$ and $H_{\eps}(l_i)$ are homotopic in $\tilde U^+\setminus (\Sing w_{h(\eps)} \cup \Per w_{h(\eps)})$ as oriented curves. So we may replace $\tilde l_i=H_{\eps}(l_i)$ by $\hat H(l_i)$: if $\hat H(A)$ is empty, then $\tilde A$ is empty as well.

Suppose that $l_1$ is of Type 3, and $l_2$ falls into the case 1 above. Then $H_\eps (l_1)$ may be replaced by $\hat H(l_1) $ as before, and $\tilde l_2=\hat H(l_2)$. Again, if $\hat H(A)$ is empty, then $\tilde A$ is empty as well.

Suppose that both  $l_1,l_2$ fall into the case 1 above. In this case, $\tilde l_i=\hat H(l_i)$, and there is nothing to prove.

Now, the following proposition implies the first assertion of the Empty annuli lemma.

\begin{proposition}
\label{prop-H0-empty}
In assumptions of Lemma \ref{lem-S2-U}, the annulus $\hat H(A)$ is empty with respect to $w_{h(\eps)}$.
\end{proposition}

\begin{proof}
By contraposition, suppose that some singular point or a cycle $\tilde c$ of $w_{h(\eps)}$ is in $\hat H(A)$. We only consider the case when $\tilde c$ is a cycle; the case of a singular point is analogous.

Due to Proposition \ref{prop-inU-or-hyp} applied to the family $W$, the cycle $\tilde c$ either belongs to $\tilde U^-$ or belongs to a continuous family of hyperbolic cycles $\tilde c_\delta $ of vector fields $w_\delta $ defined for all $\delta$ small.

In the first case, let $H_{\eps}(U_i)$ be a connected component of $H_{\eps}(U)$ that contains $\tilde c$. Then $U_i$ contains a cycle $H_{\eps}^{-1}(\tilde c)$ of $v_{\eps}$.
It remains to prove that  $U_i$ is inside $A$; this will contradict to the fact that $A$ is empty for $v_{\eps}.$

Since $\hat H(U_i)$ is the only component of $\hat H(U)$ that intersects $H_{\eps}(U_i)$ (see Proposition \ref{prop-U12-intersect}) and $\tilde c \subset \tilde U^- \subset \hat H(U)$, we have $\tilde c\subset \hat H(U_i)$. So $\hat H(U_i)$ intersects the annulus $\hat H(A)$. Thus $U_i$ intersects the annulus $A$, and since boundaries of $A$ are in one and the same connected component of $CU$, we have that $U_i \subset A$. We get a contradiction mentioned in the previous paragraph.

In the second case, the cycles $\tilde c_\delta $ belong entirely to $\hat H(A)$ because the boundary of this annulus is transversal to $w_\delta$ for any  $\delta$ small. Hence, no limit cycle of $w_\delta$ can cross this boundary. Therefore, $\tilde c_0$ belongs to $\hat H(A)$ as well.  Therefore, the vector field $v_0$ has a hyperbolic limit cycle    $\hat H^{-1}(\tilde c_0)$ in $A$. The same holds for the vector field $v_\eps$, so $A$ is not empty with respect to $v_{\eps}$.
 We get a contradiction again.
 This finishes the proof.
\end{proof}

The first statement of the Empty Annuli lemma in assumptions of Lemma \ref{lem-S2-U} is proved.

Let us prove the second one:
$G_\varepsilon$ may be extended to a homeomorphism of $A$.  The same statement for $\hat H(A)$ instead of $\tilde A$ is clear, because $\hat H$ is an orientation-preserving homeomorphism. Suppose that both $l_1,l_2$ belong to $\overline U$ (other cases are analogous but simpler). Since $l_1$ and $l_2$ are in different connected components of $U$, their images $H_{\eps}(l_1)$ and $H_{\eps}(l_2)$ are in different connected components of $\tilde U^+$ (see Proposition \ref{prop-U12-intersect}; this proposition is applicable due to the last assertion of  Proposition \ref{prop-choice}). By Lemma \ref{lem-image-case3},  the oriented curves $H_\eps(l_1)$ and $\hat H(l_1)$ are homotopic inside $\tilde U^+$. Thus as we perform the homotopy between $H_\eps(l_1)$ and $\hat H(l_1)$, all the intermediate curves do not intersect $\tilde l_2$. Similar arguments apply to the homotopy between $H_{\eps}(l_2)$ and $\hat H(l_2)$. Finally, $\tilde l_1, \tilde l_2$ are oriented with respect to $\tilde A$ in the same way as $\hat H(l_1), \hat H(l_2)$
with respect to $\hat H(A)$. The latter orientation coincides with the orientation of  $l_1,l_2$ with respect to $A$, which implies the statement.
\end{proof}

\section{Proof of the Correspondence Lemma \ref{lem-seps-behave}}   \label{sec-Corr}
\subsection{Plan of the proof}

Without loss of generality we assume that $l$ is an ingoing transversal loop. Choose $U$ following Sec. \ref{sec-choice-U}. Note that any transversal loop   $l\subset  LMF(v_\eps)$ either belongs entirely to $U$, or to its complement, due to the choice of transversal loops in Sec. \ref{ssec-tr}. So there are the following cases to consider depending on the location of $l$.

\begin{itemize}
 \item The loop $l\in LMF(v_{\eps})$ lies outside $U$. Simultaneously, we will prove the statement of Correspondence lemma for $l\in L$  (see Definition \ref{def-tr-loops} of the collection $L$), though such loops may be not included in $LMF(v_{\eps})$.

 \begin{itemize}
 \item Case 1. Some backward orbit of $l$ under $v_0$ hits $\partial U^*$ at a transversal boundary component.

\item
Case 2. All backward orbits of $l$ under $v_0$ either hit $\partial U^*$ at  non-transversal boundary components,  or do not intersect $\overline {U^*}$.
 \end{itemize}

\item The loop $l\in LMF(v_{\eps})$ lies inside $U$.

\begin{itemize}

\item Case 3. $l \subset U^*$.

\item Case 4. $l \subset U \setminus U^*$, i.e. $l$ is inside a non-interesting nest. Here we will use Correspondence lemma for $l\in L$ (Case 1 above).
\end{itemize}
\end{itemize}

In Sec. \ref{sec-corr-first} -- \ref{sssec-corr-inside-ni}, we prove the first statement of Correspondence Lemma in each of the above four cases.  In Sec. \ref{ssec-corr-second}, we prove the second statement of the Correspondence Lemma.

\subsection{The first statement of the Correspondence lemma: Case 1}\label{sec-corr-first}
   \begin{figure}[h]
 \begin{center}
  \includegraphics[width=0.3\textwidth]{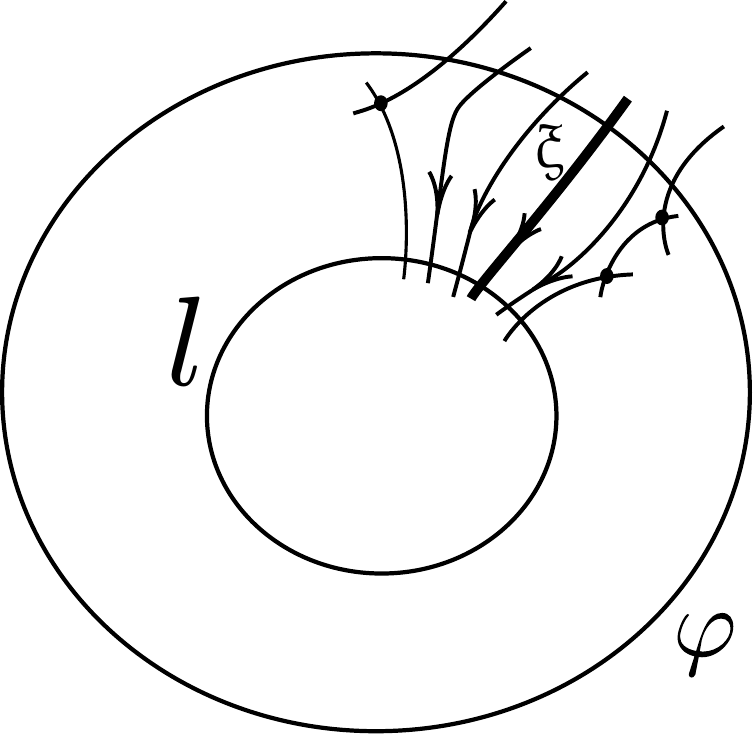}
 \end{center}
\caption{Proposition \ref{prop-corr-case1}: Poincare map between $\varphi$ and $l$.}\label{fig-corrlemma-reg}
\end{figure}
\begin{proposition}
\label{prop-corr-case1}
 The first statement of the Correspondence lemma holds true for an ingoing transversal loop $l\in LMF(v_{\eps})$ or $l\in L$, if backward orbits of \emph{some} points of $l$ under $v_0$ hit a transversal boundary component of $\partial U^*$.
\end{proposition}

\begin{proof}

Let $\varphi$ be one of these components of $\partial U^*$.

A trajectory of $v_0$ joins $\varphi$ to $l$, so a close trajectory $\xi$ of $v_{\eps}$ joins $\varphi$ to $l$ as well. We conclude that  a Poincare map along $v_{\eps}$ between some arcs of the   transversal loops $\varphi$ and $l$ is defined. The endpoints of its domain must be intersections of $\varphi$ with separatrices of $v_{\eps}$ (see Fig. \ref{fig-corrlemma-reg}). But separatrices of $v_{\eps}$ do not enter $U^*$ through $\varphi$ due to No-entrance  lemma \ref{lem-seps}. Therefore this Poincare map is defined on the whole $\varphi$, thus $l$ and $\varphi$ bound an annulus $A$ filled by trajectories of $v_{\eps}$. We conclude that the separatrices $\gamma_i$  of $v_{\eps}$ that cross $l$  also cross $\varphi$, and the intersection points $\gamma_i\cap \varphi$ are ordered clockwise with respect to $U$ (see Fig. \ref{fig-corrlemma}).
  \begin{figure}[h]
 \begin{center}
  \includegraphics[width=0.8\textwidth]{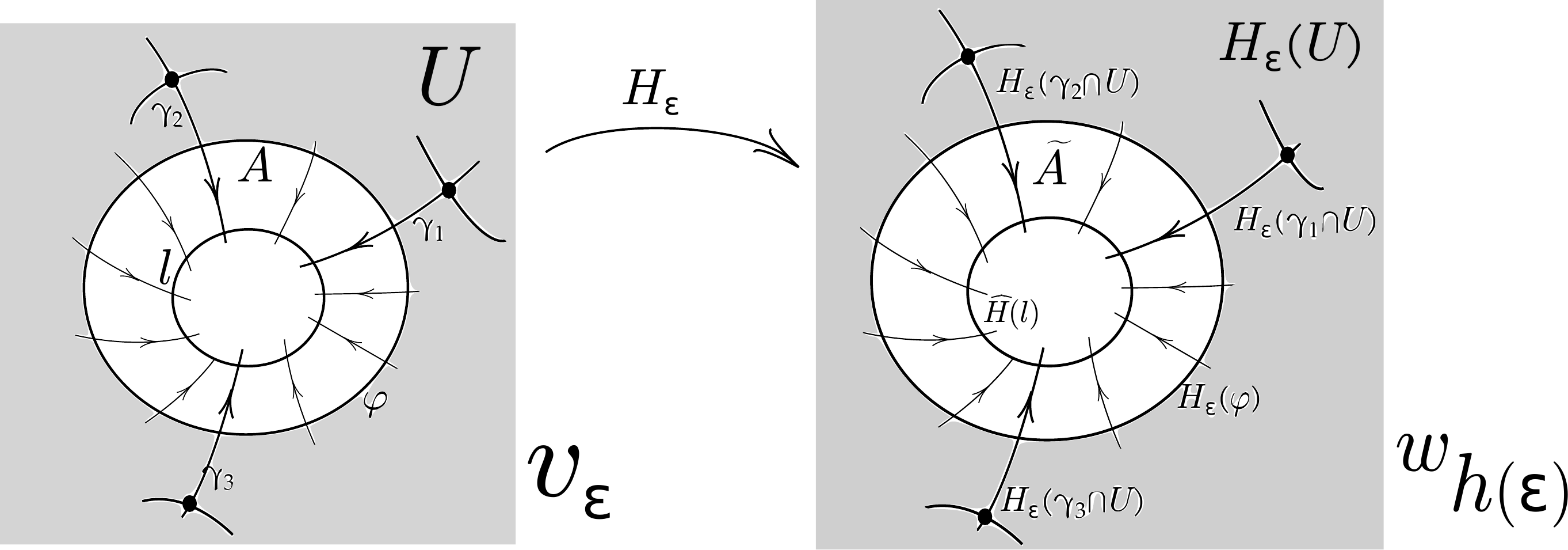}
 \end{center}
\caption{Proposition \ref{prop-corr-case1}, annuli $A$ and $\tilde A$}\label{fig-corrlemma}
\end{figure}
Due to No-entrance  lemma \ref{lem-seps}, separatrices of $v_{\eps}$ cannot enter $U^*$, see Definition \ref{def-entering}. So all separatrices of $v_{\eps}$ that cross $\varphi\subset \partial U^*$ intersect it only once, and $\gamma_i\cap U$ are their arcs starting at the corresponding singular points $P_i \in U^*$. Hence  $\tilde \gamma_i$ are separatrices of $H_\eps (P_i)$ that contain arcs $H_{\eps}(\gamma_i\cap U)$, so $\tilde \gamma_i$ intersect $H_{\eps}(\varphi)$, and the intersection points $H_{\eps}(\gamma_i\cap \varphi)$ are ordered clockwise with respect to $H_{\eps}(U)$.

Finally, note that $l$ and $\varphi$  satisfy assumptions of Empty annuli Lemma, that is, they bound an empty annulus $A$, see Fig \ref{fig-corrlemma}. This lemma yields that  $\tilde l=\hat H(l)$ and $H_{\eps}(\varphi)$ bound an empty annulus $\tilde A$ for $w_{h(\eps)}$, and the orientation on its boundaries is the same as for $A$. Thus the separatrices $\tilde \gamma_i$ that cross $H_{\eps}(\varphi)$ also cross $\tilde l$, and the intersection points $\tilde \gamma_i\cap \tilde l$ are ordered  counterclockwise with respect to it.
 \end{proof}

\subsection{Case 2}

\begin{figure}[h]
   \begin{center}
  \includegraphics[width=0.5\textwidth]{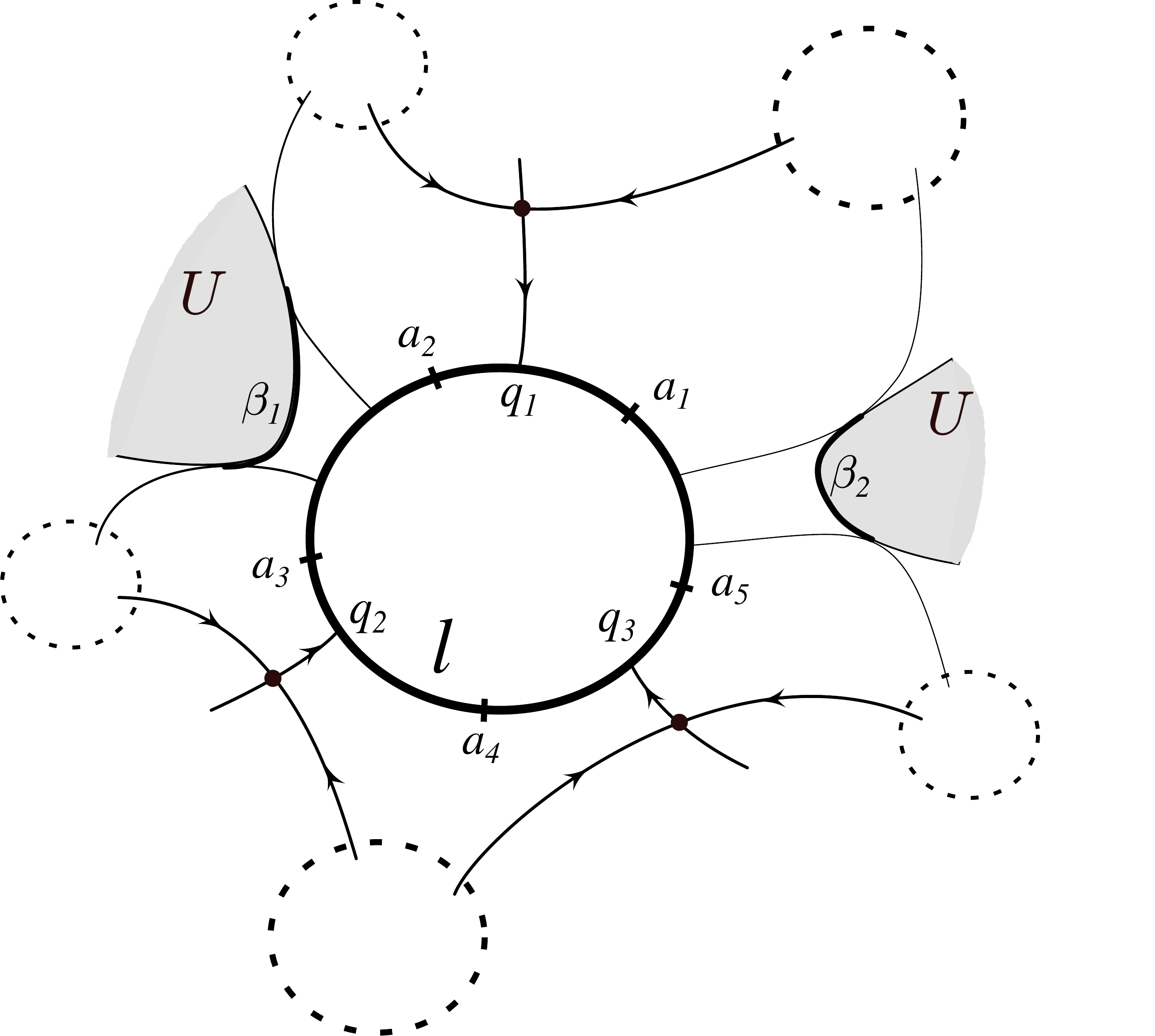}
  \end{center}
\caption{Proposition \ref{prop-corr-case2}: points $q_i, a_j$ on $l$. Dashed circles are transversal loops around non-interesting $\alpha$-limit sets outside $U^*$.}\label{fig-corrlemma-case2}
\end{figure}

\begin{proposition}
\label{prop-corr-case2}
 The first statement of the Correspondence lemma holds true for an ingoing transversal loop $l\subset LMF(v_{\eps})$ or $l\in L$, if backward orbits  of all points of $l$ under  $v_0$ either cross $\partial U^*$  by non-transversal boundary components, or do not intersect $\overline {U^*}$.
\end{proposition}
\begin{proof}
Note that these boundary components must be all of Type $2$ (see Boundary lemma for the classification). Indeed, non-transversal boundary components of $\partial U$ are of Type $1$ or Type $2$. But the trajectories originating from a Type $1$ component fill the whole disc with no transversal loops in it.

Clearly, $l$ is the union of the following sets:
\begin{itemize}
 \item (open) \emph{hyperbolic} arcs: a negative semi-trajectory of each point of this arc under $v_0$ tends to a non-interesting set and does not hit $U^*$.
  \item intersections $q_i$ with separatrices of $(v_0)|_{S^2\setminus U}$, i.e. with separatrices of hyperbolic saddles.
 \item (closed) images $P_0(\beta_i)$ of transversal outgoing arcs $\beta_i\subset \varphi_j \subset \partial U$ under Poincare maps $P_0$ along $v_0$, where each $\varphi_j$ is a boundary component of Type $2$, see Fig. \ref{fig-corrlemma-case2}.
\end{itemize}

Pick one point from each hyperbolic arc; let $a_i$ be these points (ordered cyclically along $l$). As $l$ lies outside $U$, again as in Case $1$, $\tilde l = \hat H(l)$.
Put $\tilde a_i = \hat H(a_i)$; these points are ordered cyclically along $\tilde l$.
It is sufficient to prove the statement of Correspondence lemma for each arc $[a_i, a_{i+1}]$. Put $I=[a_i, a_{i+1}]$, $\tilde I = [\tilde a_i, \tilde a_{i+1}]$.

Since each $q_j$ and each  $P_0(\beta_j)$ is adjacent to  open hyperbolic arcs on both sides, we have the following three cases for $I$:
\begin{enumerate}

 \item The arc $I$ contains the point $q$ of intersection with a separatrix $\nu$ of a hyperbolic saddle $P$ of $v_0$; $I\setminus \{q\}$ belongs to two subsequent hyperbolic arcs. The arcs $[a_1, a_2], [a_3, a_4], [a_4, a_5]$ on Fig. \ref{fig-corrlemma-case2} are of that type, as well as all arcs on Fig. \ref{fig-corrlemma-case3}.

     \begin{figure}[h]
   \begin{center}
  \includegraphics[width=0.5\textwidth]{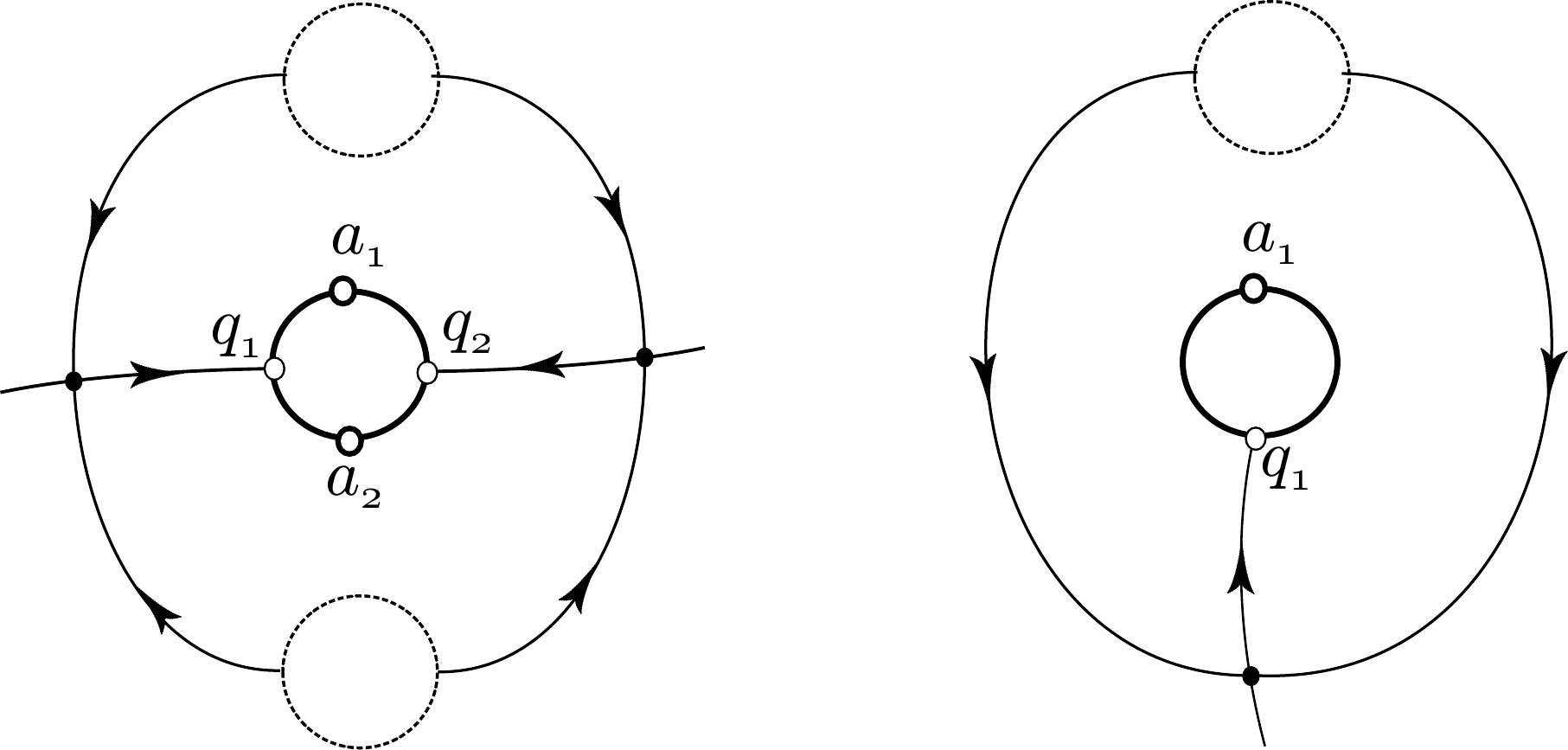}
  \end{center}
\caption{Proposition \ref{prop-corr-case2}: case 1}\label{fig-corrlemma-case3}
\end{figure}

 Then the only separatrix of $v_{\eps}$ that intersects $I$ is $\gamma:= \pi_{\eps}^{-1}(\nu)$.
 Since $\hat H$ conjugates $v_0$ to $w_0$, the separatrix $\hat H(\nu)$ intersects the arc $\hat H(I) = \tilde I$. The close separatrix of $w_{h(\eps)}$, namely the separatrix that contains a germ $\tilde \pi^{-1}_{\eps} (\hat H( \nu, P))$, also intersects  $\tilde I$. This germ is $G_{\eps}(\gamma, P)$, thus this separatrix is $\tilde \gamma$. We conclude that if $\gamma$ intersects $I$, then $\tilde \gamma$ intersects $\tilde I$. This completes the proof of the statement of Correspondence lemma for $I$ in this case.

 \item  The arc $I$ contains the image $P_0(\beta)$ of some transversal outgoing arc $\beta\subset \varphi \subset \partial U$. The set  $I \setminus P_0(\beta)$  belongs to two subsequent hyperbolic arcs. The arcs $[a_2, a_3], [a_5, a_1]$ on Fig. \ref{fig-corrlemma-case2} are of that type.

As in Lemma \ref{lem-image-Case2}, let $\beta_{\eps}$ be the maximal arc of $\varphi$ transversal to $v_{\eps}$ and close to $\beta$. Let $P_{\eps}\colon \beta_{\eps} \to l$ be the Poincare map along $v_{\eps}$.

Since separatrices of $v_{\eps}$ cannot originate from non-interesting sets, each separatrix of $v_{\eps}$ that intersects $I$ also intersects $\beta_{\eps}$. Let $\{\gamma_k\}$ be these separatrices, ordered counterclockwise along $I$; then they intersect $\beta_{\eps}$ and are ordered clockwise along it. Due to No-entrance  lemma, separatrices $\gamma_k$ intersect $\partial U^*$ only once. Thus the separatrices $\tilde \gamma_k$ are the separatrices that contain arcs $ H_{\eps}(\gamma_k \cap U)$. Since $H_{\eps}$ conjugates $v_{\eps}$ to $w_{h(\eps)}$, we conclude that the separatrices $\tilde \gamma_k$ intersect $\tilde \beta_{\eps} := H_{\eps}(\beta_{\eps})$ in a clockwise order along $\tilde \beta_{\eps}$. Lemma \ref{lem-image-Case2} implies that the Poincare map $\tilde P_{\eps}\colon \tilde \beta_{\eps} \to \tilde l $ is well-defined and takes the clockwise orientation on $\tilde \beta_{\eps}$ to the counterclockwise orientation on $\tilde l$. Thus  the  separatrices $\tilde \gamma_k$ intersect $\tilde P_{\eps}(\tilde \beta_{\eps}) \subset \tilde l$ in a counterclockwise order along $\tilde l$. Now it suffices to prove that $\tilde P_{\eps}(\tilde \beta_{\eps}) \subset \tilde I$; this will prove that $\tilde \gamma_k$ intersect $\tilde I$ and are ordered counterclockwise along it.

 Applying $\hat H$ to the inclusion $P_0(\beta) \subset I$, we get that $\tilde P_0(\tilde \beta_0) \subset \tilde I$. Lemma \ref{lem-image-Case2} implies that $\tilde P_{\eps}(\tilde \beta_\eps)$ intersects $\tilde P_0(\tilde \beta_0)$. Now, it is sufficient to prove that $\tilde P_{\eps}(\tilde \beta_\eps)$ does not contain endpoints of $\tilde I$.  Note that the negative semi-trajectories of endpoints of $\tilde I$ under $w_{h(\eps)}$ are close to their trajectories under $w_0$, thus tend to non-interesting sets and do not intersect the closure of $\tilde U^{+*}$. This completes the proof of the statement of Correspondence lemma for $I$ in this case.
\item The arc $I$ does not fall into the two previous cases. So it belongs to one hyperbolic arc, which is only possible if all $a_i $  coinside.  Then $I=l$ and all negative semi-trajectories of points of $l$ under $v_0$ tend to a non-interesting set and do not visit $U^*$. Thus for small $\eps$, no separatrices of $v_{\eps}$ intersect $l$, and there is nothing to prove.
\end{enumerate}
\end{proof}

\subsection{Case 3:  $l \subset U^*$}

For $l\subset U^*$, Correspondence lemma follows directly from No-entrance  lemma \ref{lem-seps}. 

   \begin{proposition}
        \label{prop-corr-lem-in-U-1}
          The first statement of Correspondence lemma holds if $l$ is a transversal loop of an $\alpha$- ($\omega$-) limit set inside $U^*$.
\end{proposition}
\begin{proof}
Let $\{\gamma_j\}$ be the set of all separatrices that hit $l$, as in Correspondence lemma. No-entrance  lemma \ref{lem-seps} implies that they belong completely to $U^*$, i.e. $\gamma_j\subset U^*$. So $G_{\eps}$ is induced by $H_{\eps}$ on $\gamma_j$. 
The statement follows from the fact that $H_{\eps}$ is a homeomorphism.


\end{proof}

The same arguments apply if we consider arbitrary separatrices with $\omega$-limit sets inside $U^*$. 
Let $P$ be a singular point of $v_\eps$, let $\gamma$ be its  unstable  separatrix. Suppose that $\tilde \gamma$ is the corresponding  separatrix of $w_{h(\eps)}$: $(\tilde \gamma, \tilde P) = G_{\eps}((\gamma, P))$. 

\begin{proposition}
\label{prop-corr-lem-in-U-2}
In assumptions of Main Theorem, let $\gamma$ be an unstable  separatrix of $v_{\eps}$.
For sufficiently small $\eps$, if  $\omega(\gamma)$ belongs to $U^*$, then $\omega(\tilde \gamma) = G_{\eps}(\omega(\gamma))$.
      
   Here $\omega$-limit sets are with respect to $v_{\eps}, w_{h(\eps)}$. The same holds for stable separatrices and their $\alpha$-limit sets.
\end{proposition}

The proof literally repeats the proof of the previous proposition.

\subsection{Case 4: $l \subset U\setminus U^*$}
\label{sssec-corr-inside-ni}

Suppose that $l\subset U\setminus U^*$  corresponds to a cycle $c$ that bifurcates from a non-interesting nest,  and suppose that $l$ is not homotopic in $S^2\setminus \Per v_{\eps}$ to the outer transversal loops of the nest. Then  no separatrices of $v_{\eps}$ intersects $l$. Indeed, if a separatrix $\gamma$ accumulates to $c$, then it must enter a non-interesting nest, i.e. intersect its outer transversal loop $l'\in L$. However $l'$ is separated from $l$ by cycles of $v_{\eps}$, and we get a contradiction. So the first statement of the Correspondence lemma  is trivial for $l$.

Suppose that $l$ is homotopic in $S^2\setminus \Per v_{\eps}$  to the outer transversal loop $l'$ of the nest. Then the separatrices $\{\gamma_i\}$ of $v_{\eps}$ that intersect $l$ also intersect $l'$. Cases 1,2 of Correspondence lemma (see Propositions \ref{prop-corr-case1}, \ref{prop-corr-case2})  for $l'\in L$ imply that $\tilde \gamma_i$ intersect $\hat H(l')$ and are ordered cyclically along it. Empty annuli lemma \ref{lem-annuliShaped} implies that the annulus $\tilde A$ between $H_{\eps}(l)$ and $\hat H(l')$ is empty with respect to $w_{h(\eps)}$, so $\{\tilde \gamma_i\}$ intersect $H_{\eps}(l)$. Their order is the same as for $v_{\eps}$, because the curves $H_{\eps}(l)$ and $\hat H(l')$ are oriented with respect to $\tilde A$ in the same way as $l,l'$ are oriented with respect to $A$ (see Empty annuli lemma \ref{lem-annuliShaped}). This completes the proof.

\subsection{Second statement of the Correspondence lemma}
\label{ssec-corr-second}
The first statement of the Correspondence lemma implies the second one because $v_{\eps}$ and $w_{h({\eps})}$ have the same amount of separatrices. For a detailed proof, we will need the following proposition.
\begin{proposition}
   \label{prop-sameAmount}
    For small $\eps$, the vector field $w_{h(\eps)}$ has the same amount of separatrices as $v_{\eps}$.
   \end{proposition}
\begin{proof}
Note that the number of separatrices of a vector field equals to the number of germs of separatrices at the corresponding singular points minus the number of separatrix connections.
Due to  Proposition \ref{prop-G-bij}, the map $G_{\eps}$ provides a  one-to-one correspondence on singular points  of $v_{\eps}$ and $w_{h(\eps)}$. This map preserves their topological types due to Remark \ref{rem-Geps-toptype}. So the amount of germs of separatrices at singular points is the same for $v_{\eps}$ and $w_{h(\eps)}$.

Moreover, $v_{\eps}$ and $w_{h(\eps)}$ have the same amount of separatrix connections: due to Separatrix lemma \ref{lem-sep-connect}, all separatrix connections of $v_{\eps}$, $w_{h(\eps)}$ for small $\eps$ are inside small neighborhoods of large bifurcation supports, so $H_{\eps}$ identifies separatrix connections of $v_{\eps}$ and separatrix connections of $w_{h(\eps)}$. The statement follows.
   \end{proof}

The first statement of Correspondence Lemma implies that if $k$ separatrices of $v_{\eps}$ intersect a transversal  loop $l$ outside $U^*$, then \emph{at least} $k$ separatrices of $w_{h(\eps)}$ intersect a transversal loop $\tilde l$.

Proposition \ref{prop-corr-lem-in-U-2} implies that if $k$ unstable separatrices of $v_{\eps}$ have the same $\omega$-limit set $c$ inside $U^*$, then \emph{at least} $k$ separatrices of $w_{h(\eps)}$ have the $\omega$-limit set  $H_{\eps}(c)$; the same holds for $\alpha$-limit sets of stable separatrices.

Clearly, each separatrix of $v_{\eps}$ falls into one of the two cases above.    Due to Proposition \ref{prop-sameAmount}, $v_{\eps}$ and $w_{h(\eps)}$ have the same amount of separatrices. So in each of the two cases above, the amount of separatrices  of $v_{\eps}$ equals the amount of corresponding separatrices of $w_{h(\eps)}$. This completes the proof of Correspondence lemma.

\section{Proof of the Boundary lemma}  \label{sec-U}

\subsection{Boundaries of  canonical regions}
In the proof of the Boundary lemma, we will construct a neighborhood $\Omega$ as the union of its intersections with all canonical regions of $v$. Recall that these regions are described  in Section \ref{sub:skel}. We start with an explicit description of the boundaries of canonical regions.

Let $v \in Vect^*(S^2)$. In the definitions below all the singular points, separatrices and so on are those of $v$.

\begin{definition}   \label{def-chain}
A \emph{separatrix chain} $C \subset S^2$ is one of the following sets:
 \begin{itemize}
\item A union $C = \alpha(\gamma_0) \cup \gamma_0 \cup P_1 \cup \gamma_1 \cup P_2 \cup \dots \cup \gamma_n \cup \omega(\gamma_n)$, where  $\gamma_i$ is an ingoing separatrix of a singular point $P_{i+1}$ and $\gamma_{i+1}$ is an outgoing  separatrix of  $P_{i+1}$. In what follows, we say that $\gamma_0$ is the first separatrix of the chain $C$, and $\gamma_n$ is the last separatrix of the chain; we also say that the chain $C$ connects the limit sets $\alpha(\gamma_0)$ to $\omega(\gamma_n)$.

\item A union $C=\alpha(\gamma)\cup \gamma \cup \omega(\gamma)$ where $\gamma$ is a separatrix. Then $\gamma$ is both the first and the last separatrix in the chain, and the chain connects $\alpha(\gamma)$ to $\omega(\gamma)$.

\item A singular point; it coincides with both its $\alpha$- and $\omega$-limit sets and the corresponding chain has no  separatrices.
 \end{itemize}

\end{definition}

Note that points $P_i$ with different numbers in one and the same chain may coincide.

\begin{definition}
For a canonical region $R$ of a vector field, we denote by $\alpha(R)$ and $\omega(R)$ the common $\alpha$- and $\omega$-limit set of all its points.
\end{definition}

Note that a strip canonical region is simply connected, and a spiral one is a topological annulus. Recall that due to Proposition \ref{prop-parall}, for a strip canonical region, there exists a homeomorphism $\Psi \colon \bbR \times (0,1) \to R$ that conjugates $\partial/\partial x$ to $v$.

\begin{definition}
\emph{Side boundaries} $\nu_1(R), \nu_2(R)$ of a strip canonical region $R$ are upper topological limits
$$\nu_1(R)= \overline \lim_{y\to 0}\Psi(\bbR \times \{y\}),$$

$$\nu_2(R)= \overline \lim_{y\to 1}\Psi(\bbR \times \{y\}).$$
\end{definition}
Clearly, $\partial R$ is a union of two side boundaries of $R$. Each one of the side boundaries includes $\alpha(R)$ and $\omega(R)$.

\begin{lemma}[Side boundaries of strip canonical regions]
\label{lem-sides}
For a vector field $v\in Vect^*\,S^2$, side boundaries $\nu_1(R)$, $\nu_2(R)$ of a strip canonical region $R$ of $v$ are chains of separatrices that join $\alpha(R)$ to $\omega(R)$.
\end{lemma}

We expect that this lemma is known to experts, but we did not find it in the literature.

\begin{remark}
One can prove that the homeomorphism $\Psi$ can be so chosen  that it extends continuously to $\psi_1\colon \bbR\times \{0\}\to S^2$  and $\psi_2\colon \bbR \times \{1\}\to S^2$. The images of $\psi_{1,2}$ contain $\nu_{1,2}(R)\setminus (\alpha(R)\cup \omega(R))$ respectively and are contained in $\nu_{1,2}(R)$. Note that $\psi_1, \psi_2$ may glue subsegments of their domains  in various ways, see Fig. \ref{fig-canonreg-ex}.

\end{remark}

We will not prove this statement, because we are going to use it in some heuristic arguments only. 

\begin{proof}[Proof of Lemma \ref{lem-sides}]
We prove the lemma for $\nu_1(R)$.
Note that $\nu_1(R)$ is a closed and $v$-invariant set; also, $\nu_1(R)\subset \partial R \subset S(v) = \Sing v \cup \Per v \cup \Sep v$.

Note that $\nu_1(R)\setminus \alpha(R) \setminus \omega(R)$ may not contain limit cycles. Indeed, since $\alpha$-, $\omega$-limit sets of all points of $R$ are $\alpha(R)$, $\omega(R)$, the set $\nu_1(R)$ is detached from basins of attraction and repulsion of all other $\alpha$-, $\omega$-limit sets of $v$, and may not contain limit cycles other than $\alpha(R)$ or $\omega(R)$.
Therefore  $\nu_1(R)\setminus \alpha(R) \setminus \omega(R) \subset \Sing v \cup \Sep v$.

Since $\nu_1(R)$ is connected as a limit of connected sets, $\nu_1(R)\setminus \alpha(R) \setminus \omega(R)$ may not contain isolated singular points; it is either empty or contains a separatrix.

If $\nu_1(R)\setminus \alpha(R) \setminus \omega(R)$ is empty, the argument that $\nu_1(R)$ is connected implies that $\alpha(R)$ and $\omega(R)$ intersect. This is only possible if  $\alpha(R)=\omega(R)$ is a singular point, and the statement is proved (this may happen when $R$ is an elliptic sector of a complex singular point).

Suppose that $\nu_1(R)\setminus \alpha(R) \setminus \omega(R)$ contains a separatrix $\gamma$ (this is the last case to consider). Since $\nu_1(R)$ is a limit of trajectories $\Psi(\bbR \times \{y\})$ and $\Psi$ is injective, a local analysis in each flow-box surrounding $\gamma$ shows that there exists a semi-neighborhood $U_{\gamma}$ of $\gamma$ that belongs to $R$.

Note that $\omega(\gamma)\subset \nu_1(R)$ because $\nu_1(R)$ is closed. There are the following possibilities for $\omega(\gamma)$:
\begin{itemize}
 \item $\omega(\gamma)$ is a cycle or a polycycle. Then all the points in a neighborhood of  $\gamma$ are also attracted to this set, including some points of $R$; thus $\omega(R) = \omega(\gamma)$. So $\gamma$ will be the last separatrix in the chain.
 \item $\omega(\gamma)$ is a singular point $P$, and the semi-neighborhood $U_{\gamma}$ contains a piece of parabolic or elliptic sector near $(\gamma, P)$. Similarly, $\omega(R) = \omega(\gamma)$, and $\gamma$ will be the last separatrix in the chain.
 \item $\omega(\gamma)$ is a singular point $P$, and the semi-neighborhood $U_{\gamma}$  contains a piece of a hyperbolic sector near $(\gamma, P)$; so $\gamma$ is a separatrix of $P$. Then $\gamma$ will be a separatrix $\gamma_i$ in the middle of the chain,  $P_{i+1}=P$, and $\gamma_{i+1}$ is another border of the same hyperbolic sector. Now we may repeat our arguments for $\gamma_{i+1}$ and find $P_{i+2}, \gamma_{i+2}$, etc.
\end{itemize}
The same arguments apply to $\alpha(\gamma)$ and allow us to enumerate separatrices of $\nu_1(R)$ as required. Possibly we will have only one separatrix $\gamma_0=\gamma_n$ and no singular points $P_i$.   This may happen, for instance, when $\partial R$ is a union of a singular point and its homoclinic curve, a separatrix, and $R$ is an elliptic sector (see Fig. \ref{fig-canonreg-ex} middle). Note also that one and the same singular point may appear several times in the list $\{P_i\}$, see Fig. \ref{fig-canonreg-ex} right.

\begin{figure}[h]
\begin{center}
\includegraphics[width=0.2\textwidth]{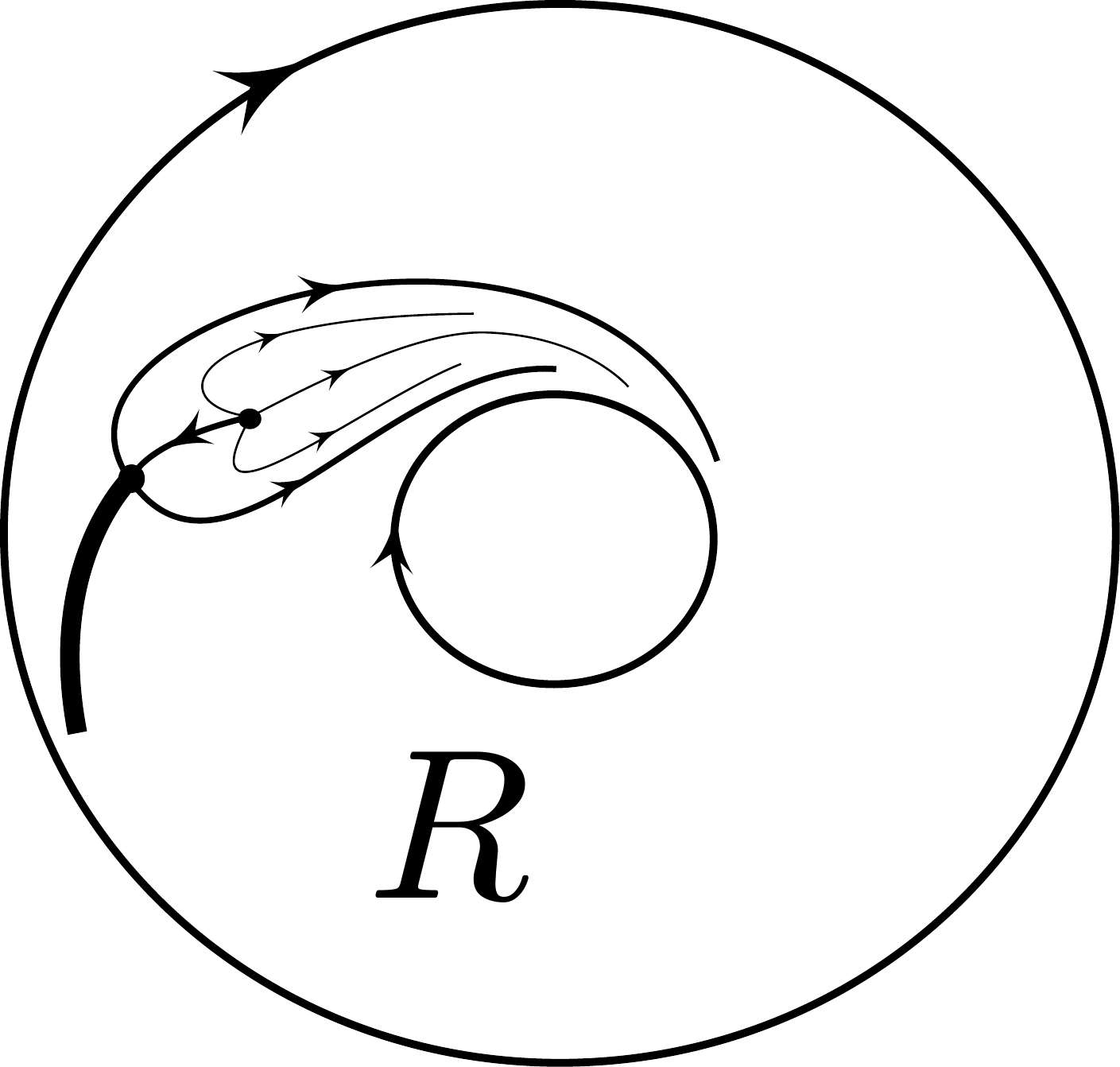}
\hfil
\includegraphics[width=0.7\textwidth]{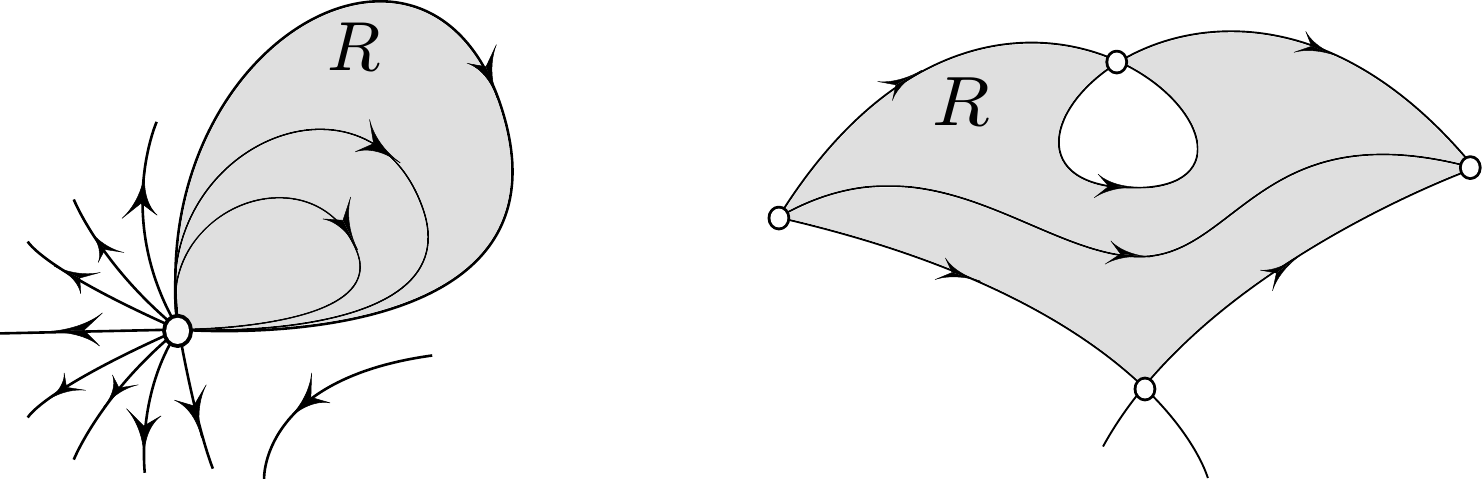}
\end{center}
 \caption{Some possible shapes of canonical regions}\label{fig-canonreg-ex}
\end{figure}

It is easy to see that the union of  semi-neighborhoods of $\gamma_i$ and hyperbolic sectors at $P_i$ is saturated by trajectories of $v$, so it exhausts all $\Psi(\bbR\times(0,\eps) )$ for small $\eps$; hence $\nu_1(R)$ coincides with the chain $\alpha(R)\cup \omega(R)\cup \{\gamma_i\} \cup \{P_i\}$.
\end{proof}

\begin{lemma}\label{lem:spiral}
The boundary of a spiral canonical region is the union of its $\alpha$- and $\omega$-limit sets:
$$
  \partial R = \alpha(R) \cup \omega(R).
$$
\end{lemma}

The proof is obvious.

The following proposition provides a key tool for the proof of the  Boundary lemma.

\begin{proposition} \label{prop-Z-inters-R}
For a vector field $v\in Vect^*\,S^2$, let $Z\subset S^2$ be a closed, $v$-invariant set with Sep-property.

1) Let $C= \{\gamma_i\}_{i=1}^{n-1}\cup \{P_i\}_{i=1}^{n}$ be a union of singular points and separatrix connections of $v$: $\gamma_i$ is  a separatrix connection between $P_i$ and $P_{i+1}$. Then $Z$ either contains $C$, or does not intersect it.

2) For each $\alpha$- or $\omega$-limit set $c$ of $v$, the set $Z \cap c$ is either empty, or coincides with $c$.
\end{proposition}
In particular, 1) applies to any chain of separatrices (see Definition \ref{def-chain}) if we remove $\alpha(\gamma_0), \gamma_0, \gamma_n$, and $\omega(\gamma_n)$ from the chain.
\begin{proof}
1) Suppose that $Z$ contains $P_i\in C$.

A separatrix connection is both a stable and an unstable separatrix; due to Definition \ref{def-Sep-pr} of Sep-property, if $\gamma_i$ does not belong to $Z$, then both its $\alpha$- and $\omega$-limit sets $P_i, P_{i+1}$ are detached from $Z$. So for $i>1$, $P_i\in Z$ implies $\gamma_{i-1}\subset Z$, and due to closedness, $P_{i-1}\in Z$.  Similarly, for $i \not = n$, $P_i\in Z$ implies $\gamma_{i}\subset Z$ and  $P_{i+1}\in Z$. The induction in $i$ proves the statement.

2) If $c$ is a singular point or a cycle, this clearly follows from $v$-invariance of $Z$. If $c$ is a monodromic polycycle, then the statement follows from 1).
\end{proof}

\begin{figure}
\begin{center}
 \includegraphics[width=\textwidth]{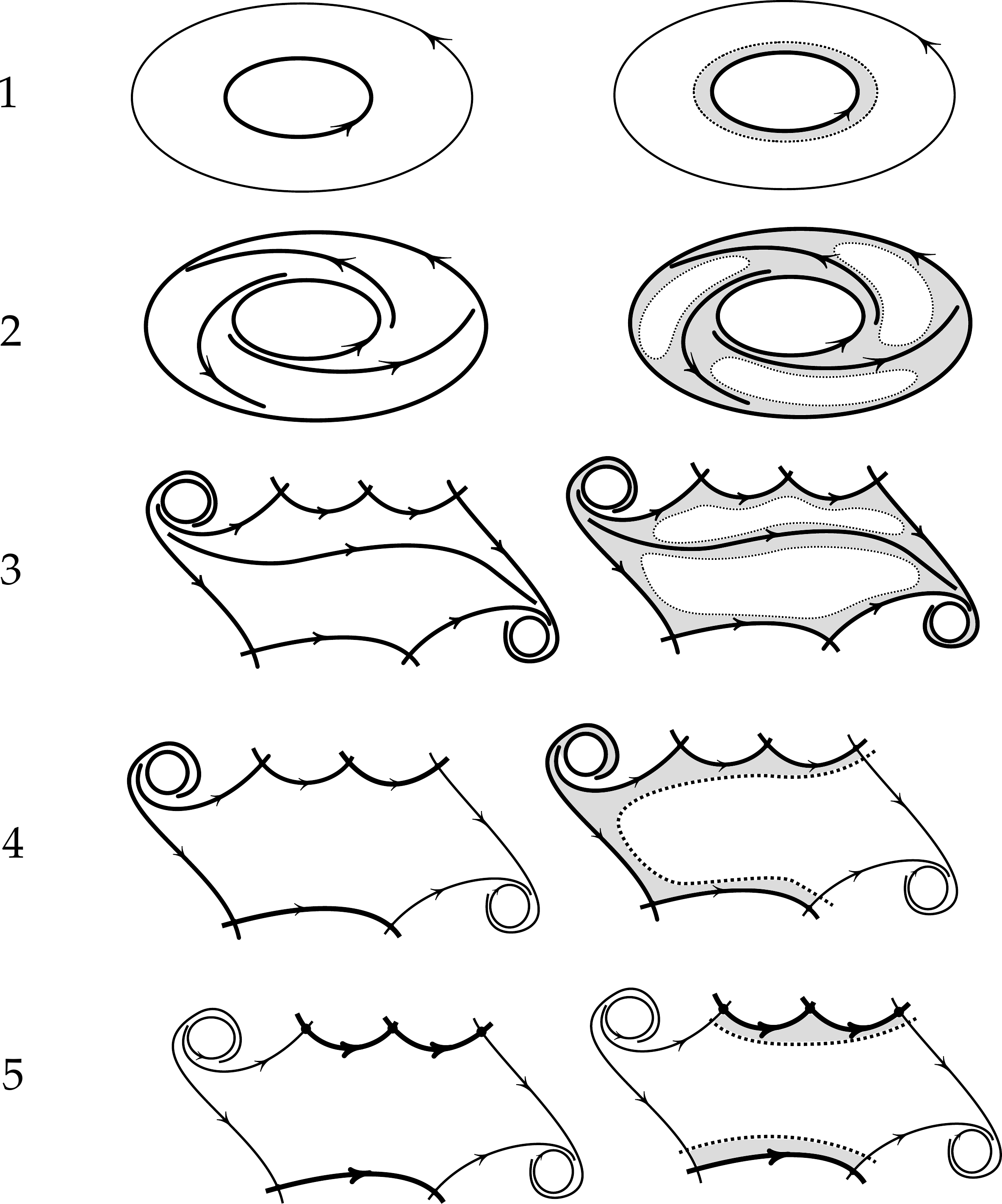}
\end{center}
 \caption{Intersections $Z\cap \overline R$ (left; shown in thick) and $\Omega\cap \overline R$ (right; $\partial \Omega$ is dotted) in all 5 possible cases}\label{fig:5cases}
\end{figure}

\subsection{Plan of the proof of the Boundary lemma}
In order to construct the required neighborhood $\Omega\supset Z$, we describe its intersection with each $\overline R$, where $R$ is a canonical region of $v$:
$$
\Omega_R = \Omega \cap \overline R.
$$
Note that the number of canonical regions for $v\in Vect^*\,S^2$ is finite, because $v\in Vect^*\,S^2$ has a finite number of limit cycles, separatrices and singular points.
So we set

\begin{equation}\label{eqn:ome}
\Omega = \cup \Omega_R.
\end{equation}

If $Z\cap \overline R$ is empty, then $\Omega \cap \overline R$ will be empty; we do not discuss this case any more.

Depending on the type of $R$ and the type of intersection $Z \cap \overline R$, we have the following five cases for $R$. If $R$ is a spiral canonical region, either \textbf{1)} $Z\cap R = \varnothing$ or \textbf{2)} $Z\cap R \neq \varnothing$.  If $R$ is a strip canonical region, either \textbf{3)} $Z\cap \overline R$ contains both $\alpha(R)$ and $\omega(R)$, or \textbf{4)} one of them, or \textbf{5)} none of them (there are no other cases due to Proposition \ref{prop-Z-inters-R} part 2).

The first three cases give rise to Type $1$ and Type $3$ boundary components of $\Omega$ that belong to $R$ entirely. These components are constructed  in Section \ref{sssec-123}.

The last two cases 4), 5) give rise to Type $2$ boundary components of $\Omega$ that belong to the union of several adjacent canonical regions. These components are constructed   in Section \ref{sssec-45}.

Then we define the set $\Omega$ by \eqref{eqn:ome}, and prove that it has the required properties.

\subsection{Construction of $\Omega \cap R$  in  the cases 1), 2), 3)}\label{sub:123}

We have to construct an ``arbitrary small'' neighborhood of $Z$ with certain properties. This means that it must belong to a preassigned neighborhood $\Omega_0$ of $Z$. From now on, this latter neighborhood  is fixed.

\label{sssec-123}
\begin{itemize}
\item 1): $R$ is a spiral canonical region, $ {Z\cap R = \varnothing}$. Due to Proposition \ref{prop-Z-inters-R} part 2, $Z \cap \overline R$ is $\alpha(R)$, $\omega(R)$,  or $\alpha(R) \cup \omega(R)$.
\end{itemize}
Take  $\Omega_R:= \Omega\cap \overline R$ to be a thin strip around $\alpha(R)$ or $\omega(R)$ (or two strips around both) bounded by its smooth transversal loop. This yields one or two Type $3$ boundary components, see row 1 of Figure \ref{fig:5cases}.

Complete semi-trajectories of points of such boundary components under $v|_\Omega$ stay in $\Omega$, because these trajectories wind around $\alpha(R)$ or $\omega(R)$ respectively. Clearly, separatrices of $v|_{S^2\setminus \Omega}$ do not enter $\Omega_R$. The set $\partial \Omega_R \cap R$ consists of one or two topological circles. They are boundary components of Type $3$.

\begin{itemize}
\item 2): $R$ is a spiral canonical region, $ {Z\cap R \neq  \varnothing}$, see row 2 of Figure \ref{fig:5cases}.
\end{itemize}
In the case 2), $Z$ contains a trajectory of $v|_R$. So it contains $\alpha$- and $\omega$-limit sets of this trajectory, i.e. $\alpha(R)$ and $\omega(R)$. Therefore $\overline R\setminus Z$ is a union of at most countably many open strips with parallel flows in them.
For each such strip $S$,
\begin{itemize}
\item[-] If $S \subset \Omega_0$, we  include it completely in $\Omega$.

\item[-] Otherwise, let $D \subset S $ be a large ellipse in the rectifying chart for $v$ in $S$ such that $S\setminus D \subset \Omega_0$, and let $\partial D$ have two quadratic tangency points with the vector field. Let $S\setminus D =: \Omega \cap S$.   

This yields a finite number of Type $1$ boundary components.

Complete semi-trajectories of points of such boundary component under $v|_\Omega$ stay in $\Omega$ because they stay in $S$.
\end{itemize}
Clearly, separatrices of $v|_{S^2\setminus \Omega}$ do not enter $\Omega_R$. Again, the set $\partial \Omega_R \cap R$ consists of a finite number of topological circles. They are boundary components of Type $1$.
%

\begin{itemize}
\item 3): $R$ is a strip canonical region, $Z \cap \overline R$ contains $\alpha(R)$ and $\omega(R)$, see row 3 of Figure \ref{fig:5cases}.
\end{itemize}
Due to Sep-property, $Z$ also contains the first and the last separatrices of $\nu_1(R)$, $\nu_2(R)$; due to closedness, $Z$ contains endpoints of these separatrices.

Note that $\nu_i(R)$ without the first and the last separatrix and without $\alpha(R), \omega(R)$ is a chain that satisfies Proposition \ref{prop-Z-inters-R} part 1. So $Z$ contains the whole $\nu_1(R)$ and $\nu_2(R)$. It can also contain several trajectories of $v|_{R}$. So  this case is analogous to case 2) and yields a finite number of Type $1$ boundary components.

\subsection{Construction of $\Omega\cap R$ in the cases 4), 5).}
\label{sssec-45}
First, on the whole sphere, we choose marked points on all separatrices of $v$ that "leave" a neighborhood of $Z$. In more detail, suppose that for a singular point $P\in Z$, its separatrix $\gamma$ does not belong to $Z$. Then some arc of $\gamma$ starting at $P$ belongs to $\Omega_0$.  Fix one point on this arc; this point will be called marked. We will use marked points later in the construction; namely, $\partial \Omega$ will intersect $\gamma$ at the marked point.

In the cases 4) and 5), $Z\cap R$ is empty; otherwise $Z$ would contain both $\alpha$ and $\omega$-limit set of a trajectory of $v|_{R}$, thus satisfy assumptions of case 3) above.

\begin{itemize}
\item 4): $Z\cap R =\varnothing$,  $\alpha(R)\subset Z$, and $\omega(R)$ does not intersect $Z$ (or vice versa: $\alpha$ and $\omega$ are exchanged), see row 4 of Figure \ref{fig:5cases}.
\end{itemize}
Due to Sep-property, $Z$ contains the first separatrix of $\nu_1(R)$, $\nu_2(R)$; due to closedness, $Z$ contains endpoints of these separatrices. Now, due to Proposition \ref{prop-Z-inters-R} part 1, $Z$ contains all $\nu_1(R), \nu_2(R)$ except their last separatrices and $\omega(R)$. $Z$ cannot contain last separatrices of $\nu_1(R), \nu_2(R)$, because it does not contain their $\omega$-limit set $\omega(R)$.

Finally, $Z\cap \overline R$ is the union of $\alpha(R)$ and $\nu_{1,2}(R)$ except for their last separatrices and  $\omega(R)$.  Note that last separatrices of $\nu_{1,2}(R)$ have marked points on them.

Take $\Omega_R \subset \Omega_0$ to be a  neighborhood of $Z\cap \overline R$ in $\overline R$ bounded by a smooth curve $\varphi(R)\subset \overline R$ that is  transversal to $v$ and connects marked points of last separatrices of $\nu_1(R)$ and $\nu_2(R)$. Take $\varphi(R)$ to be orthogonal to the corresponding separatrices at marked points.
The existence of $\varphi(R)$ follows from the fact that $R$ is parallel. The endpoints of $\varphi(R)$ may coincide, then it is  a topological circle (see Fig. \ref{fig:BL-4}); otherwise $\varphi(R)$ is a topological segment (see Fig. \ref{fig:5cases}, row 4).
After we put $\Omega=\cup \Omega_R$ at the end of the proof, we will have that in the first case, $\varphi (R)$ is a transversal Type $2$ boundary component, and in the second case, it is a part of Type $2$ boundary component, namely a  transversal subarc in $\partial \Omega $ crossed by separatrices of $v$ at its endpoints.

\begin{figure}
\begin{center}
\includegraphics[width=0.3\textwidth]{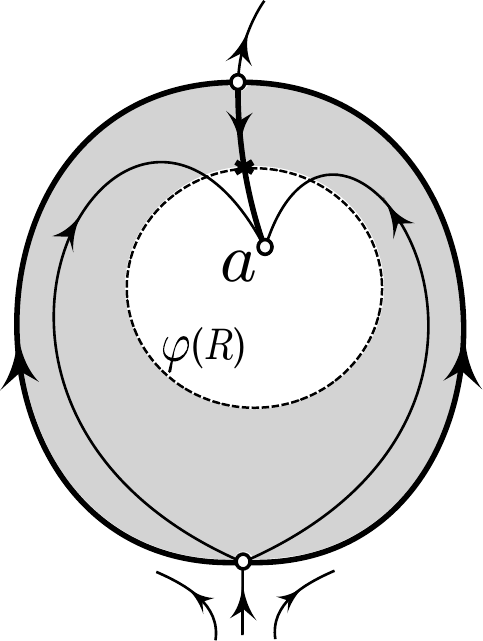}
\end{center}
\caption{Case 4) in the proof of  Boundary lemma; $\varphi(R)$  is a topological circle. The domain $\Omega_R$ is shadowed}\label{fig:BL-4}
\end{figure}

\begin{itemize}
\item 5): $Z\cap R =\varnothing$,  both $\alpha(R)$ and $\omega(R)$ do not intersect $Z$, see row 5 of Figure \ref{fig:5cases}.
\end{itemize}
If $\nu_1(R)$ intersects $Z$, then $Z$ contains the whole $\nu_1(R) $ except for its first and last separatrices, and  $\alpha(R), \omega(R)$, due to Proposition \ref{prop-Z-inters-R} part 1. Note that both the first and the last separatrix of $\nu_1(R)$ have marked points on them.

\begin{figure}[h]
 \begin{center}
\centering
\subcaptionbox{\label{fig-BLemma-5-0}}{  \includegraphics[width=0.3\textwidth]{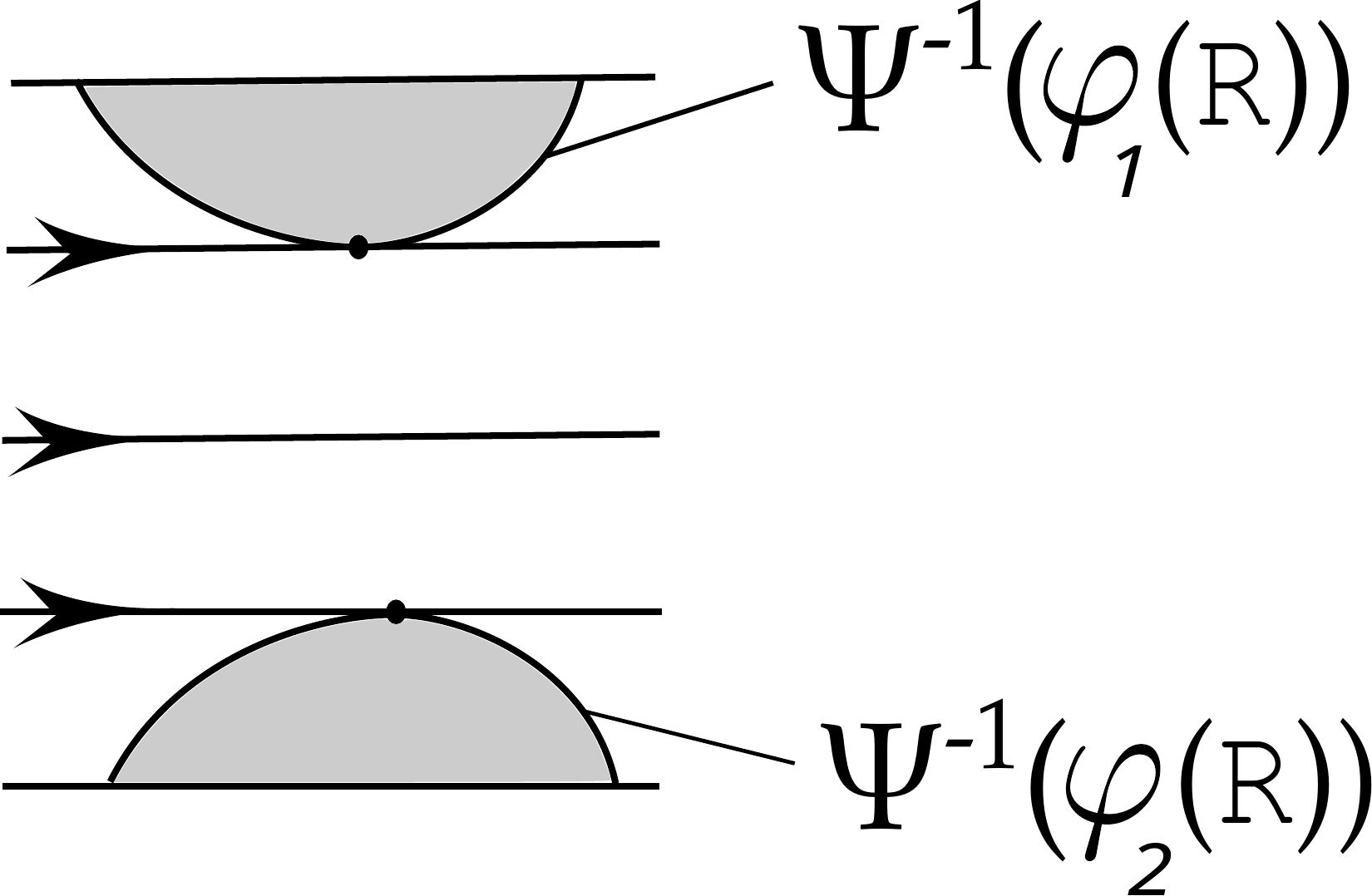} }
 \hfil
\subcaptionbox{\label{fig-BLemma-5-1}}{  \includegraphics[width=0.25\textwidth]{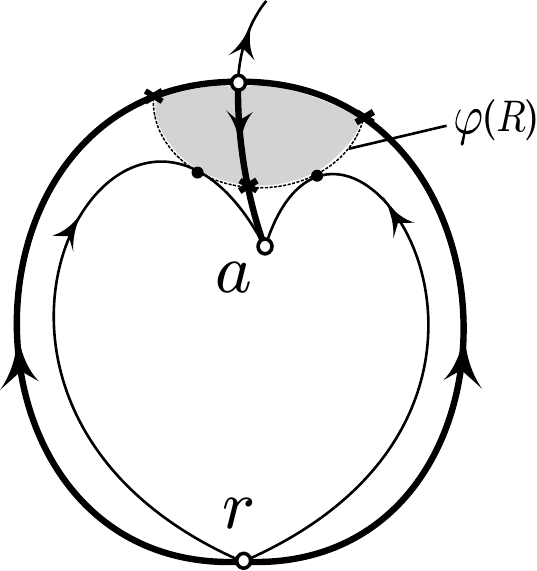} }
 \hfil
\subcaptionbox{\label{fig-BLemma-5-2}}{   \includegraphics[width=0.25\textwidth]{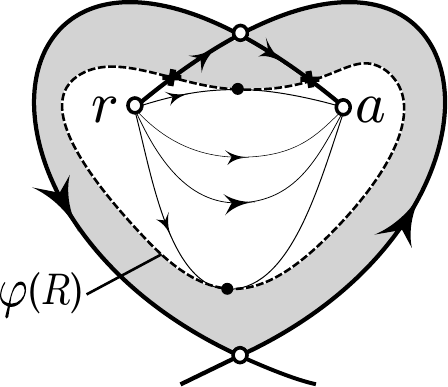}}
\caption{Case 5) in the proof of Boundary lemma. (a) shows $\Psi^{-1}(\varphi(R))$ and the domain $\Psi^{-1} (\Omega)$ (shadowed), (b) and (c) show that $\varphi(R)$ can be a topological segment and a topological circle respectively. The domain $\Omega_R$ is shadowed}\label{fig:BL-5}
\end{center}
\end{figure}

Take a smooth curve $\varphi_1(R) \subset R$ with the following properties: $\varphi_1(R)$ connects the marked points on the first and last separatrices of $\nu_1(R)$,  is close to $Z\cap \nu_1(R)$, is perpendicular to the first and the last separatrix at its endpoints, and has one quadratic tangency point with $v$. It is easy to construct an appropriate curve in the rectifying chart, i.e. in $\bbR\times [0,1]$ (see Fig. \ref{fig-BLemma-5-0}); let $\varphi_1(R)$ be its image under $\Psi$.

If $\nu_1(R)$ intersects $Z$ and $\nu_2(R)$ does not, we put $\varphi(R):=\varphi_1(R)$, and $\Omega_R$ is bounded by $\varphi_1(R)$ and an arc of $\nu_1(R)$.  
If $\nu_2(R)$ also intersects $Z$, we choose the curve $\varphi_2(R)$ in a similar way, and put $\varphi(R):=\varphi_1(R)\cup \varphi_2(R)$. Then $\Omega_R$  is the union of two domains, one between $\varphi_1(R)$ and $\nu_1(R)$ and the other one between $\varphi_2(R)$ and $\nu_2(R)$; see Figure \ref{fig:5cases} row 5.

Note that the two curves $\varphi_1(R)$ and $\varphi_2(R)$ may have one or two common endpoints, see Fig. \ref{fig-BLemma-5-1}, \ref{fig-BLemma-5-2} respectively. So $\varphi(R)$ can be either two smooth curves with one contact point on each, or one simple curve with two contact points, or a closed loop with two contact points.
In any case,  $\varphi(R)$ will be a part of a Type $2$ boundary component; in the third case, it is the whole Type $2$ boundary component with two contact points.

\begin{remark}
\label{rem-trajectory-case5}
 Under assumptions of Boundary lemma, let $R$ be a canonical region satisfying assumptions of case 5) above. Then for any small neighborhood $\Omega$ of $Z$,  $R$ contains a trajectory of $v$ that does not intersect $\overline{\Omega}$ (see Fig. \ref{fig-BLemma-5-0}).
\end{remark}
We will use this remark in the next section.

\subsection{End of the proof of the Boundary lemma }

We have constructed an intersection $\Omega_R$ of the neghbourhood $\Omega $ with the closure of any canonical domain $R$. Now take $\Omega $ to be the union of all $\Omega_R$. This is a neghbourhood of $Z$ that belongs to $\Omega_0$. Let us prove that its boundary components satisfy the Boundary Lemma.

By construction, $\partial \Omega $ is a $C^1$-smooth one-dimensional compact submanifold of the sphere. Hence it is a finite union of topological circles. Consider an arbitrary connected component $\varphi $ of $\partial \Omega $.

If $\varphi $ intersects a canonical region $R$  of case  1), 2), or 3),  then it belongs entirely to $R$ and is of Type $1$ or $3$ as proved in Section \ref{sub:123}.

\begin{figure}[h]
 \begin{center}
  \includegraphics[width=0.4\textwidth]{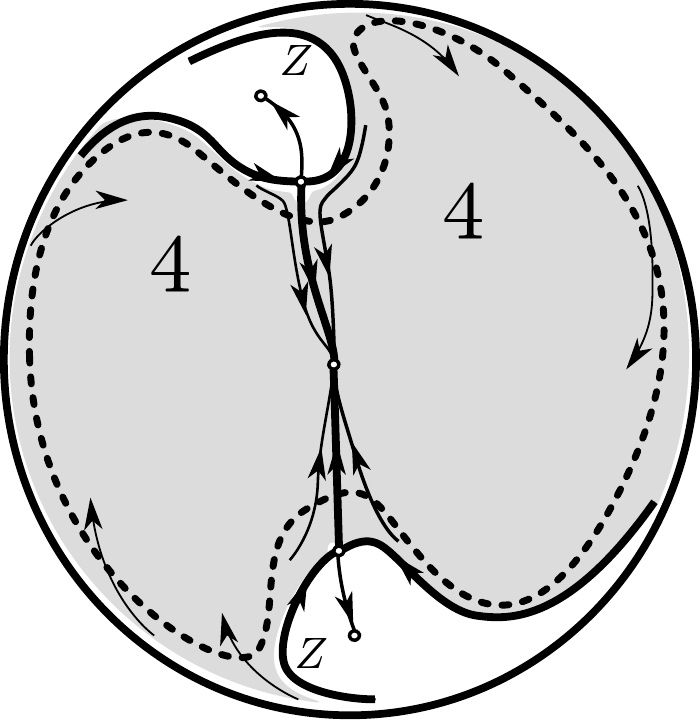}
 \end{center}
\caption{Boundary component of Type $2$ without contact (dashed) in the union of two canonical regions of case 4) (shadowed). Boundaries of canonical regions are shown in thick}\label{fig:CanonReg4}
\end{figure}

Suppose that $\varphi $ intersects canonical regions of case 4) only. Then it has no contacts with $v$, see Fig. \ref{fig:CanonReg4}, i.e. is transversal. Assume that it is outgoing. All future semi-orbits of $v$ that start on $\varphi $ do not intersect $\Omega $ and have the same $\omega$-limit set, which is clear for the orbits located inside each canonical region of case 4). Let $R$ be any canonical region that contains a subarc of  $\varphi $; then $\varphi$ intersects the first or the last separatrices in boundary chains of $R$, so $\varphi$ intersects at least one separatrix of $v|_{\Omega}$.
Hence $\varphi $ is a boundary component of Type $2$ transversal to $v$.

\begin{figure}[h]
 \begin{center}
  \includegraphics[width=0.4\textwidth]{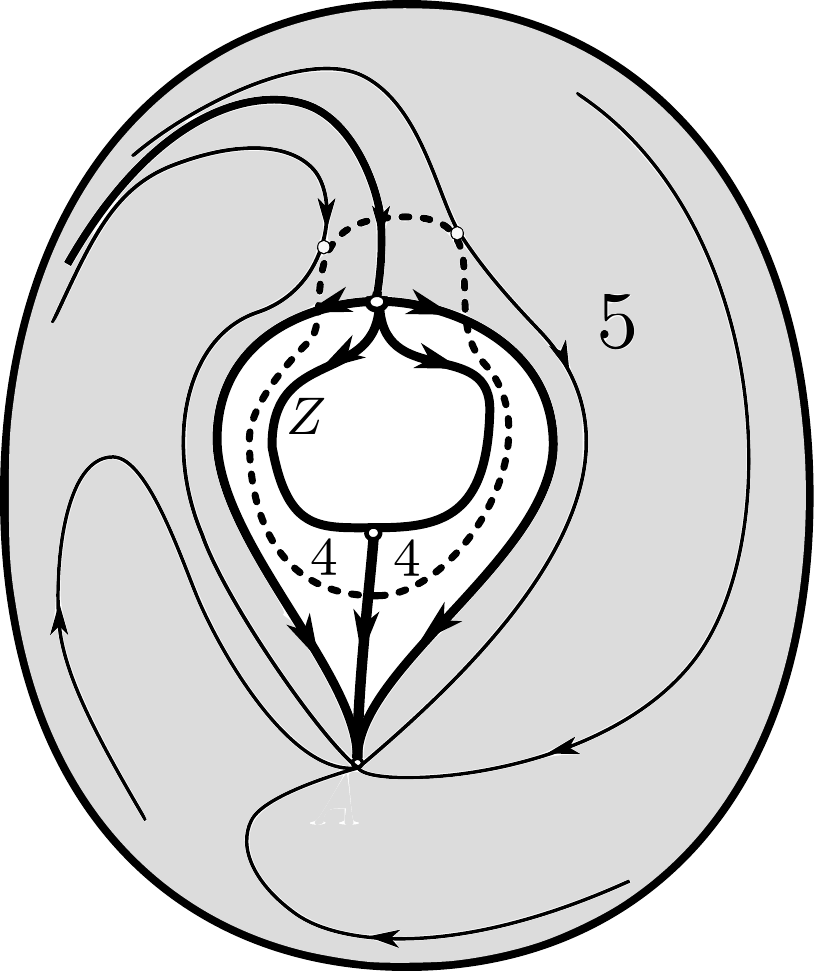}
 \end{center}
\caption{Boundary component of Type $2$  (dashed)  with two contact points in the union of one canonical region of case 5 (shadowed) and two canonical regions of case 4. Boundaries of canonical regions are shown in thick}\label{fig-CR-inters-1}
\end{figure}

Suppose that $\varphi $ intersects at least one canonical region of case 5), see Fig. \ref{fig-CR-inters-1}. Then it has at least one point of outer quadratic tangency with $v$ (thus it has at least two tangency points). Let $\beta $ be a transversal arc of $\varphi $ between two such points. Then
$$\beta = \beta'\cup \bigcup_{i=1}^{k-1} \varphi(R_i) \cup \beta''$$
where $R_j$ are regions of case 4) and $\beta' , \beta''$ are subarcs of $\varphi (R_0), \varphi(R_k)$; here $R_0, R_k$ are canonical regions of case 5). Subarcs $\beta' , \beta''$ contain the endponts of $\beta $.

As before, all the orbits of $v|_{\Omega\cap {R_j}}$ that start at $\varphi (R_j)$ or at $\beta', \beta''$, stay in $\Omega$ and  have the same $\omega$-limit set. Since the arcs $\varphi (R_i), \varphi (R_{i+1})$ have common  endpoints, this holds for the whole arc $\beta $ too.

Finally, $\beta $ is crossed by the first or the last separatrices  of the boundary chains of the corresponding canonical domains. They are separatrices of $v|_{\Omega_R}$, due to the description of boundaries of canonical regions.

Hence $\varphi$ is a boundary component of type 2, with at least two outer tangency points with $v$.
This completes the proof of the Boundary lemma.

\section{Images of  boundary components of $U$} \label{sec-Case123}

Here we prove Lemma \ref{lem-image-case1}, Lemma \ref{lem-image-Case2} and Lemma \ref{lem-image-case3}.
As before,  we assume that $U$, $\tilde U^{\pm}$ are chosen as in Proposition \ref{prop-choice} (recall that this proposition only uses Boundary lemma, and this lemma is already established).

\subsection{Canonical regions for vector fields in open domains on the sphere}

We will need the generalizations of Propositions \ref{prop-canonreg-common-limset}, \ref{prop-parall} to the case of vector fields on subdomains of $S^2$.

Take an open set $D \subset S^2$ such that $\partial D$ is a union of finitely many continuous curves homeomorphic to $S^1$ and having finitely many topological tangencies with  $v$. We assume that singular points, limit cycles and monodromic polycycles that belong to $\overline D$ also belong to   $D$.
\begin{definition}[Canonical regions in domains]
\label{def-canonreg-subdomain}
 For a vector field $v \in Vect^*\,S^2$ and an open set $D\subset S^2$ as above, let $S(v, D)$ be the union of all singular points, separatrices and limit cycles of $v|_D$, and let $Tang(v,D)$ be the union of trajectories under $v|_{D}$ of topological tangency points of $v$ with $\partial D$. A \emph{canonical region} of $v|_{D}$ is a  connected component of $D\setminus (S(v,D)\cup Tang(v,D))$.
 \end{definition}

 \begin{proposition}
 \label{prop-canonreg-limset-subdomain}
  For a vector field $v \in Vect^*\,S^2$ and an open set $D\subset S^2$ as above, all points of the same canonical region of $v|_{D}$
  \begin{itemize}
   \item either have the same $\omega$-limit set under $v$ inside $D$, and their future semi-trajectories stay in $D$;
   \item or their future  semi-trajectories under $v|_{D}$ terminate on the same connected component of $\partial D$.
  \end{itemize}
  The same alternative holds for $\alpha$-limit sets and past semi-trajectories.

 \end{proposition}
\begin{proof}
Let $R$ be a canonical region of $v|_{D}$. Consider a set $G$ of points in $R$ such that their trajectories stay in $D$ and  have one and  the same $\omega$-limit set $A$ under $v$. The set $A$ is inside $\overline D$, thus inside $D$, due to our assumptions on $D$.

The set $G$ is open; the proof is similar to that in Proposition \ref{prop-canonreg-common-limset}. The only new argument to be added is, that if the trajectory of a point stays in $D$ and has its $\omega$-limit set inside $D$, then the
trajectories of close points also stay in $D$.

Now, consider a set $G$ of points in $R$ such that their future semi-trajectories under $v|_{D}$ terminate on one and the same connected component of $\partial D$. We will prove that the set $G$ is also open. Let $x \in G$,  and $y\in \partial D$ be the endpoint of its future semi-trajectory under $v|_{D}$.  In a sufficiently small flow-box around $y$, $\partial D$ is a continuous curve that intersects all trajectories of $v$; this follows from the fact that $\partial D$ has only finitely many tangencies with $v$ and $y$ is not an inner tangency point itself.

Now it suffices to notice that each future semi-trajectory of $v$ that starts near $x$ eventually reaches the flow-box of $y$, thus intersects the same connected component of $\partial D$.

Finally, since $R$ cannot be a union of several open disjoint sets, it coincides with one of the sets above: either all its points have the same  $\omega$-limit set under $v$ inside $D$, and their future semi-trajectories stay in $D$; or their future semi-trajectories  terminate on the same connected component of $\partial D$.
\end{proof}

 \begin{proposition}
\label{prop-canonreg-parall-subdomain}
  For a vector field $v \in Vect^*\,S^2$ and an open set $D\subset S^2$ as above,  each canonical region of $v|_{D}$ is parallel, i.e. equivalent to a strip flow or a spiral flow.
 \end{proposition}
\begin{proof}
 The proof is the same as for the case of $D=S^2$, see \cite[Proposition 1.42, p. 34]{DumLlibArt} for omitted details. Namely, the quotient space obtained by collapsing orbits of $v|_{R}$ into points is a (Hausdorff) connected one dimensional manifold (i.e. $S^1$ or $\bbR$), and
the natural projection of $R$ to this quotient space is a locally trivial fibering. So it can be homeomorphic to $\bbR \times \bbR \to \bbR$ (then we have a strip flow), or $S^1 \times \bbR \to \bbR$ (spiral flow), or
$\bbR \times S^1 \to \bbR$ (annular flow. However in this case, $v$ has infinitely many periodic orbits which is impossible for $v\in Vect^*\, S^2$).
\end{proof}

\subsection{Images of Type $2$ boundary components}

The following proposition is the main part of the proof of Lemma \ref{lem-image-Case2}.

\begin{figure}[ht]
 \begin{center}
  \includegraphics[width=0.3\textwidth]{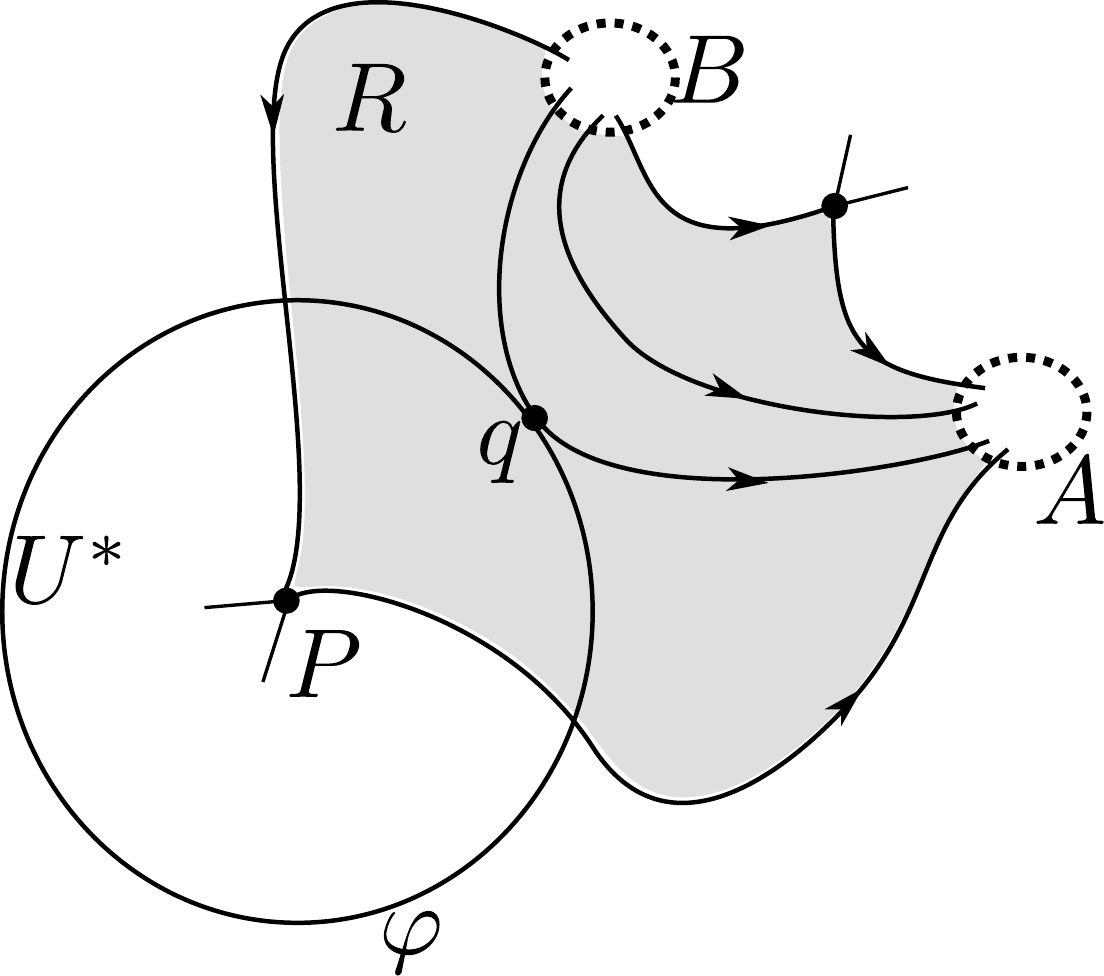}
 \end{center}
\caption{Canonical region of an outer topological tangency point of $v_0$. The sets
$A,B$ are located inside the domains with the dotted boundaries, and not shown
on the figure.}\label{fig-canonreg-of-p}
\end{figure}

 \begin{figure}[ht]
 \begin{center}
\centering
\subcaptionbox{\label{fig-canonreg-of-p-eps-1}}{  \includegraphics[width=0.35\textwidth]{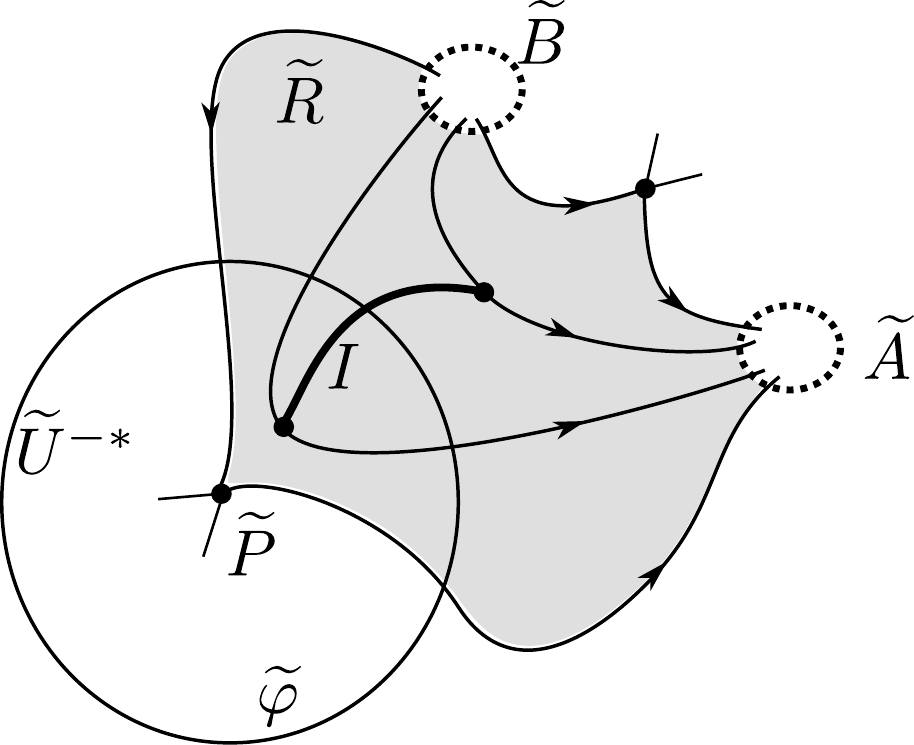} }
 \hfil
\subcaptionbox{\label{fig-canonreg-of-p-eps-2}}{   \includegraphics[width=0.35\textwidth]{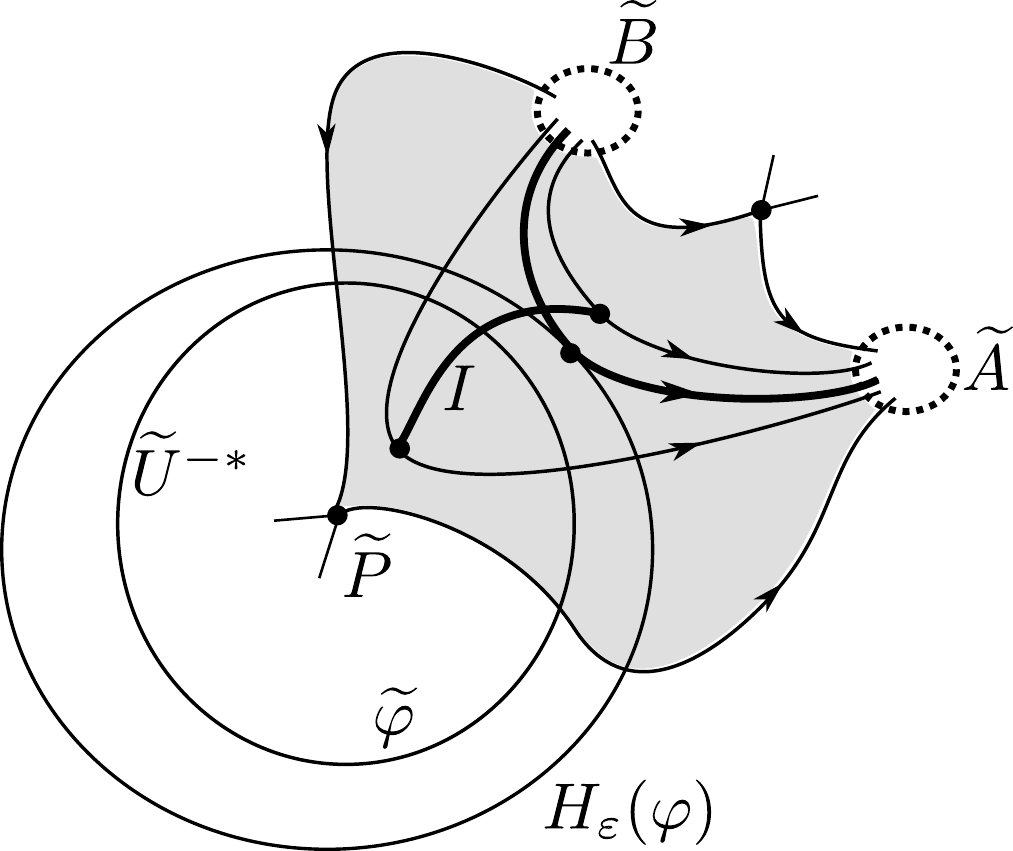}}
\caption{Canonical region of outer tangency points a) for $w_0$ b) for $w_{h(\eps)}$. On this figure, the objects with tilde are the  images under $\hat H$ of the corresponding objects without tilde. The sets $\tilde A, \tilde B$ are located inside the domains with the dotted boundaries, and not shown
on the figure.}\label{fig-canonreg-of-p-eps}
\end{center}
\end{figure}

 \begin{proposition}
 \label{prop-outer-cont-pt-w}
    Under assumptions of the Main theorem and Proposition \ref{prop-choice}, for sufficiently small $\eps$, for each outer topological  tangency point of  $\partial H_{\eps}(U)$ with $w_{h(\eps)}$, its trajectory under $w_{h(\eps)}$ belongs to $S^2\setminus H_{\eps}(U^*)$.

\end{proposition}
\begin{proof}

Let $q$ be an outer tangency point of a boundary component $\varphi \subset \partial U^*$. Let $\beta\subset \varphi$ be an outgoing transversal arc with the endpoint $q$. Consider a canonical region $R$ of $v_0$ that contains $q$. Recall that   $A :=\omega_{v_0}(q)$ and $B:=\alpha_{v_0}(q)$ do not intersect $LBS^*(V)$, and $\beta $ intersects separatrices of $v_{0}|_{U}$, due to  the Boundary lemma. So $R$ cannot contain the whole arc $\beta$: $\partial R$ contains an intersection of a separatrix of $v_{0}|_{U}$ with $\beta$. Thus $\partial R$ contains a singular point $P \in LBS^*(V)$, the $\alpha$-limit set of this separatrix; $P\in LBS^*(V)$ due to No-entrance lemma. Finally, $R$ is of case 5) according to the classification introduced in  the proof of the Boundary lemma: $\alpha(R)$ and $\omega(R)$ do not intersect  $LBS^*(V)$, and $\partial R$ contains a point  $P\in LBS^*(V)$ on its boundary (see Fig. \ref{fig-canonreg-of-p}).

Similarly, if $R$ contains two outer tangency points with $U^*$, then it contains points  $P_1, P_2 \in LBS^*(V)$ on both its side boundaries $\nu_1(R), \nu_2(R)$.

Now, $\tilde R = \hat H(R)$ is a canonical region for $w_0$, its $\alpha$-, $\omega$-limit sets $\tilde A :=\hat H(A)$, $\tilde B :=\hat H(B)$ do not belong to $LBS^*(W)$, and it has a singular point $\tilde P := \hat H(P)\in LBS^*(W)$ on its boundary, see Fig. \ref{fig-canonreg-of-p-eps-1}. It also contains trajectories that do not intersect $\tilde U^{*+}$ (see Remark \ref{rem-trajectory-case5}; it is applicable because $\tilde U^{*+}$ satisfies Boundary lemma, due to Proposition \ref{prop-choice}).

Take an arc $I\subset \tilde R$ transversal to $w_0$, such that the trajectory of one of its endpoint under $w_0$ does not visit $\tilde U^{+*}$ (this is possible due to Remark \ref{rem-trajectory-case5}), and the trajectory of another its endpoint is close to $\hat H(P)$, so visits $\tilde U^{-*}$, see Fig. \ref{fig-canonreg-of-p-eps-1}. Both trajectories connect $\hat H(A)$ to $\hat H(B)$. The same holds for the trajectories of the endpoints of $I$ under $w_{h(\eps)}$ with small $\eps$.
 Namely, both of them connect transversal loop around $\hat H(A)$ and $\hat H(B)$; one of them intersects $H_\eps(\varphi)$, and the other does not.
 Due to the continuity of orbits of $w_{h(\eps)}$ with respect to the initial conditions, and the fact that $\tilde U^{-*}\subset H_{\eps}(U^*) \subset \tilde U^{*+}$, there exists a point in $I$ whose trajectory under $w_{h(\eps)}$ visits  $\overline{H_{\eps}(U^*)}$ and does not visit $H_{\eps}(U^*)$. This trajectory contains one or several topological tangency points  of $H_{\eps}(U)$ with $w_{h(\eps)}$ (see Fig. \ref{fig-canonreg-of-p-eps-2}).

Finally, if $R$ contains a point of outer tangency of $\partial U$ with $v_0$, then $\hat H(R)$ contains at least one point of outer tangency of $H_{\eps}(U)$ with $w_{h(\eps)}$.
If $R$ contains two points of outer tangency of $U$ (i.e. points of $LBS(V)$ on both boundaries), then the same construction yields at least two  tangency points in $\hat H(R)$.

 Let the total number of outer tangency points of $\partial U$ with $v_0$ be $N$. The construction above yields at least $N$ topological outer tangency points of $H_{\eps}(U)$ with $w_{h(\eps)}$, and their trajectories under $w_{h(\eps)}$ do not visit $H_{\eps}(U)$.

 On the other hand, the total number of outer tangency points of $H_{\eps}(U)$ with $w_{h(\eps)}$ is $N$, because $H_{\eps}$ identifies outer tangency points of $U$ and $H_{\eps}(U)$. So we have found all of them. Hence  trajectories of all outer tangency points of $H_{\eps}(U)$ with $w_{h(\eps)}$ stay in $S^2\setminus \overline {H_{\eps}(U^*)}$.

\end{proof}

\begin{proof}[Proof of Lemma \ref{lem-image-Case2}]

Let $\beta, \ \tilde \beta_\eps$ and $\tilde l$ be the same as in Lemma \ref{lem-image-Case2}. Obviously $\tilde \beta_\eps$ is topologically transversal to $w_{h(\eps)}$. Trajectories of its points under $w_{h(\eps)}|_{S^2\setminus H_\eps(U^*)} $ are not trajectories of outer tangency points due to Proposition \ref{prop-outer-cont-pt-w},
and none of them are separatrices of $w_{h(\eps)}|_{S^2\setminus H_{\eps}(U)}$, due to No-entrance  lemma \ref{lem-seps} (applied to $\tilde U^+\supset H_{\eps}(U)$ and the family $W$). Due to Definition \ref{def-canonreg-subdomain} of canonical regions in domains, $\tilde \beta_\eps$ belongs to one canonical region of $w_{h(\eps)}|_{S^2\setminus H_{\eps}(U^*)}$.

Let us prove that all trajectories of this canonical region cross $\tilde l$; this will imply that the Poincare map $\tilde P_{\eps}$ is defined. Due to Proposition \ref{prop-canonreg-limset-subdomain}, it is sufficient to prove this statement for one trajectory. By assumption of Lemma \ref{lem-image-Case2}, $\beta$ is a maximal transversal arc of a boundary component of $U$ of Type $2$. Hence, there exists a separatrix $\gamma$ of $v_0$ that crosses $\beta$. Let $r = \gamma \cap \beta$.
\begin{proposition}
\label{prop:sep-pin}
Let $\gamma, \beta $, and $r$ be the same as above. Then the trajectory of $H_{\eps}(r) \in \tilde \beta_\eps$ under $w_{h(\eps)}|_{S^2\setminus H_{\eps}(U^*)}$ crosses $\tilde l$ for small $\eps$, and the intersection point is close to $p:=\hat H(\gamma)\cap \tilde l$ for small $\eps$.
\end{proposition}
\begin{proof}
The future semi-trajectory of $\hat H(r)$ under $w_0$ (i.e. the part of the  separatrix $\hat H(\gamma)$) crosses $\tilde l$ at the point $p$; it does not visit $H_{\eps}(U)\subset \tilde U^{+*}$ due to No-entrance  lemma \ref{lem-seps} applied to the family $W$ and $\tilde U^{+*}$.

Since $\mathbf H$ is continuous on  $\Sep v_0$ (see Requirement \ref{it-H-contin} of Definition \ref{def-moderate-local} of moderate equivalence), the point  $H_{\eps}(r)$ is close to $\hat H(r)$. Since $w_{h(\eps)}$ is close to $w_0$, the trajectory of $H_{\eps}(r)$ under $w_{h(\eps)}$ intersects $\tilde l$ and does not visit $\tilde U^{+*}$. The intersection point is close to $p$.
\end{proof}
We showed that one trajectory starting at $\tilde \beta_\eps$ does not visit $\tilde U^{+*}$ and crosses $\tilde l$. Since $\tilde \beta_\eps$ belongs to one canonical region, the same holds for all trajectories of $\tilde \beta_\eps$ (Proposition \ref{prop-canonreg-limset-subdomain}), and the Poincare map $\tilde P_{\eps}$ is well-defined. Proposition \ref{prop:sep-pin} also implies that  $\tilde P_{\eps}(\tilde \beta_\eps)$ contains a point close to $p\in \tilde P_0(\tilde \beta_0)$. Since $p$ is the inner point of $\tilde P_0(\tilde \beta_0)$, we conclude that for small $\eps$, $\tilde P_{\eps}(\tilde \beta_\eps)$ intersects $\tilde P_0(\tilde \beta_0)$. This completes the proof.

\end{proof}

\subsection{Images of Type $1$ boundary components}

\begin{proof}[Proof of Lemma \ref{lem-image-case1}]

 Let  $\varphi$ be the same as in Lemma \ref{lem-image-case1}. Let $C$ be a connected component of $S^2 \setminus \overline{H_{\eps}(U)}$ adjacent to $H_{\eps}(\varphi)$. Clearly, $H_{\eps}(\varphi)$ is a union of two topologically transversal arcs to $w_{h(\eps)}$.
 The endpoints of these arcs are points of inner topological  tangency of $H_{\eps}(\varphi)$  with $w_{h(\eps)}$ (here "inner" means "inner with respect to $H_\eps(U)$").
  Let $p$ be one of these endpoints.

Consider the canonical region $R$ of $w_{h(\eps)}|_{C}$ that contains $p$ on its boundary. Let us prove that it contains $H_{\eps}(\varphi)$ and coincides with $C$. By contraposition, suppose that the intersection $R \cap H_{\eps}(\varphi)$ is a proper subarc of $H_{\eps}(\varphi)$.

Consider an endpoint $q$ of this arc. By  Definition \ref{def-canonreg-subdomain} of canonical regions, $q$
either belongs to an orbit tangent to $\partial C$, or belongs to a separatrix of $w_{h(\eps)}|_C$. The second option (see Fig. \ref{fig-case1-a})is impossible due to No-entrance lemma \ref{lem-seps} applied to $W$ and $\tilde U^+$. Prove that the first option is also impossible.

First, it is not possible that the point $q$ is itself a point of tangency of $w_{h(\varepsilon )}$ and $H_\varepsilon (\varphi )$  whose orbit locally belongs to $\overline C $ (see Fig. \ref{fig-case1-b}). Indeed,  all orbits of points of $\varphi $ under $v_\varepsilon $ enter $U$ either in the positive, or in the negative time. So do the
 orbits of the points of $H_\varepsilon (\varphi )$ under $w_{h(\varepsilon )}$. Hence, none of these orbits locally belong to $\overline C$.

\begin{figure}[h]
 \begin{center}
 \subcaptionbox{\label{fig-case1-a}}{\includegraphics[width=0.25\textwidth]{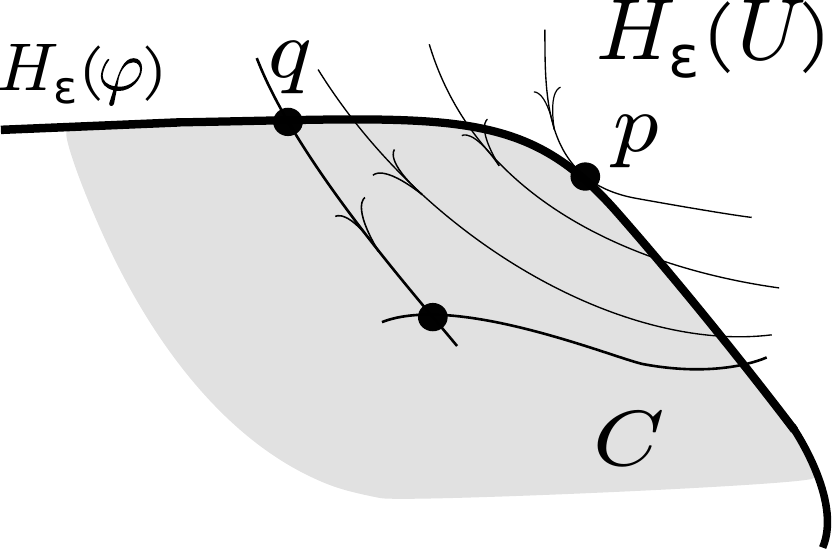}}
 \hfil
 \subcaptionbox{\label{fig-case1-b}}{\includegraphics[width=0.25\textwidth]{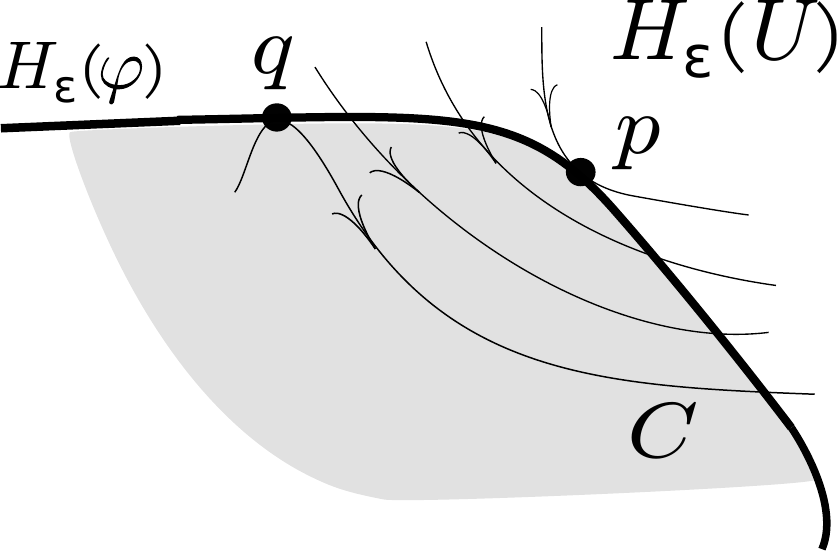}  }
 \hfil
 \subcaptionbox{\label{fig-case1-c}}{\includegraphics[width=0.25\textwidth]{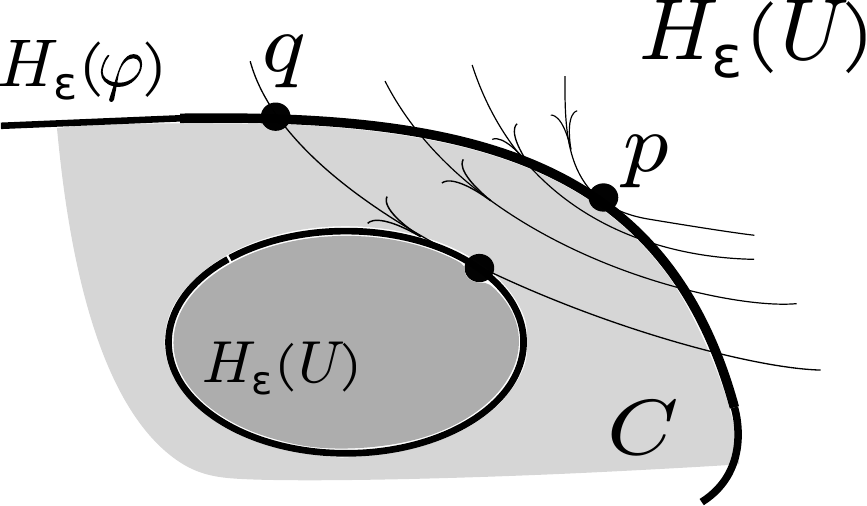}  }
 \end{center}
\caption{Images of Type $1$ components; all three pictures are impossible}\label{fig-case1}
\end{figure}

Second, it is not possible that the orbit of $q$ under $w_{h(\varepsilon )}|_{\overline C}$ contains another point of outer tangency with $\partial H_{\eps}(U)$ -- say, with some connected component of $H_\varepsilon (U)$ that belongs to  $C$, see Fig. \ref{fig-case1-c}. Here ``outer'' is understood as ``outer with respect to $H_{\eps}( U)$''. Indeed, by Proposition \ref{prop-outer-cont-pt-w}, the orbit of this outer tangency point must stay in the complement of $H_{\eps}( U)$, and thus cannot reach any point $q \in \partial H_{\eps}( U)$.

The contradiction obtained proves that the canonical region of $w_{h(\eps)}|_{C}$  that contains $p$ must contain the whole $H_{\eps}(\varphi)$, thus coincides with $C$. Due to Proposition \ref{prop-canonreg-parall-subdomain}, this canonical region is parallel. This completes the proof.
\end{proof}

\subsection{Images of Type $3$ boundary components}

\begin{proof}[Proof of Lemma \ref{lem-image-case3} (see Fig. \ref{fig-case3})]

 In what follows, the open annulus between two curves $\gamma_1, \gamma_2$ is denoted by $A(\gamma_1, \gamma_2)$. We also use this notation for the annulus between a polycycle or a singular point and its transversal loop.

\begin{figure}[h]
 \begin{center}
  \includegraphics[width=0.8\textwidth]{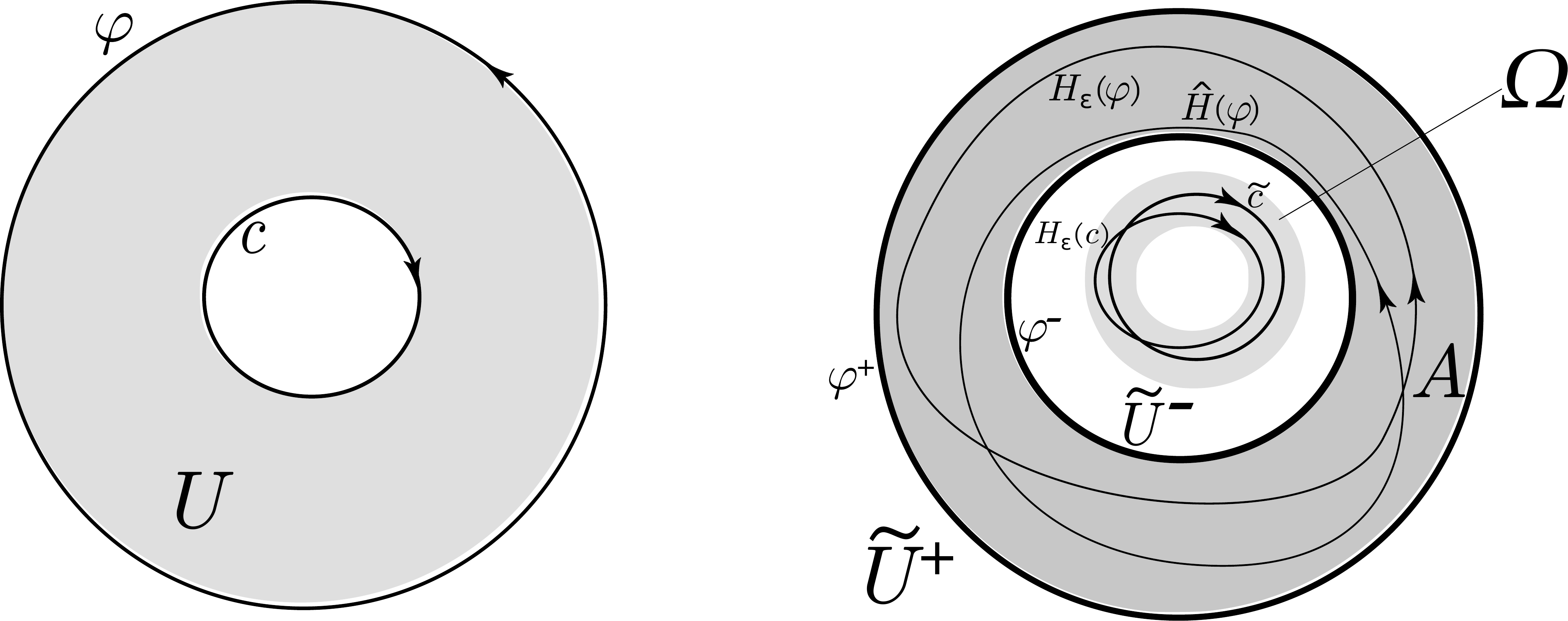}
 \end{center}
\caption{Images of Type $3$ components}\label{fig-case3}
\end{figure}

Recall that $\varphi$ is a boundary component of Type $3$, i.e. a transversal loop of some singular point, limit cycle, or polycycle $c$ of $v_0$. Consider the corresponding object  $\tilde c:= \hat H(c)$  of $w_0$.
Let $D$ be a disc bounded by  $\tilde c$  that contains $\hat H(\varphi)$.
Since $\tilde U^{+}, \tilde U^-$ satisfy Boundary lemma, they contain Type 3 boundary components $\varphi^{\pm}$ that correspond to $\tilde c$ and are in $D$. So  $\varphi^{\pm}$ are two transversal loops of $\tilde c$. Let $D^+ \ (D^-)$ be a disc bounded by $\varphi^+ \ (\varphi^-)$ and containing $\tilde c$.

 We proceed in the following steps.

 \textbf{Step 1: $H_{\eps}(\varphi)$ belongs to $A:=D^+ \setminus D^-$ for any small $\eps$ (including $\eps=0$).}

 Since $\tilde U^- \subset H_{\eps}(U)\subset \tilde U^+$ (see Proposition \ref{prop-choice}), we have $\partial H_{\eps}(U) \subset \tilde U^+\setminus \tilde U^-$. In particular, $H_{\eps}(\varphi) \subset \tilde U^+\setminus \tilde U^-$.  However it is not clear why this curve belongs to $A$ and not to  another connected component of $\tilde U^+\setminus \tilde U^-$.

 \textbf{Step 1.1: $H_{\eps}(\varphi)$ belongs to $D^+$}

Take a small neighborhood $\Omega$ of $\tilde c $ that belongs to $\tilde U^-$. Take $\eps$ so small that $H_{\eps}(c) \subset \Omega $; this is possible because $H_{\eps}(c)$ is close to $H_0(c)=\tilde c$, see  Definition  \ref{def-moderate-local} of moderate equivalence. Let $D_{\eps}$ be the disc bounded by $H_{\eps}(c)$ that contains $\varphi^-$ and $\varphi^+$; this disc is close to $D$, and $D \bigtriangleup D_{\eps} \subset \Omega$.

Let $U_c$ be the connected component of $\tilde U^+$ that contains $\tilde c$; then $U_c \subset D^+$.  The annulus $H_{\eps}(A(c, \varphi))\subset H_{\eps}(U)$ is connected,  belongs to $ \tilde U^+ $ and contains $H_{\eps}(c)\subset \Omega$, thus has a non-empty intersection with $U_c$.  Hence it belongs to $U_c$, therefore to $D^+$. Finally, $H_{\eps}(\varphi)\subset D^+ $.

\textbf{Step 1.2: $H_{\eps}(\varphi)$ belongs to $D_{\eps}$}

  If $c$ is a singular point, this is clear because $S^2 \setminus D_{\eps}$ is one point $H_{\eps}(c)$. Let $c$ be a limit cycle or a monodromic polycycle. Consider the time orientation on it. 
 Without loss of generality, we may assume that $\varphi$ is to the left with respect to this orientation of  $c$. The maps $H_{\eps}, H_0$ induce an orientation on  $H_{\eps}(c), H_0(c)=\tilde c$. With this orientation, $H_0(\varphi)$ is to the left with respect to $\tilde c$, thus $\varphi^{\pm}$ are to the left of $\tilde c$. 

  The curve $H_{\eps}(c)$ is close to $\tilde c$ by Definition  \ref{def-moderate-local} of moderate equivalence.
  Hence, the curves $\varphi^{\pm}$ lie to the left of both $H_{\eps}(c)$ and $\tilde c$; therefore, the disc $D_{\eps}$ is to the left of $H_{\eps}(c)$.

  On the other hand, as  $H_{\eps}$ preserves the orientation, and $\varphi$ is to the left of $c$, we conclude that  $H_{\eps}(\varphi)$ is to the left of  $H_{\eps}(c)$. Finally,  $H_{\eps}(\varphi)\subset D_{\eps}$, q.e.d.

\textbf{Step 1.3: $H_{\eps}(\varphi)$ does not intersect $D^-\cap D_{\eps}$ (thus does not intersect $D^-$). }
Indeed, $D \cap D^- =A(c,\varphi^-)\subset \tilde U^-$ and $D \bigtriangleup D_{\eps}\subset \Omega\subset \tilde U^-$. So $D^-\cap D_{\eps}\subset \tilde  U^-$, and the curve  $H_{\eps}(\varphi)\subset \tilde U^+ \setminus \tilde U^-$ cannot intersect this set.

We conclude that $H_{\eps}(\varphi) \subset D^+ \setminus D^-=A$.

\textbf{Step 2: The curve $H_{\eps}(\varphi)$ is non-contractive in the annulus $A$}

Indeed, this annulus is between two transversal loops of $\tilde c$, thus is saturated by trajectories of $w_0$ and by trajectories of a close vector field $w_{h(\eps)}$. Since $H_{\eps}$ preserves topological transversality, $H_{\eps}(\varphi)$ is topologically transversal to $w_{h(\eps)}$, thus is non-contractive in $A$.

\textbf{Step 3: The curves $H_0(\varphi)$ and  $H_{\eps}(\varphi)$ are homotopic in $A$}

Recall that we orient $\varphi$ so that $U$ is to the left with respect to it.
Both curves $H_0(\varphi)$ and  $H_{\eps}(\varphi)$  are non-contractive in $A$ and oriented so that $H_0(U), H_{\eps}(U)$ are to the left with respect to them, i.e. $\varphi^-$ is to the left of them. So they are oriented in the same way, thus homotopic in $A$ as oriented curves.

\textbf{Step 4: End of the proof}

Recall that each singular point and each limit cycle of $w_{h(\eps)}$ either belongs to $\tilde U^- \subset H_{\eps}(U)$, or is close to a hyperbolic singular point or a cycle of $w_0$ (thus does not intersect $\tilde U^+\supset H_{\eps}(U)$), see Proposition \ref{prop-inU-or-hyp}. Also, $LBS(W)\subset \tilde U^-$. In the previous step, we have proved that $H_0(\varphi)$ and $H_{\eps}(\varphi)$ are homotopic in $A \subset  \tilde U^+ \setminus \tilde U^-$. So they are  homotopic in a larger domain $S^2 \setminus (LBS(W) \cup \Sing w_{h(\eps)} \cup \Per w_{h(\eps)})$, q.e.d.


\end{proof}

Thus all the auxiliary lemmas are proved. Together with them, the Main Theorem is proved too.

\section*{Acknowledgements}
The authors were supported in part by the grant RFBR 16-01-00748-a and Laboratory Poncelet.

The authors are grateful to Dmitry Filimonov who read the paper and made fruitful comments, and to Ilya Schurov who drew Fig. \ref{fig:onecompu}, \ref{fig:onecompcu}, \ref{fig:CanonReg4}, \ref{fig-CR-inters-1}, \ref{fig-canonreg-of-p}, \ref{fig-canonreg-of-p-eps} (based on  handwritten sketches).

\section*{ Note of the second author}
The first part of this paper (Sections \ref{sec-intro} -- \ref{sec-proof}) was prepared by the two authors together. The most difficult second part was designed and written by the first author; the second made only editorial changes in this part, when the first draft was already written.


\begin{thebibliography}{99}

\bibitem{ALGM} A. A. Andronov, E. A. Leontovich, I. I. Gordon, and A. G. Mayer, Qualitative Theory of Second-Order Dynamical Systems (Nauka, Moscow, 1966).

\bibitem{AAIS}   V. I. Arnold, V. S. Afrajmovich, Y. S. Ilyashenko, and L. P. Shilnikov.
Dynamical Systems V. Bifurcation Theory and Catastrophe Theory. Encyclopaedia of Mathematical Sciences. Berlin: Springer-Verlag, viii+271 pp. doi:10.1007/978-3-642-57884-7. url:http://www.springer.com/gp/book/9783540181736. Trans. of V. I. Arnold, V. S. Afrajmovich, Yu. S. Ilyashenko, and L. P. Shilnikov. Teorija bifurcacij. Russian. Dinamicheskije sistemy V. 1986

\bibitem{Dum} F.Dumortier, ''Singularities of vector fields on the plane``, Journal of differential equations 23 (1977),  pp. 53-106.

\bibitem{DumLlibArt} Dumortier F., Llibre J., Artes J. Qualitative theory of planar differential systems. Universitext. Springer-Verlag, Berlin, 2006.

\bibitem{E} J. Ecalle, Introduction aux fonctions analysables et preuve constructive de la conjecture de Dulac,  Actualites  Mathematiques.  [Current  Mathematical  Topics],  Hermann,  Paris,  1992.

\bibitem{Fed} R. M. Fedorov, Upper bounds for the number of orbital topological types of planar polynomial vector fields modulo limit cycles, Proceedings of the Steklov Institute of Mathematics 2006, Volume 254, Issue 1, pp 238-254.

\bibitem{YuINS} N. Goncharuk, Yu. Ilyashenko, N.Solodovnikov, Global bifurcations in generic one-parameter families with a parabolic cycle on $S^2$, submitted

\bibitem{IKS} Yu.Ilyashenko, Yu. Kudriashov, I. Schurov, An open set of structurally unstable families of vector fields in the two-sphere, accepted to Inventiones Matematica

\bibitem{I91} Yu. S. Il'yashenko, Finiteness theorems for limit cycles, Translations of Mathematical Monographs, 94, Amer. Math. Soc., Providence, RI, 1991.

\bibitem{I} Yu. S. Ilyashenko. "Towards the general theory of global planar bifurcations". In:
Mathematical  Sciences  with  Multidisciplinary  Applications.  In  Honor  of  Professor Christiane Rousseau. And In Recognition of the Mathematics for Planet Earth Initiative. Ed. by B. Toni. Vol. 157. Springer Proceedings in Mathematics and Statistics. Springer, 2016, pp. 269-299.
isbn: 978-3-319-31321-4.
doi: 10.1007/978-3-319-31323-8.

\bibitem{Kleb} Kleban, Order of the topologically sufficient jet of a smooth vector field on the real plane at a singular point of finite multiplicity, in: Concerning  the  Hilbert  16th  problem. Vol. 165. American Mathematical Society Translation Series 2. Providence, RI: Amer. Math. Soc., 1995,  pp 131 -- 153

\bibitem{KS} A. Kotova and V. Stanzo. "On few-parameter generic families of vector fields on
the two-dimensional sphere", in: Concerning  the  Hilbert  16th  problem. Vol. 165.
American Mathematical Society Translation Series 2. Providence, RI: Amer. Math.
Soc., 1995, pp. 155-201.

\bibitem{Mark} L.Markus, Quadratic differential equations and non-associative algebras. Ann.Math.Stud.45:185-213, 1960.

\bibitem{MP}  Malta, I.P.; Palis, J. \emph{Families of vector fields with finite modulus of stability }, Lecture Notes in Math. 1981, v. 898, p. 212-229.

\bibitem{Neum} D.Neumann, Classification of continuous flows on 2-manifolds. Proc.Am.Math.Soc. 48:73-81, 1975,

\bibitem{Peix} M.M.Peixoto. On the lassification of flows on 2-manifolds. Academic, New York, Pages 389-419. Dynamical systems (Proc.Sympos., Univ. Bahia, Salvador, 1971)

\bibitem{Rous85} R. Roussarie. "Weak and continuous equivalences for families on line diffeomor-
phisms". English. In: Dynamical systems and bifurcation theory, Proc. Meet.
(Rio de Janeiro/Braz. 1985). Pitman Res. Notes Math. Ser. 160. 1987, pp. 377-385.

\bibitem{Rous} R. Roussarie, Bifurcations of Planar Vector Fields and Hilbert's Sixteenth Problem, Birkhauser, 1998

\bibitem{Graph} R. L. Wilder, W. L. Ayres, Lectures in Topology; The University of Michigan Conference of 1940. Ann Arbor: The University of Michigan press, 1941.

\end{thebibliography}
\end{document}